\documentclass[11pt]{article}
\usepackage{ae,lmodern} 
\usepackage[latin1]{inputenc}
\usepackage[T1]{fontenc}
\usepackage{amsmath} 
 \usepackage{amsthm}
 \usepackage{hyperref}
\hypersetup{
	colorlinks   = true, 
	urlcolor     = blue, 
	linkcolor    = blue, 
	citecolor   = red 
}
\usepackage{mathtools}
\usepackage{amscd}
\usepackage{amssymb}
\usepackage{amsfonts}
\usepackage{cancel}
\usepackage[
final
]{showlabels}
\usepackage{geometry}
\usepackage{cite}
\usepackage{enumitem}
\usepackage{soul} 
\usepackage[hang,flushmargin]{footmisc}

\usepackage[lofdepth,lotdepth]{subfig}
\usepackage[percent]{overpic}
\usepackage{grffile}%
\usepackage{tikz,pgfplots}
\pgfplotsset{compat=1.16}
\setenumerate[2]{label=(\alph*)}
\setenumerate[1]{label=\textup{(\roman*)}}
\setenumerate[3]{label=(\arabic*)}

\newcommand\inner[2]{{\langle #1, #2 \rangle}}
\newcommand{\normW}[1]{{\vert\kern-0.25ex\vert\kern-0.25ex\vert #1 
		\vert\kern-0.25ex\vert\kern-0.25ex\vert}}
\newcommand{\normWbig}[1]{{\Big\vert\kern-0.25ex\Big\vert\kern-0.25ex\Big\vert #1 
		\Big\vert\kern-0.25ex\Big\vert\kern-0.25ex\Big\vert}}

\newcommand{\B}{\mathcal{B}}
\newcommand{\D}{\mathcal{D}}
\newcommand{\C}{\mathbb{C}}
\newcommand{\N}{\mathbb{N}}

\newcommand{\R}{\mathbb{R}}

\newcommand{\Z}{\mathbb{Z}}

\newcommand{\boB}{\mathcal{B}}
\newcommand{\boC}{\mathcal{C}}
\newcommand{\boD}{\mathcal{D}}
\newcommand{\boE}{\mathcal{E}}
\newcommand{\boF}{\mathcal{F}}
\newcommand{\boG}{\mathcal{G}}
\newcommand{\boH}{\mathcal{H}}
\newcommand{\boI}{\mathcal{I}}
\newcommand{\boJ}{\mathcal{J}}

\newcommand{\boN}{\mathcal{N}}

\newcommand{\boV}{\mathcal{V}}
\newcommand{\boW}{\mathcal{W}}

\newcommand{\W}{\mathcal{W}}
\newcommand{\boX}{\mathcal{X}}

\newcommand{\ga}{\mathfrak{a}}
\newcommand{\gb}{\mathfrak{b}}
\newcommand{\gc}{\mathfrak{c}}

\newcommand{\gp}{\mathfrak{p}}
\newcommand{\gq}{\mathfrak{q}}
\newcommand{\q}{\mathfrak{q}}

\newcommand{\gs}{\mathfrak{s}}

\newcommand{\ka}{\kappa}
\newcommand{\ve}{\varepsilon}
\renewcommand{\d}{\textup{d}}
\newcommand{\p}{\textup{p}}

\newcommand{\wh}{\widehat}
\newcommand{\ptl}{{\partial}}
\newcommand{\Emin}{E_\textup{min}}

\newcommand{\loc}{\operatorname{loc}}

\newcommand{\card}{\operatorname{card}}
\providecommand{\abs}[1]{|#1 |}
\providecommand{\norm}[1]{\lVert#1 \rVert}

\renewcommand{\Re}{\operatorname{Re}}
\renewcommand{\Im}{\operatorname{Im}}
\newcommand{\sech}{\operatorname{sech}}
\newcommand{\atanh}{\operatorname{atanh}}
\newcommand{\atan}{\operatorname{atan}}

\theoremstyle{plain}
\newtheorem{theorem}{Theorem}[section]
\newtheorem{proposition}[theorem]{Proposition}
\newtheorem{lemma}[theorem]{Lemma}
\newtheorem{corollary}[theorem]{Corollary}

\theoremstyle{definition}
\newtheorem{step}{Step}

\newtheorem{remark}[theorem]{Remark}

\theoremstyle{remark}

\newtheorem*{merci}{Acknowledgments}

\numberwithin{equation}{section}

\begin{document}
\title{Exotic traveling waves for  a quasilinear Schr\"odinger equation with nonzero background }                               
\author{ Andr{\'e} de Laire\thanks{\noindent Univ.\ Lille, CNRS, Inria, UMR 8524 - Laboratoire Paul Painlev\'e, F-59000 Lille, France. \\ 	Emails: {\tt andre.de-laire@univ-lille.fr, erwan.lequiniou@univ-lille.fr} } \and
 {Erwan Le Quiniou\footnotemark[1]}}   
 
		\date{}
	\maketitle
\begin{abstract}
We study a defocusing quasilinear Schr\"odinger equation with nonzero conditions at infinity in dimension one. This quasilinear model corresponds to a weakly nonlocal approximation of the nonlocal Gross--Pitaevskii equation, and can also be derived by considering the effects of surface tension in superfluids. When the quasilinear term is neglected, the resulting equation is the classical Gross--Pitaevskii equation, which possesses a well-known stable branch of subsonic traveling waves solution, given by dark solitons. 

Our goal is to investigate how the quasilinear term and the intensity-dependent dispersion affect the traveling-waves solutions.
We provide a complete classification of finite energy traveling waves of the equation, 
in terms of two parameters: the speed and the strength of the quasilinear term.
This classification leads to the existence of dark and antidark solitons, as well as more exotic localized solutions like dark cuspons, compactons, and composite waves, even for supersonic speeds. Depending on the parameters, these types of solutions can coexist, showing that finite energy solutions are not unique.  Furthermore, we prove that some of these dark solitons can be obtained as minimizers of the energy, at fixed momentum, and that they are orbitally stable.
\end{abstract}   
	
	\maketitle
	
	\medskip
	\noindent{{\em Keywords:}
			Quasilinear Schr\"odinger equation, Gross--Pitaevskii equation, traveling waves, dark solitons, 		dark cuspons, nonzero conditions at infinity, orbital stability.
		
		\medskip
		\noindent{2020 {\em MSC}}:
				35Q55, 
                35J62,  
				35C07, 
                35C08, 
                34A05,      
                35J20, 
			    35B35,  
                35D35, 
		35Q60, 
  		82D50, 
		%
        34A12.   	
  \tableofcontents

\section{Introduction}
\label{intro}
\subsection{The quasilinear equation and related models}

We consider the following defocusing quasilinear Gross--Pitaevskii equation in dimension one
		\begin{equation}
			\label{QGP}
   \tag{QGP}
			i\ptl_t\Psi=\ptl_{xx}\Psi+\Psi(1-|\Psi|^2)+\ka\Psi\ptl_{xx}(1-|\Psi|^2),\quad\text{ in }\R\times\R,
		\end{equation} 
where $\ka$ is a real parameter, and $\Psi:\R\times\R\to\C$, satisfies the nonzero conditions at infinity
		\begin{equation}
			\label{nonzero0}
			\lim_{\abs{x}\to \infty}\abs{ \Psi(x,\cdot)}=1,
		\end{equation}
		representing a (normalized) nonzero background.

Let us notice that the equation is recast as in \eqref{QGP} for practical purposes. Indeed, we can write a more general Schr\"odinger equation as
 \begin{equation}\label{quasilin0}
			i\ptl_t\Phi=\ptl_{xx}\Phi+\gs \Phi\big( |\Phi|^2+\ka \ptl_{xx}|\Phi|^2)).
		\end{equation} 
 When $\kappa=0,$ it corresponds to the cubic NLS equation, a classical model for Bose--Einstein condensates, superfluidity, and nonlinear optical fibers, depending on the sign of $\gs$, and on the background conditions \cite{fibich,KevFrCa0}.
 For instance, in  Bose--Einstein condensates, the term in $\gs$  models attractive interatomic interactions if $\gs>0$,
 and repulsive interactions if $\gs<0$. In nonlinear optics, it corresponds to the Kerr effect in a focusing fiber if $\gs>0$, and to ionization effects in a defocusing one if $\gs<0$. Equation \eqref{quasilin0} enables the description of significant and experimentally relevant nonlinear phenomena such as {\em solitons}. Solitons are particular types of solutions that travel with constant speed and with a profile that remains unchanged. They provide important information for the analysis of dispersive equations.  Although the most common solitons are bright and dark solitons, there are more exotic solitons dispersive PDEs, that are not smooth, such as the cupsons and compactons \cite{Roseneau,OstroPeliShira,ARORAsolitons,OSTERfinitelaticeNonlocal,GermainCompactons,HarropMarzuola,ZhangZhou}, which we will discuss below. 

Bright solitons are characterized by having a localized amplitude peak, rapidly decaying to zero. The existence and properties of these solutions are a classical subject for the focusing NLS equation, that has received much attention \cite{SuleSul0,CuccaPeli,Weinstein} in the case $\kappa=0$. In the case $\kappa\geq 0$, 
Kr\'owlikowski and Bang  obtained in \cite{krolikowski2000} an explicit formula for the bright solitons to \eqref{quasilin0} with $\gs=1$, taking the form of a standing wave $\Phi(x,t)=v_{\omega,\kappa}(x)e^{-i\omega t},$ for every $\omega>0$,  
with $v_{\omega,\kappa}$ a real-valued profile given by  $v_{\omega,\kappa}(x)=F_{\omega,\kappa}^{-1}(|x|)$, for all $x\in\R$, where 
\begin{equation}\label{eq:implibright}
    F_{\omega,\kappa}(y)=\frac{1}{\sqrt{\omega}}\atanh\Big(\frac{1}{\sqrt{2\omega}}\sqrt{\frac{2\omega-y^2}{1+2\ka y^2}}\Big)+2\sqrt{\ka}\atan\Big( \sqrt{2\ka} \sqrt{\frac{2\omega-y^2}{1+2\ka y^2}}\Big).
\end{equation}
These bright solitons are unique for up to invariances (translation by a constant and multiplication by a phase shift).
Letting $\ka\to0$ in \eqref{eq:implibright}, we recover the profile of the cubic NLS bright soliton $$v_{\omega,0}(x)=\sqrt{2\omega}\sech(\sqrt{\omega}x).$$ 

On the one hand,  Colin, Jeanjean, and Squassina~\cite{Colin2}  showed the existence of bright solitons for \eqref{quasilin0}, with  $\gs=1$ and $\ka>0$, in any dimension. Moreover, using formula \eqref{eq:implibright} in the one-dimensional case, the results in \cite{Colin2} imply that these bright solitons are orbitally stable solutions to \eqref{quasilin0}. We also refer to Iliev and Kirchev \cite{ilievquasilinstab} who showed existence of bright solitons and a stability criterion for equations similar to \eqref{quasilin0} with $\gs=1$ in one dimension, with more general nonlinearities.

 On the other hand, in the defocusing case, there are no bright solitons, but the existence of dark solitons is expected \cite{kivshar}. Dark solitons have a localized amplitude dip or notch (of their absolute value) on a nonzero background density. Although these solutions are physically relevant, they have been less studied in the literature. They can be obtained explicitly for the cubic defocusing NLS equation (i.e.\ \eqref{quasilin0} with $\gs=-1$ and $\kappa=0$), and they are given, up to invariances, by
$$\Phi(x,t)=u_c(x-ct)e^{it},$$
 for $c\in [0,\sqrt 2),$ where 
 \begin{equation}
	\label{sol:1D}
	u_{c}(x)=\sqrt{\frac{2-c^2}{2}}\tanh\Big(\frac{\sqrt{2-c^2}}{2}x\Big)-i\frac{c}{\sqrt{2}}.
\end{equation} Notice that this branch of dark solitons satisfies the nonzero background condition $\abs{\Phi(x,\cdot)}\to 1$, as $\abs{x}\to\infty$. We refer to \cite{bethuel2008existence,bethuelasympt} for more details and stability results for these dark solitons, and to \cite{chironexistence1d} for some generalizations. 

Since dark solitons can only exist in the defocusing case, from now on, we only consider the case $\gs=-1$. To avoid the dependence on $t$ of the solitary waves, it is usual to perform the change of variables $$\Psi(x,t)=\Phi(x,t)e^{-it}, $$
which transforms equation \eqref{quasilin0} into \eqref{QGP}, and so that the nontrivial condition at infinity appears more explicitly in the equation.

The study of generalizations of Schr\"odinger equations with intensity-dependent dispersion, has recently gained significant attention, 
mostly in the context of solutions vanishing at infinity \cite{PelinovskyClassif1,kevrekidis2024stabilitysmoothsolitarywaves,PelinovskyVar}.
We refer to the work of Pelinovsky and Plum  \cite{PelinovskyGPBlackInIntensityDependantEq}  for the study of the stability of the black soliton for an intensity-dependent dispersion Schr\"odinger equation.

Given the increasing interest in intensity-dependent dispersion equations and quasilinear Gross-Pitaevskii equations, our aim is to provide a complete classification of finite energy traveling waves of the equation, and the stability of some of them.

We end this subsection by giving some physical motivations for the quasilinear model \eqref{QGP}, with the nonzero condition \eqref{nonzero0}. The evolution of a one-dimensional optical beam of intensity $|{\Psi}|^2$ in a defocusing nonlocal Kerr-like medium is given by the  nonlocal Gross--Pitaevskii equation
\begin{equation}\label{eq:nonloc}
i\partial_t\Psi=\partial_{xx}\Psi+\Psi(\boW*(1-|\Psi|^2)),
\end{equation}
 where $\W$ characterizes the nonlocal response of the medium \cite{nikolov2004,krolikowski2000,delaire,de2011nonexistence}. As explained in \cite{krolikowski2000,horikisWnonloc}, in a  weakly nonlocal medium we can replace $\boW$ by   $\boW_\ve(\cdot)=\boW(\cdot/\ve)/\ve$, for a  small positive $\ve$. Then,
performing a Taylor expansion of $\eta(x-y)=(1-|\Psi|^2)(x-y,t)$ near $x$, for fixed $t$, we have 
$$
\eta(x-y)=\eta(x)-y\eta'(x)+\frac{y^2}{2}\eta''(x)+{O}(y^3),
$$
hence, using that $\boW_\ve$ is even, the convolution product in \eqref{eq:nonloc} can be formally computed as
\begin{equation}\label{eq:wnonloc}
	(\boW_\ve*\eta)(x)=\int_\R \boW_\ve(y)\eta(x-y)dy=\eta(x)+\ka_\ve\eta''(x)+O(\ve^3),
 \text{ with } \ka_\ve=\frac{\ve^2}2\int_\R y^2\boW(y).
\end{equation}
Therefore,  equation \eqref{QGP} follows from \eqref{eq:nonloc} and \eqref{eq:wnonloc}, in the regime $\ve$ small, neglecting the term $O(\ve^3)$.
For this reason, \eqref{QGP} is known as the weakly nonlocal Gross--Pitaevskii equation in the nonlinear optics literature. 

Formally, we can also see a connection between \eqref{eq:nonloc} and \eqref{QGP}
by considering the potential given with its Fourier transform
\begin{equation}
\label{W:kappa}
    \wh \boW_\ka(\xi)=1-\kappa\xi^2,
\end{equation}
that can be seen as a limit case for the results of the existence of dark solitons for \eqref{eq:nonloc} (see Remark~\ref{rem:nonlocal}).

In addition, \eqref{QGP} was also obtained using the least action principle by 
Kurihara in \cite{kurihara}, in order to describe  a superfluid $^4\rm{He}$ film, where $\ka\leq 0$ and  
$|\ka|$ measures the surface tension of the superfluid. 
 
\subsection{Classification of finite energy traveling waves}
Equation  \eqref{QGP} has a Hamiltonian structure, and its energy,  given by
		\begin{equation}\label{eq:energie}
			E_\ka(\Psi(\cdot,t))= \frac{1}{2} \int_\R \abs{\ptl_x\Psi(x,t)}^2 dx +\frac{1}{4}\int_\R
			\Big(1-\abs{\Psi(x,t)}^2\Big)^2dx-\frac{\ka}{4}\int_\R\big(\ptl_x \abs{\Psi(x,t)} ^2 \big)^2  dx,
		\end{equation}
		is formally conserved, as well as the (renormalized) momentum 
		\begin{equation}\label{def:moment}
			p(\Psi(\cdot,t))=\frac{1}{2}\int_\R\inner{i\ptl_x\Psi(x,t)}{\Psi(x,t)}\Big(1-\frac{1}{\abs{\Psi(x,t)}^2}\Big)dx,
		\end{equation}
		  whenever $\inf_{x\in\R}\abs{\Psi(x,t)}>0$, where we used the inner product $\langle z_1,z_2 \rangle=\Re(z_1)\Re(z_2)+
		\Im(z_1)\Im(z_2)$, for $z_1$, $z_2\in \C$. 

	We are interested in solutions to \eqref{QGP} of the form
		$$\Psi_c(x,t)=u(x-ct),$$
		which represents a traveling wave with profile $u:\R\to\C$ propagating at speed $c\in\R$. Hence, the profile $u$ satisfies
		\begin{equation}
			\label{TWc}\tag{TW$(c,\ka)$}
			icu'+u''+u (1-\abs{u}^2)+\ka u\big(1-        \abs{u}^2\big)''=0.
		\end{equation}	
		Notice that taking the complex conjugate of $u$ in equation \eqref{TWc}, we are reduced to the case $c\geq 0$.		
To study physically relevant solutions to \eqref{TWc} in function of $\ka$, we define 
 energy space
		$$\boX(\mathbb{R})=
		\{v \in H^{1}_{\loc}(\mathbb{R}) : 1-|v|^{2}\in L^2(\mathbb{R}), \ v' \in L^{2}(\mathbb{R})\}.$$
Let us recall that 	 $\mathcal{X}(\mathbb{R})\subset L^\infty(\R)\cap \boC(\R)$, and that any function in $\boX(\R)$ satisfies the nontrivial condition at infinity  \eqref{nonzero0}
(see Lemma~\ref{lem:finiteenergyassum}).   
We will use extensively, as a new variable, the {\em intensity profile} of $u$, given by  $$\eta_{u}=1-\abs{u}^2,$$ so that the condition at infinity \eqref{nonzero0} becomes $\eta_{u}(x)\to0$, as $\abs{x}\to\infty$. Omitting the subscript of $\eta_u$ for notational simplicity, we can recast the energy functional as
\begin{equation}\label{eq:Eu}
E_\ka(u)=\frac{1}{2}\int_\R \abs{u'}^2+\frac{1}{4}\int_\R\eta^2-\frac{\ka}{4}(\eta')^2,\quad\text{ for all }u\in\boX(\R).
\end{equation}
Moreover, the momentum is well-defined in the nonvanishing energy space defined by
	\begin{equation}
	    \boN\boX(\mathbb{R})=
		\{v \in \boX(\mathbb{R})  :  \inf_{\R}\abs{v}>0\},
	\end{equation}
and writing the lifting $u=\sqrt{1-\eta}e^{i\theta}\in\boN\boX(\R)$ (see Lemma~\ref{lem:phaseregu}), we have
\begin{equation}
\label{def:mom2}
    p(u)=-\frac{1}{2}\int\langle iu',u\rangle\frac{\eta}{1-\eta}=\frac{1}{2}\int_\R\eta\theta'.
\end{equation}
Furthermore, the energy can be written in this case as 
\begin{equation}\label{eq:Eupol}
E_\ka(u)=\frac{1}{8}\int_\R\frac{(\eta')^2(1-2\ka+2\ka\eta)}{1-\eta}+\frac{1}{2}\int_\R(1-\eta)(\theta')^2+\frac{1}{4}\int_\R\eta^2.
\end{equation}
We can check that the energy space $\boX(\R)$ is equal to the domain of $E_\kappa$ if $\kappa\leq 0$, and is strictly included in the domain of $E_\kappa$ if $\kappa>0$. 
  
To be more precise, we say that    $\Psi\in L^1_\mathrm{loc}(\R;H^1_{\loc}(\R))$ is a \textit{weak solution} to  \eqref{QGP} if, for all $\varphi\in\boC_0^\infty(\R;H^1(\R))$,
	\begin{equation}\label{eq:weakqgp}
		\iint_{\R\times\R}
		\langle {i\Psi},{\ptl_t\varphi}\rangle dxdt=\iint_{\R\times\R}\big(
		\langle\partial_{x}\Psi, {\partial_x\varphi}\rangle
		-(1-\abs\Psi^2)
		\langle \Psi,\varphi \rangle+
		\ka  \partial_{x}(1-\abs\Psi^2)\partial_x  \langle\Psi,\varphi \rangle\big)
		dxdt.
	\end{equation}
Hence, we can check that $\Psi(x,t)=u(x-ct)$, with $u\in \boX(\R)$  is a weak solution to \eqref{QGP}
if and only if for all $\phi\in \boC_0^\infty(\R;\C)$,
		\begin{equation}
			\label{TW:weak}
			\int_\R  \langle  icu'+u (1-\abs{u}^2),\phi\rangle   
			- \langle u', \phi'\rangle 
			+2\ka  \langle u,u'\rangle
			\langle  u,\phi\rangle '=0,
		\end{equation}
where we used that 	$(1-\abs{u}^2)'=-2\langle u,u'\rangle$. Therefore, \eqref{TW:weak} corresponds to the weak formulation of \eqref{TWc}.

To obtain analytical solutions, we will show that if $u$ is a finite energy solution to \eqref{TW:weak}, then its intensity profile $\eta=1-\abs{u}^2$ satisfies the ODE
\begin{equation}
\label{EDO:intro}
    (1-2\ka+2\ka\eta)(\eta')^2=\eta^2(2-c^2 -2\eta),
\end{equation}
as long as $u$ is smooth enough. Thus, the rest of the analysis relies on the study of possible singularities, combined with the computation of the primitives of
		\begin{align*}
			f(y)=-y\sqrt{\frac{2-c^2-2y}{1-2\ka+2\ka y}},
		\end{align*} 
according to the parameters $c$ and $\kappa$. As explained in Section~\ref{sec:construction:smooth}, 
taking $x$ large enough in \eqref{EDO:intro}, we expect that 
\begin{equation}
\label{cond:c-ka}
    0<(2-c^2)(1-2\kappa),
\end{equation}
is a necessary condition for the existence of nontrivial solutions. From \eqref{cond:c-ka}, we see that $c=\sqrt 2$ and $\kappa=1/2$ are critical values, as shown below in our results.
In fact, the value $c=\sqrt 2$ corresponds to the speed of sound for \eqref{quasilin0} (see the explanations in \cite{de2011nonexistence,delaire-salva-guillaume} applied to \eqref{W:kappa}), while the value $\kappa=1/2$ corresponds to the critical case, for which the linear dispersion and the nonlinear disperse
on (i.e.\ the quasilinear terms of order 2) cancel as the intensity $|\Psi|^2$ approaches $1$ \cite{ARORAsolitons,OSTERfinitelaticeNonlocal,lenells2005traveling}.

To explain our results, we split the set of parameters according to the critical values, as represented in Figure~\ref{fig:domains}, by defining the following regions
	 \begin{align}\label{def:D1}
\D_1&=\{(c,\ka) : 0\leq c<\sqrt{2} \text{ and }  0<\ka<1/2  \},\ 
	 	\D_2=\{(c,\ka) : 0\leq c<\sqrt{2} \text{ and }  \ka\leq0  \},\\
	 	\D_3&=\{(c,\ka) : c>\sqrt{2} \text{ and }  \ka>1/2  \}, \ \text{ and, } \	\D=\D_1\cup \D_2\cup  \D_3;
   \label{def:D3}
		\end{align}
the boundary sets associated with the sonic speed
	 \begin{equation}
	 	\label{B-}
	 	\B_-=\{(c,\ka):c=\sqrt{2}\text{ and }0<\ka<1/2\},\quad 
\B_+=\{(c,\ka):c=\sqrt{2}\text{ and }\ka>1/2\},
	\end{equation}
 and the boundary set associated with the critical value for $\ka$
 \begin{equation}
     \label{def:C}
     \boC=\{(c,\ka):c\geq0\text{ and }\ka=1/2\}.
 \end{equation}
\begin{figure}[ht!]	
\centering
		\begin{overpic}
				[scale=0.75,trim=100 460 100 110,clip]{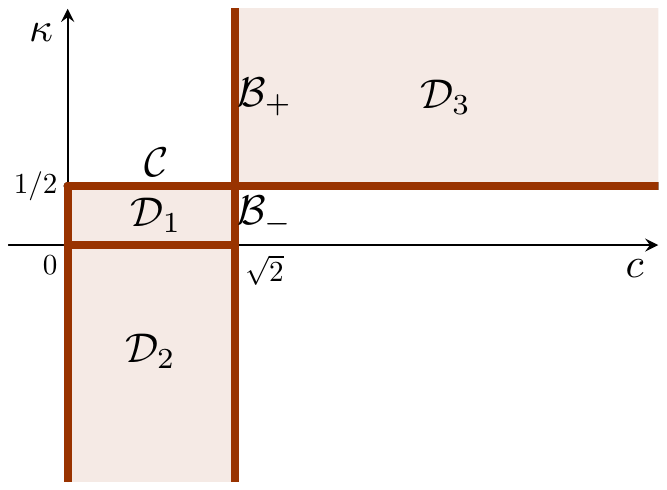}
			\end{overpic}
\caption{Sets of parameters to classify the solutions of \eqref{TWc}.}
\label{fig:domains}
\end{figure}
Our main classification result shows the existence of multiple branches of localized traveling waves indexed by $\ka$ and $c$, that can coexist at fixed parameters, showing the nonuniqueness of finite energy solutions. Before stating our results, we give some vague definitions used in the literature,  to understand the nature of these different types of solitons $u$ for \eqref{TWc}. The adjective {\em dark} refers to the nonzero constant background condition $\abs{u(x)}\to 1$, as $\abs{x}\to\infty$. Moreover, $u$ is a:
\begin{enumerate}[topsep=1pt,itemsep=-3pt,partopsep=1ex,parsep=1ex]
\label{def:solitontypes}
\item  \emph{dark soliton} if $u$ is smooth and $\abs{u}^2$ has a  localized dip.
\item \emph{antidark soliton} if $u$ is smooth and  $\abs{u}^2$ has a localized bump.
\item \emph{black soliton} if $u$ is smooth and $u$ vanishes at some point.
\item \emph{dark cuspon} if $u$  is continuous and has a localized cusped dip, $u'$ is unbounded at the cusp but is smooth elsewhere.
\item \emph{antidark cuspon} if $u$  is continuous and has a localized cusped bump,  $u'$ is unbounded at the bump but is smooth elsewhere.
\item \emph{compacton} if $u$  is continuous and $1-\abs{u}^2$  is compactly supported.
\item \emph{composite wave} if $u$ is continuous, 
built using pieces of the solitons mentioned above at different intervals, and the gluing is not $\boC^2$.
\end{enumerate}
With these formal definitions at hand, we can schematically summarize our results on the existence and uniqueness (up to invariances)
of nontrivial finite energy solution to \eqref{TWc}, in the following table. 
\begin{center}
 \begin{tabular}{||l|l||}
    \hline    
      In $\D_1$& There exist a unique \emph{dark soliton}, a unique \emph{antidark cuspon},\\
      &and \emph{composite waves}.\\
   \hline
   In $\D_2$& There exists a unique \emph{dark soliton}. \\
    \hline
      In  $\D_3$& There exist a unique \emph{antidark soliton}, a unique \emph{dark cuspon},\\
      &and \emph{composite waves}. \\
    \hline
       In $\B_-$& There exist a unique \emph{antidark cuspon}, and  \emph{composite waves}. \\    
        \hline
       In $\B_+$& There exist a unique \emph{dark cuspon},  and \emph{composite waves}. \\ 
           \hline
In $\boC$& There exist \emph{composite waves} built from  \emph{compactons}. \\
       \hline
       Elsewhere & There is no nontrivial finite energy traveling wave. \\
        \hline        
    \end{tabular}
\end{center}
To give explicit formulas for the dark and antidark soliton solutions to \eqref{TWc}, which we will denote by $u_{c,\ka}$, 
we introduce the following intervals
 \begin{align*}
 \boI_c=(0,1-c^2/2], \quad  \boI_c^\circ=(0,1-c^2/2) \quad \text{ if } c<\sqrt{2},\\
 \boJ_c=[1-c^2/2,0), \quad 	\boJ_c^\circ=(1-c^2/2,0)\quad\text{ if } c>\sqrt{2}.
 \end{align*}
Indeed, these intervals will correspond to the image set of $\eta_{c,\ka}:=1-\abs{u_{c,\ka}}^2$.
Then, we will check that the following functions are well-defined depending on the set of parameters \eqref{def:D1}--\eqref{def:D3} (see 	Lemmas~\ref{lem:etafunctions} and \ref{lem:etafunctions2}).
(i) For $(c,\ka)\in \D_1$, we set $F_{c,\ka}: \boI_c\to \R$ the function defined, for all $y\in \boI_c$, by
	 	\begin{align}\label{eq:FD1}
	 		F_{c,\ka}(y)=2\sqrt{\ka}\atan\Big(\sqrt{\ka}\sqrt{\frac{ 2-c^2-2y}{1-2 \ka +2\ka y}}\Big)+2\sqrt{\frac{1-2\ka}{2-c^2}}\atanh\Big(\sqrt{\frac{(1-2 \ka) (2-c^2-2y)}{(2-c^2)(1-2\ka+2\ka y)}}\Big).
	 	\end{align}
(ii)
	 	For $(c,\ka)\in \D_2$, we set $G_{c,\ka}: \boI_c\to \R$ the function defined, for all $y\in \boI_c$, by
	 	\begin{equation}\label{eq:FD2}
	 		G_{c,\ka}(y)=-2\sqrt{-\ka}\atanh\Big(\sqrt{-\ka}\sqrt{\frac{ 2-c^2-2y}{1-2 \ka +2\ka y}}\Big)
	 		+ 2\sqrt{\frac{1-2\ka}{2-c^2}}\atanh\Big(\sqrt{\frac{(1-2 \ka) (2-c^2-2y)}{(2-c^2)(1-2 \ka +2\ka y)}}\Big).
	 	\end{equation}
(iii)
	 	For $(c,\ka)\in \D_3$, we set $H_{c,\ka}: \boJ_c\to \R$ the function defined, for all $y\in \boJ_c$, by
	 	\begin{equation}\label{eq:FD3}
	 		H_{c,\ka}(y)=	F_{c,\ka}(y).
	 	\end{equation}
 We will prove that the functions \eqref{eq:FD1}--\eqref{eq:FD3} are injective and take their values in $(0,\infty)$.  Thus, we can define their inverse  as follows
	 \begin{equation}
	 	\label{def:inverses}
	 	\boF_{c,\ka}=F^{-1}_{c,\ka}, \text{ for }  (c,\ka)\in \D_1;  
	 	\
	 	\boG_{c,\ka}=G^{-1}_{c,\ka}, \text{ for }  (c,\ka)\in \D_2;\
	 	\boH_{c,\ka}=H^{-1}_{c,\ka}, \text{ for }  (c,\ka)\in \D_3. 
	 \end{equation}
 Also, we can extend $\boF_{c,\ka},$ $\boG_{c,\ka}$ and $\boH_{c,\ka}$ to even $\boC^\infty$-functions on $\R$, which we still denote by $\boF_{c,\ka},$ $\boG_{c,\ka}$ and $\boH_{c,\ka}$. 
 
 We can now state our first classification result for {\em smooth} solutions. 
 Using the sets defined in \eqref{def:D1}--\eqref{def:D3}, we first show that the only finite energy smooth solution to \eqref{TWc} in $\boD^c$ are the trivial ones, i.e.\ the constants of modulus one. Then, for $(c,\kappa)\in \boD$,
 the description of dark solitons is explicitly provided in terms of $\boF_{c,\ka}$ and $\boG_{c,\ka}$, while antidark solitons are described by $\boH_{c,\ka}$. Moreover, we establish their uniqueness among smooth solutions, up to invariances, i.e.\ up to a translation and a phase shift.
\begin{theorem}
\label{thm:classiftwregu}
\begin{enumerate}[leftmargin=4.5ex,topsep=1pt,itemsep=-3pt,partopsep=0ex,parsep=1ex]
\item\label{thm:item:regunoexi} If $(c,\ka)\notin\D$ and $u_{c,\ka}\in \boX(\R)\cap\boC^2(\R)$ is a solution to \eqref{TWc}, then  there exists $\phi\in\R$ such that $u_{c,\ka}(x)=e^{i\phi}$, for all $x\in\R$.
\item\label{thm:item:reguexi}  For $(c,\ka)\in\D$, we set
\begin{equation}
\label{def:eta:c}
\eta_{c,\ka}=\boF_{c,\ka},  \text{ if }  (c,\ka)\in \D_1;\  
\eta_{c,\ka}=\boG_{c,\ka}, \text{ if }  (c,\ka)\in \D_2;\ 
\eta_{c,\ka}=\boH_{c,\ka}, \text{ if }  (c,\ka)\in \D_3. 
\end{equation}	
If $c>0$, then $\eta_{c,\ka}<1$ on $\R$, and   $u_{c,\kappa}=\rho_{c,\kappa} e^{i\theta_{c,\kappa}}$ 
is the unique nontrivial  solution to \eqref{TWc} in $\boC^2(\R)\cap\boX(\R)$, up to invariances,  where 
\begin{equation}\label{formule:soli}
\rho_{c,\kappa}(x)=\sqrt{1-\eta_{c,\kappa}},\quad  \text{ and }\quad  \theta_{c,\kappa}(x)=\frac{c}2\int_0^x\frac{\eta_{c,\kappa}(y)}{1-\eta_{c,\kappa}(y)}dy.
\end{equation}
If $c=0$, then the real-valued odd function
	\begin{align}
 \label{black-soliton}
u_{0,\ka}(x)=\pm\sqrt{1-\eta_{0,\ka}(x)}, \quad \text{ for all } \pm x\geq0,
	\end{align}
 is the unique nontrivial  solution to \eqref{TWc} in $\boC^2(\R)\cap\boX(\R)$, up to invariances.
\end{enumerate}
In addition, for any $(c,\ka)\in \D$, $u_{c,\kappa}$ belongs to $\boC^\infty(\R)$,
and $\eta_{c,\ka}$ is even,  reaching a  global extremum at the origin with 
\begin{equation}
\label{maximum:value}
    \eta_{c,\kappa}(0)=1-c^2/2.
\end{equation}
 Moreover,  we have $\eta_{c,\kappa}>0$ and $\eta_{c,\kappa}'<0$ on $(0,\infty)$ if $(c,\ka)\in \D_1\cup \D_2$, while  $\eta_{c,\kappa}<0$ and $\eta_{c,\kappa}'>0$ on $(0,\infty)$ if $(c,\ka)\in \D_3$. Finally, there exist constants $A,C>0$ such that for all $k\in\N$, the following exponential decay holds
\begin{align}\label{eq:decaydark}
	\abs{D^k u_{c,\ka}'(x)}+\abs{D^k\eta_{c,\ka}(x)}\leq Ae^{-C\abs{x}},\quad\text{ for all } x\in\R.
\end{align}
\end{theorem}
Notice that \eqref{maximum:value} implies that the solutions $u_{c,\kappa}$ in Theorem~\ref{thm:classiftwregu} correspond to dark solitons that do not vanish if $c\neq 0$ (gray solitons) and to black solitons if $c=0$.
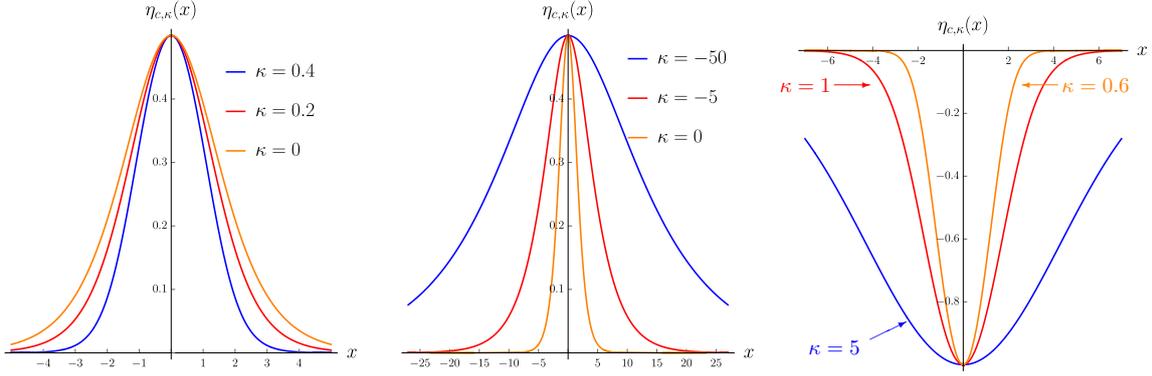
\begin{figure}[ht!]
		\begin{tabular}{ccc}
		\resizebox{0.3\textwidth}{!}{
			\begin{overpic}
				[scale=0.5,clip]{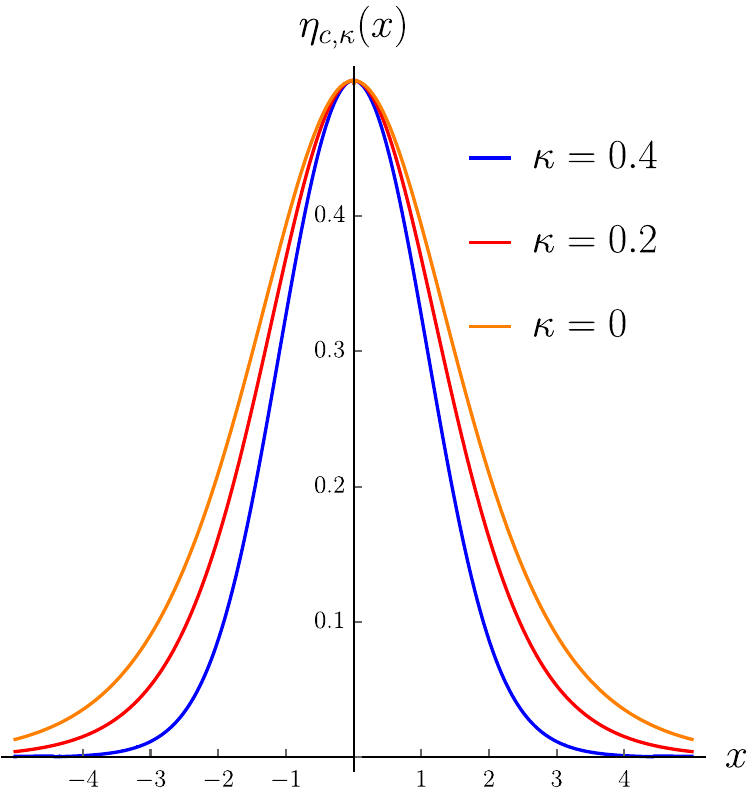}
			\end{overpic}
		}
		&\resizebox{0.3\textwidth}{!}{
			\begin{overpic}
				[scale=0.5,clip]{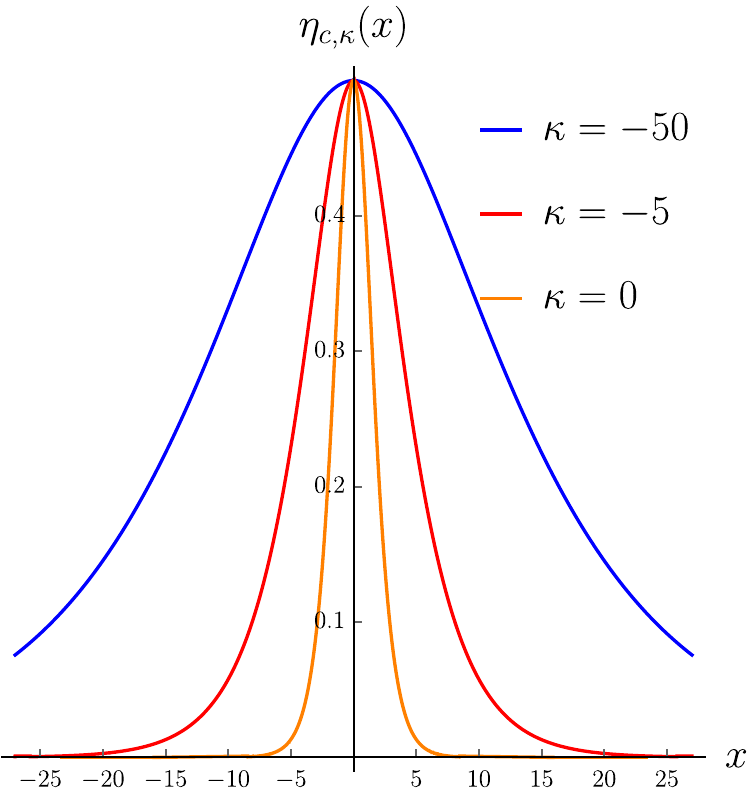}
			\end{overpic}	
		}
		& \resizebox{0.3\textwidth}{!}{
			\begin{overpic}
				[scale=0.5,clip]{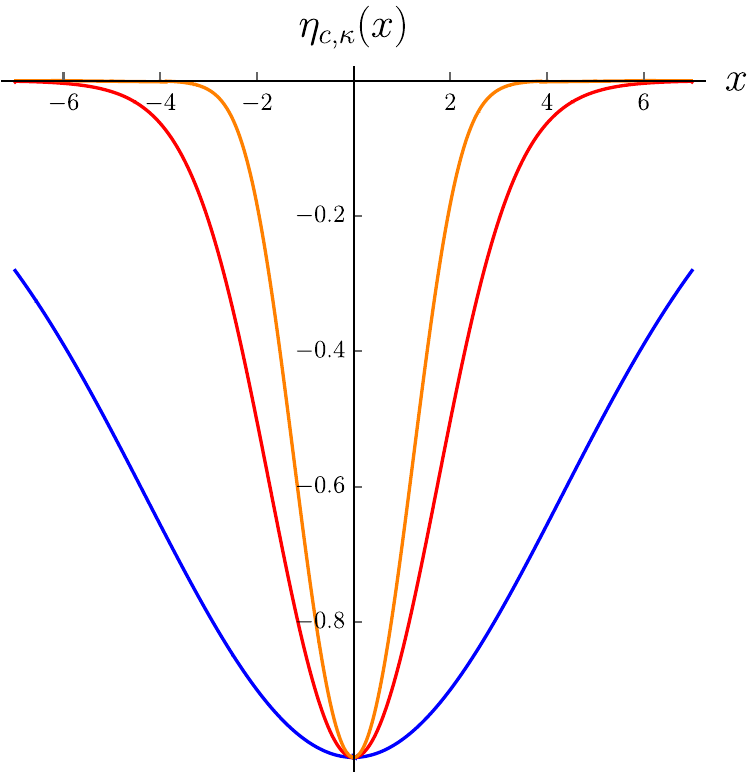}
				\put(10,80){\color{red}\vector(3,0){10}}
				\put(-5,78){\color{red}$\ka=1$}
				\put(72,80){\color{orange}\vector(-3,0){10}}
				\put(73,78){\color{orange}$\ka=0.6$}
				\put(20,9){\color{blue}\vector(2,1){10}}
				\put(3,5){\color{blue}$\ka=5$}
			
			\end{overpic}
		}
	\end{tabular}
\caption{Plot of intensity profiles $\eta_{c,k}$  in Theorem~\ref{thm:classiftwregu} for several values of $(c,\kappa)$. 
The left panel displays profiles for  $(1,0), (1,0.2), (1,0.4)\in\D_1$ associated with dark solitons. 
In the center, profiles for  $(1,-50), (1,-5), (1,0)\in\D_2$ associated also with dark solitons. 
On the right, profiles for  $(2,0.6), (2,1), (2,5)\in\D_2$, associated with antidark solitons. 
}
\label{fig:solitons}
\end{figure}

In Figure \ref{fig:solitons}, we depict the intensity profile $\eta_{c,k}$ given by 
 \eqref{def:eta:c} in Theorem~\ref{thm:classiftwregu}
for several values of $(c,\ka)\in\D$. In the left and center panels, we plotted $\eta_{c,\ka}$ for parameters 
in $\boD_1$ and $\boD_2$, respectively, associated with dark solitons; 
while plots for parameters in $\D_3$, associated with antidark solitons, are shown in the right panel.  The case $\kappa=0$ corresponds to the dark soliton \eqref{sol:1D}, so that the intensity profile is given by 
\[\eta_{c,0}(x)=1-\abs{u_{c,\kappa}}^2=\frac{(2-c^2)}{2}\sech^2\Big(\frac{\sqrt{2-c^2}}{2}x\Big),\quad \text{ for all }x\in\R.\]
From the left and center panel, we deduce that, at a fixed speed, the profile wavelength of the dark soliton narrows when $\kappa$ increases from $-\infty$ to $1/2$. We notice the same effect for the antidark solitons in the right panel when $\kappa$ decreases from  $ \infty$ to $1/2$.

In addition, we show  in Proposition~\ref{prop:reguparam1} smooth dependence of the dark soliton $u_{c,\ka}$,
 with respect to $(c,\ka)$ in $\D$. Consequently, the dark solitons $u_{c,\ka}$ in Theorem~\ref{thm:classiftwregu} converge to the dark soliton \eqref{sol:1D}, 
 as $\ka\to0$, in $H^{s}(\R)$, for all $s\in \N$.  

Notice that the uniqueness stated in Theorem~\ref{thm:classiftwregu} is in the set  $\boX(\R)\cap\boC^2(\R)$. This restriction is due to the existence of finite energy weak solutions that are not of class $\boC^2(\R)$ in the new set of parameters  
$$
	\tilde\D=\D_1\cup   \D_3\cup \B_-\cup\B_+,
$$
and also in $\boC$. 
We argue that these sets are well-suited for studying weak solutions. 
Indeed, the next result shows that there are no solutions outside $\tilde{\D}\,\cup \boC$.
\begin{theorem}\label{thm:non-singular-sol}
 Let $c\geq 0$ and $\kappa\in \R$. Assume that $u_{c,\ka}\in\boX(\R)$ is a solution to \eqref{TWc}.
If $(c,\ka)\notin (\tilde{\D}\,\cup \boC)$, then 
	either $(c,\ka)\in \D_2$ and $u_{c,\ka}\in\boC^2(\R)$ so that it is explicitly given by Theorem~\ref{thm:classiftwregu}, or 
$u_{c,\ka}$ is a constant function of modulus one.
\end{theorem} 
We call a weak solution that is not $\boC^2(\R)$, a singular solution. The rest of this subsection is devoted to explaining our results for singular solutions. We show that if $u_{c,\ka}\in\boX(\R)$ is a singular solution, then the function  $|u_{c,\ka}|^2$ must reach the value $1/(2\ka)$ at some point. Thus, we call the singular set of $u_{c,\ka}$ the nonempty set given by 
\begin{align}\label{def:gamma}
\Gamma(u_{c,\ka})=\Big\{x\in\R:\abs{u_{c,\ka}(x)}^2=\frac{1}{2\ka}\Big \}.
\end{align} This condition comes from the following observation:
Given  a  solution $u_{c,\ka}=u_1+i u_2\in\boX(\R)$ to \eqref{TWc}, with $u_1=\Re(u_{c,\ka})$, $u_2=\Im(u_{c,\ka})$, 
and recalling that $u_{c,\ka}$ is continuous (see Lemma~\ref{lem:finiteenergyassum}),
we can show that \eqref{TW:weak} can be recast 
as  a system of two {\em real} equations, satisfied in the weak sense,
\begin{equation}
	\label{eq:systtw}
	A(u_{c,\ka})\begin{pmatrix}
		u_1\\u_2
	\end{pmatrix}''=
	c	\begin{pmatrix}
		u_1\\u_2
	\end{pmatrix}'
	+(\eta_{c,\ka}-2\ka\abs{u_{c,\ka}'}^2)\begin{pmatrix}
		u_2\\
		-u_1
	\end{pmatrix},
	\text{ with } A(u_{c,\ka})=\begin{pmatrix}
		2\ka u_2u_1&-1+2\ka u_2^2\\
		1-2\ka u_1^2&-2\ka u_1u_2
	\end{pmatrix},
\end{equation}
so that 
\begin{equation}
	\label{determinant}
	\det (A(u))=1-2\ka|u_{c,\ka}|^2=1-2\ka+2\ka\eta_{c,\ka}.
\end{equation}
Thus, the set $\Gamma(u_{c,\ka})$ corresponds to the points where $\det(A(u_{c,\ka}))$
vanishes, so that the dispersion part of equation \eqref{eq:systtw} is singular. 
In this manner, as proved in Lemma~\ref{lem:omegaregu}, $u_{c,\kappa}$ is  smooth on  the complement of $\Gamma(u_{c,\ka})$, denoted by 
$\Omega(u_{c,\ka})=\Gamma(u_{c,\ka})^c$. Also, if we define the set of points where
$u_{c,\ka}$ is not differentiable (in the classical sense)
\begin{equation*}
	\boN(u_{c,\ka})=\{x\in\R : u_{c,\ka} \text { is not differentiable in }x \},
\end{equation*}
then $\boN(u_{c,\ka})\subseteq\Gamma(u_{c,\ka})$. We want to focus now on the case  where $\Gamma(u_{c,\ka})$ is not empty and bounded so that we can define the real numbers
\begin{equation}
\label{def:supcrit}
\ga_{c,\ka}=\inf\Gamma(u_{c,\ka})\quad \text{ and }\quad\gb_{c,\ka}=\sup\Gamma(u_{c,\ka}).
\end{equation}		
For instance, by condition \eqref{nonzero0}, the set  $\Gamma(u_{c,\ka})$ is bounded for any singular solution in the case $\kappa\neq 1/2$. To give explicit formulas for singular solutions,  we use the same approach as in Theorem~\ref{thm:classiftwregu}, introducing
the  intervals 
 \begin{align*}
 	\mathfrak{T}_\ka&=[ 1-1/(2\ka),0),\quad\mathfrak{T}_\ka^\circ=(1-1/(2\ka),0),\quad\text{ if }\kappa\in (0,1/2),\\
 	\mathfrak{J}_\ka&=(0,1-1/(2\ka)],\quad	\mathfrak{J}_\ka^\circ=(0,1-1/(2\ka)),\quad\text{ if }\ka\in (1/2, \infty),
 \end{align*}
and the following functions (see Lemma~\ref{lem:etawfunctions}).\\
 (i) For $(c,\ka)\in \D_1$, we set $f_{c,\ka}:\mathfrak{T}_\ka\rightarrow\R$ the function defined, for all $y\in\mathfrak{T}_\ka$, by
\begin{align}\label{eq:FD1W}
f_{c,\ka}(y)=2\sqrt{\ka}\atan\Big(\sqrt{\ka}\sqrt{\frac{ 2-c^2-2y}{1-2 \ka +2\ka y}}\Big)+2\sqrt{\frac{1-2\ka}{2-c^2}}\atanh\Big(\sqrt{\frac{(2-c^2)(1-2\ka+2\ka y)}{(1-2 \ka) (2-c^2-2y)}}\Big)-\pi\sqrt{\ka}.
 \end{align}
(ii) For $(c,\ka)\in \D_3$, we set $h_{c,\ka}:\mathfrak{J}_\ka\rightarrow\R$ the function defined, for all $y\in\mathfrak{J}_\ka$, by
 \begin{align}\label{eq:FD3W}
h_{c,\ka}(y)=f_{c,\ka}(y).
 \end{align}  
 (iii) For $(\sqrt{2},\ka)\in \B_-$, we set $g_\ka:\mathfrak{T}_\ka\rightarrow\R$ the function defined, for all $y\in\mathfrak{T}_\ka$, by
 \begin{align}\label{eq:FD4W}
 	g_\ka(y)=2\sqrt{\ka}\atan(\sqrt{\frac{-2\ka y}{1-2\ka+2\ka y}})+\sqrt{2}\sqrt{-\frac{1-2\ka+2\ka y}{y}}-\pi\sqrt{\ka},
 \end{align}
 (iv) For $(\sqrt{2},\ka)\in \B_+$, we set $\tilde{g}_\ka:\mathfrak{J}_\ka\rightarrow\R$ the function defined, for all $y\in\mathfrak{J}_\ka$, by
\begin{align}
\label{eq:FD5W}
	\tilde{g}_\ka(y)=g_\ka(y).
\end{align}
By Lemma~\ref{lem:etawfunctions}, the functions  \eqref{eq:FD1W}--\eqref{eq:FD5W} are injective and positive-valued, and their inverse functions 
$f_{c,\ka}^{-1}$,		$h_{c,\ka}^{-1}$, 		$g_{\ka}^{-1}$, $\tilde{g}_{\ka}^{-1}$
are well-defined on $(0, \infty)$. Notice that the functions $f_{c,\kappa}^{-1}$ and $g_{\kappa}^{-1}$ will describe antidark cuspons, while $h_{c,\kappa}^{-1}$ and $\tilde{g}_{\kappa}^{-1}$ will describe dark cuspons (see Figure~\ref{fig:cuspons}). Let $(c,\ka)\in\Tilde{\D}$ and $u_{c,\ka}\in\boX(\R)$ be a singular solution, so that $\ga_{c,\ka}$ and $\gb_{c,\ka}$ given by \eqref{def:supcrit} are reals. The next result establishes that $u_{c,\ka}$
is smooth outside the interval $[\ga_{c,\ka},\gb_{c,\ka}]$, that 
$\ga_{c,\ka}$ and $\gb_{c,\ka}$ belong to $\boN(u_{c,\ka})$,
and that $u_{c,\ka}$ is explicitly given by one of the functions in \eqref{eq:FD1W}--\eqref{eq:FD5W}, i.e.\ that it has a dark or antidark cuspon profile on $\R\setminus [\ga_{c,\ka},\gb_{c,\ka}]$.
\begin{theorem}
 \label{thm:classifcuspon}
 Let $(c,\kappa)\in \tilde{\D}$, if $u_{c,\ka}\in\boX(\R)$ is a singular solution for \eqref{TWc} so that $u_{c,\ka}$ is nontrivial and the critical set $\Gamma(u_{c,\ka})$ is not empty and bounded. 
 Let $\ga=\ga_{c,\ka}$ and $\gb=\gb_{c,\ka}$ be defined as in \eqref{def:supcrit} and assume without loss of generality that $\gb=0$. Then $\eta_{c,\ka}=1-\abs{u_{c,\ka}}^2$ satisfies, for all $x\geq 0$,
	  \begin{equation}\label{eq:etaw}
 \begin{aligned} 
		\eta_{c,\ka}(x)=f_{c,\ka}^{-1}(x),\ \text{ if }(c,\ka)\in \D_1,\quad
				\eta_{c,\ka}(x)=h_{c,\ka}^{-1}(x),\ \text{ if }(c,\ka)\in \D_3,\\
				\eta_{c,\ka}(x)=g_{\ka}^{-1}(x),\ \text{ if }(c,\ka)\in \boB_-,\quad
				\eta_{c,\ka}(x)=\tilde{g}_{\ka}^{-1}(x),\ \text{ if }(c,\ka)\in \boB_+.\\
		\end{aligned}
		\end{equation}
For all $x\leq\ga$, $\eta_{c,\ka}$ is obtained by reflection as $\eta_{c,\ka}(x)=\eta_{c,\ka}(\ga-x)$.
Additionally, for all $x\in(-\infty,\ga)\cup(0, \infty)$, we have 
$$1-\frac{1}{2\ka}<\eta_{c,\ka}(x)<0,\text{ if }(c,\ka)\in\D_1\cup\B_-,\quad\text{ whereas }\quad 0<\eta_{c,\ka}(x)<1-\frac{1}{2\ka},\text{ if }(c,\ka)\in\D_3\cup\B_+,$$
and 
$u_{c,\ka}$ is explicitly given by
\begin{align}\label{formule:cuspon}
    u_{c,\ka}(x)=\sqrt{1-\eta_{c,\ka}(x)}e^{i\theta_{c,\ka}(x)},\quad\text{ with }\quad\theta_{c,\ka}'(x)=c\eta_{c,\ka}(x)/(2-2\eta_{c,\ka}(x)).
\end{align}
Also, 
 $u_{c,\ka}\in\boC^\infty((-\infty,\ga)\cup(0, \infty)))\cap\boC(\R)$. 
 Moreover, 
    $\eta_{c,\ka}'(\ga^-)=-\infty$ and 
  $\eta_{c,\ka}'(0^+)= \infty,$ if $(c,\ka)\in\D_1\cup\B_-,$ while 	$\eta_{c,\ka}'(\ga^-)= \infty$ and 
    $\eta_{c,\ka}'(0^+)=-\infty,$
     if $(c,\ka)\in\D_3\cup\B_+.$	
 
 Finally $u_{c,\ka}$ satisfies the following decay estimates: for all $j\in\N$, there exist $C_1,C_2,C_3>0$ such that, for all $ x\geq 1,$
 $\abs{D^j u_{c,\ka}'(x)}+\abs{D^j\eta_{c,\ka}(x)}\leq C_1e^{-C_2x}$, if $(c,\ka)\in \D_1\cup \D_3$, and 
 $ \abs{ D^j u_{c,\ka}'(x)}+\abs{D^j\eta_{c,\ka}(x)}\leq C_3 x^{-(2+j)}$, 
 if $(c,\ka)\in \boB_+\cup \boB_-$.
\end{theorem}
 In Figure~\ref{fig:cuspons}, we place ourselves in the case $\ga_{c,\ka}=\gb_{c,\ka}=0$ to illustrate the antidark cuspon intensity profile given by the function $f^{-1}_{c,\ka}$, and the dark cuspon intensity profile given by $h^{-1}_{c,\ka}$.
We show in Lemma~\ref{lem:Z1/2} that if $\ka=1/2$, then $\ga_{c,\ka}=-\infty$ and $\gb_{c,\ka}=\infty$, so that there are infinitely many points in $\Gamma(u_{c,\ka})$.
Therefore, Theorem~\ref{thm:classifcuspon} allows us to complete the analysis if $u_{c,\ka}$ has only one singular point, as explained in the following result. 
\begin{corollary}[Cuspons]
\label{cor:cuspon}
Let $(c,\ka)\in\tilde{\D}$. Assume that $u_{c,\ka}\in\boX(\R)$ is a nontrivial solution for \eqref{TWc} such that $\Gamma(u_{c,\ka})=\{0\}$. 
Then, up to phase shift, the solution  $u_{c,\ka}$ is explicitly given by $u_{c,\kappa}=\sqrt{1-\eta_{c,\ka}}e^{i\theta_{c,\ka}},$
where  $\eta_{c,\ka}$ is the function in \eqref{eq:etaw}, 
and $\theta_{c,\kappa}(x)=\frac{c}2\int_0^x\frac{\eta_{c,\kappa}(y)}{1-\eta_{c,\kappa}(y)}$. In particular $\boN(u_{c,\ka})=\{0\}$.
\end{corollary}

\begin{figure}[ht!]
\centering
\begin{tabular}{cc}
			\resizebox{0.4\textwidth}{!}{
			\begin{overpic}
				[scale=0.45,trim=0 0 0 0,clip]{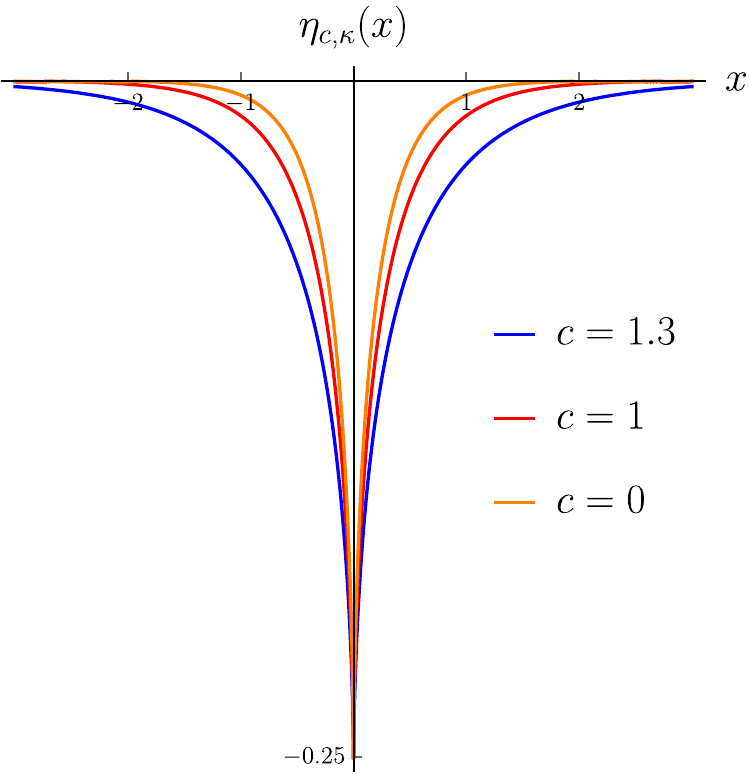}
			\end{overpic}
		}&
  \qquad
		\resizebox{0.4\textwidth}{!}{
			\begin{overpic}
				[scale=0.45,clip]{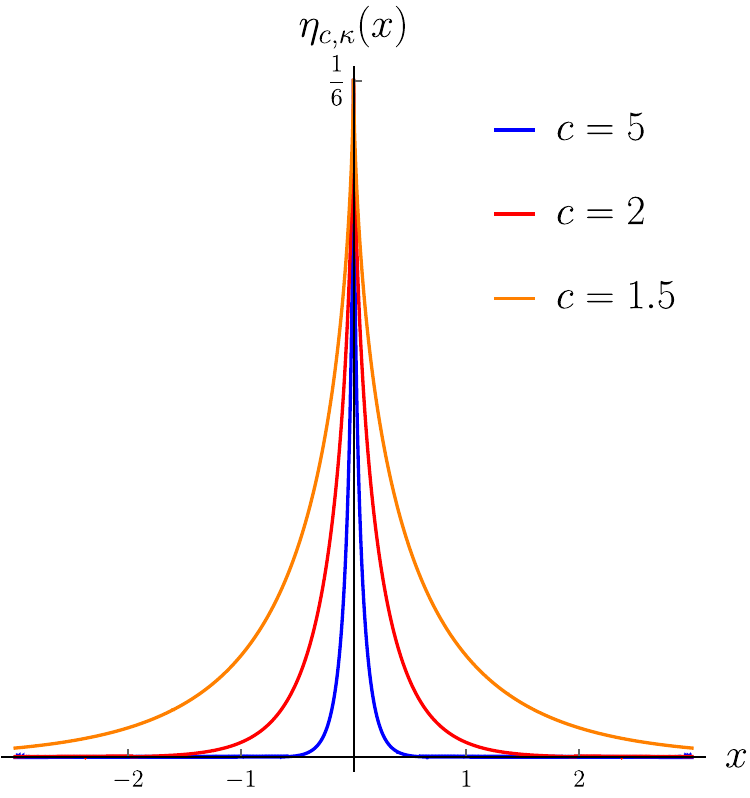}
			\end{overpic}
}
\end{tabular}
  \caption{Plots of the antidark cuspon intensity profile in Corollary~\ref{cor:cuspon}. The left panel displays antidark cuspons for the parameters $(0,0.4)$, $(1,0.4)$ and $(1.3,0.4)$ in  $\D_1$.  On the right, dark cuspons for the parameters $(1.5,0.6)$, $(2,0.6)$ and $(5,0.6)$ in  $\D_3$.}
 \label{fig:cuspons}
\end{figure}
with\subsubsection*{Classification  of solutions with two or more singular points}
We have provided the classification of solutions in the cases $\mathrm{card}(\Gamma(u_{c,\ka}))=0$ and $\mathrm{card}(\Gamma(u_{c,\ka}))=1$. Now we will focus on the case $$\mathrm{card}(\Gamma(u_{c,\ka}))\geq2,$$ 
 i.e.\ when $u_{c,\ka}$ has at least two singular points, and we do not require $\Gamma(u_{c,\ka})$ to be bounded anymore. Between two singular points $a,$ $b\in\Gamma(u_{c,\ka})$, the function $|u_{c,\ka}|^2$ may exhibit one of the following behaviors:
 \begin{enumerate}[leftmargin=4.5ex,topsep=1pt,itemsep=-3pt,partopsep=0ex,parsep=1ex]
     \item\label{intro:item:maxloc} It reaches a non-singular local extremum at some $x_0\in(a,b)$, so that $|u_{c,\ka}(x_0)|^2\ne1/(2\ka)$.
     \item\label{intro:item:constant} It  remains constant so that $|u_{c,\ka}(x)|^2=1/(2\ka)$, for all $x\in(a,b)$.
 \end{enumerate}
Assume that case~\ref{intro:item:maxloc} holds, so that there exists  $x_0\in (a,b)$ belonging  to the set
\begin{align}\label{def:boZ}
	\mathcal{Z}(u_{c,\ka})=\{x\in\R\backslash\Gamma(u_{c,\ka}):x\text{ is a local extremum of }1-\abs{u_{c,\ka}}^2\}.
\end{align}
When $(c,\ka)\in\tilde{\D}$, we know, by Theorem~\ref{thm:classifcuspon}, that $\eta_{c,\ka}$ coincides with a cuspon-like solution outside the bounded set $(\ga_{c,\ka},\gb_{c,\ka})$, so that it is monotonous on $\R\setminus(\ga_{c,\ka},\gb_{c,\ka})$. Therefore, we deduce that $\mathcal{Z}(u_{c,\ka})\subset(\ga_{c,\ka},\gb_{c,\ka})$.
 In conclusion, for every $(c,\ka)\in\tilde{\D},$
  we can define the closest singular points to any $x_0\in\mathcal{Z}(u_{c,\ka})$ as the real numbers
\begin{align}\label{def:ab0}
\ga^0_{c,\ka}=\sup\{x \in [\ga_{c,\ka},x_0) :x\in\Gamma(u_{c,\ka})\},\quad\text{ and }\quad\gb^{0}_{c,\kappa}=\inf\{x \in (x_0,\gb_{c,\ka}]:x\in\Gamma(u_{c,\ka})\}.
\end{align} 
In the case, $\ka=1/2$, by Lemma~\ref{lem:Z1/2}, we have  $\ga_{c,\ka}=-\infty$ and $\gb_{c,\ka}=\infty$,   and we can still define $\ga^0_{c,\ka}$ and $\gb^0_{c,\ka}$ as
above, with the obvious modifications.
 In particular, for any $(c,\ka)\in\tilde{\D}\cap \boC$, by Lemma~\ref{lem:omegaregu}, $u_{c,\ka}$ is smooth on $(\ga^0_{c,\ka},\gb^0_{c,\ka})$, and $\eta'_{c,\ka}(x_0)=0$. Letting $\eta_0=\eta_{c,\ka}(x_0)$, we establish in Lemma~\ref{lem:buleq} that there exist $K_0\geq0$ and $K_1\in\R$ such that $\eta=\eta_{c,\ka}$ satisfies the following ODEs in $(\ga^0_{c,\ka},\gb^0_{c,\ka})$: 
\begin{align}
		\label{eta2bintro}
		2(1-2\ka+2\ka\eta)\eta''+2\ka(\eta')^2&=P(\eta)+(\eta-\eta_0)P'(\eta),\\
		\label{eta1bintro}
		(1-2\ka+2\ka\eta)(\eta')^2&=(\eta-\eta_0)P(\eta),
	\end{align} 
 where 
 \begin{equation}
     \label{def:P}
     P(y)=-2y^2+(2-c^2-2\eta_0)y+(2-c^2)\eta_0-4K_0-4cK_1,\quad\text{ for all } y\in\R.
 \end{equation}
Using \eqref{eta2bintro}--\eqref{eta1bintro}, we establish that for any $x_0\in\mathcal{Z}(u_{c,\ka})$, the function $\eta_{c,\ka}$ must be symmetric on  
$(\ga^0_{c,\ka},\gb^0_{c,\ka})$ with respect to $x_0$, and strictly monotone on $(x_0,\gb^0_{c,\ka})$,
as follows.
\begin{theorem}\label{thm:symmetry}
     Let $(c,\ka)\in\tilde{\D}\cup\boC$. Assume that $u_{c,\ka}\in\boX(\R)$ is a solution to \eqref{TWc} and let $\eta_{c,\ka}=1-\abs{u_{c,\ka}}^2$. Suppose that $\mathrm{card}(\Gamma(u_{c,\ka}))\geq2$ and that there exists  $x_0\in\mathcal{Z}(u_{c,\ka})$, and consider its closest singular points
 $\ga^0_{c,\ka}$ and $\gb^0_{c,\ka}$, as in \eqref{def:ab0}. Then 
$\ga^0_{c,\ka}=2x_0-\gb^0_{c,\ka}$ and $\eta_{c,\ka}(x)=\eta_{c,\ka}(2x_0-x)$, for all $x\in(\ga^0_{c,\ka},x_0)$.
Moreover, if $0<\abs{u_{c,\kappa}(x_0)}^2<1/(2\kappa)$, then $\eta_{c,\ka}'<0$ in $(x_0,\gb^0_{c,\ka})$, while  if 
$\abs{u_{c,\kappa}(x_0)}^2>1/(2\kappa)$, then $\eta_{c,\ka}'>0$ in $(x_0,\gb^0_{c,\ka})$.
\end{theorem}
This monotonicity property implies that if $u_{c,\ka}$ vanishes at some points, then these points must be isolated. We deduce the following result that simplifies the analysis. It guarantees that the solutions are nonvanishing if $c>0$, while if $c=0$, the problem \eqref{TWc} reduces to one real differential equation. 
\begin{proposition}\label{prop:nonvanishsingu}
    Let $(c,\ka)\in\tilde{\D}\cup\boC$ and $u_{c,\ka}\in\boX(\R)$ be a solution to \eqref{TWc}. If $c>0$, then $u\in\boN\boX(\R)$. On the other hand, if $c=0$, then there exists $\phi\in\R$ such that $e^{i\phi}u_{c,\ka}(x)\in\R$ for all $x\in\R$.
\end{proposition}
An important remark is that, using $(\eta_{c,\ka}')^2\geq0$, equation \eqref{eta1bintro} yields algebraic constraints on $P$. For instance, we establish in the proof of the Theorem~\ref{thm:symmetry} that $P(\eta_{c,\ka})<0$, for all $x\in[x_0,\gb_{c,\ka}^0)$, thus, using the monotonicity of $\eta_{c,\ka}$, we conclude that
\begin{align}\label{condalgintro}
-2y^2+(2-c^2-2\eta_0)y+(2-c^2)\eta_0-4K_0-4cK_1<0,\text{ for all }y\in[\eta_0,1-1/(2\ka)).
\end{align}
Also, if $u_{c,\ka}$ is a singular solution, then we infer that $K_0=|u_{c,\ka}'(x_0)|^2$. 
 When $\eta_0<1$,  we show in Lemma~\ref{lem:K10} that $K_1=0$ so that,   \eqref{condalgintro} reduces to 
\begin{equation}\label{condalgintro22}
-2y^2+(2-c^2-2\eta_0)y+(2-c^2)\eta_0-c^2\eta_0^2/(1-\eta_0)<0,\text{ for all }y\in[\eta_0,1-1/(2\ka)),
\end{equation} 
whereas if $\eta_0=1$  (and thus $c=0$), then \eqref{condalgintro} becomes
\begin{equation}\label{condalgintro20}
y^2+2K_0-1>0,\text{ for all }y\in(1-1/(2\ka),1].
\end{equation}

Returning to case~\ref{intro:item:constant}, the obtained solution, which remains constant between two singular points $a<b$ is sometimes referred to as a stumpon \cite{lenells2005traveling}. However, the following proposition establishes that case~\ref{intro:item:constant} cannot occur unless $\ka=1/2$, in which case we have $|u(x)|^2=1$, for all $x\in(a,b)$.
\begin{proposition}\label{prop:constant}
    Let $(c,\ka)\in\tilde{\D}\cup\boC$. Assume that  $u_{c,\ka}\in\boX(\R)$ is a solution to \eqref{TWc}. If there exist  real numbers $a<b$ such that 
    \begin{equation}\label{eq:constant}
        |u_{c,\ka}(x)|^2=1/(2\ka),\quad\text{ for all }x\in(a,b),
    \end{equation}
    then $\ka=1/2$.
\end{proposition} 
In view of Proposition~\ref{prop:constant}, one might expect that case~\ref{intro:item:constant} implies that $u_{c,\ka}$  is a constant function  on $\R$, with $\abs{u_{c,\ka}}\equiv 1$. However, the critical case $\ka=1/2$ is degenerate, and there are many more possibilities for the solutions. The most remarkable singular solutions are dark compactons, i.e.\ such that the intensity profile $\eta_{c,\ka}=1-\abs{u_{c,\ka}}^2$ has compact support; some of them are explicitly given in Proposition~\ref{prop:compacton}.


So far we ruled out case~\ref{intro:item:constant}, and obtained, in case~\ref{intro:item:maxloc}, the necessary conditions \eqref{condalgintro22}--\eqref{condalgintro20}.
 We show in Propositions~\ref{prop:exibu}--\ref{prop:exibu3} that these conditions are, in fact, sufficient to construct local solutions to \eqref{TWc}.
More precisely, if $\eta_0<1$ is such that \eqref{condalgintro22} holds, then there exists $\ga_{c,\ka}^0<\gb_{c,\ka}^0$ and a unique $\eta\in \boC^2((\ga_{c,\ka}^0,\gb_{c,\ka}^0))\cap H^1((\ga_{c,\ka}^0,\gb_{c,\ka}^0))$ satisfying \eqref{eta2bintro}--\eqref{eta1bintro} with $\eta(0)=\eta_0$ and $K_0=(c\eta_0)^2/(4-4\eta_0)$  so that 
\begin{equation}
\label{localsol}
    u=\sqrt{1-\eta}e^{i\theta},\quad\text{ with }\quad\theta'=\frac{c\eta}{2(1-\eta)},
\end{equation}
is a $\boC^2((\ga_{c,\ka}^0,\gb_{c,\ka}^0))$-solution to \eqref{TWc}.
Similarly, if $\eta_0=1$ and $K_0\geq0$ is such that \eqref{condalgintro20} holds, then  there exists $\ga_{c,\ka}^0<\gb_{c,\ka}^0$ and a unique $\eta\in \boC^2((\ga_{c,\ka}^0,\gb_{c,\ka}^0))\cap H^1((\ga_{c,\ka}^0,\gb_{c,\ka}^0))$ satisfying \eqref{eta2bintro}--\eqref{eta1bintro} with $\eta(0)=1$ so that 
\begin{equation}\label{localsol0}
    u(x)=\pm\sqrt{1-\eta(x)},\quad\text{ for all }\pm x\in[0,\gb_{0,\ka}^0),
\end{equation} 
is a $\boC^2((\ga_{c,\ka}^0,\gb_{c,\ka}^0))$-solution to \eqref{TWc} with $c=0$.
In both cases, we also find that the intensity profile of $u$ satisfies the boundary conditions $\eta(\ga_{c,\ka}^0)=\eta(\gb_{c,\ka}^0)=1-1/(2\ka).$

These local solutions allow us to build every possible solution such that $\card(\mathcal{Z}(u_{c,\ka}))<\infty$. Using Lemma~\ref{lem:buildw} to continuously glue together finitely many local solutions and extending them on $\R$ with a cuspon-like solution (respectively with constants of modulus one if $\ka=1/2$), one obtains a composite wave solution (respectively a compacton). This ends the classification in the case $2\leq\card(\Gamma(u_{c,\ka}))<\infty$, 
since in this case, as explained in Lemma~\ref{lem:classifcard}, we have $\card(\mathcal{Z}(u_{c,\ka}))=\card(\Gamma(u_{c,\ka}))-1$.



We can explicitly compute the set of admissible $\eta_0$ in terms of $(c,\ka)\in\tilde{D}\cup\boC$ using \eqref{condalgintro22}--\eqref{condalgintro20}, however, for the sake of simplicity, we only show that singular solutions $u\in\boX(\R)$ such that $\mathcal{Z}(u)=\{0\}$ exists when $\eta_0$ is negative enough. 
  \begin{proposition}[Composite waves]
\label{prop:compos}
 For any  $(c,\ka)\in\tilde{\D}\cup\boC$, there exists a negative constant $A<1-1/(2\ka)$, such that if $\eta_0\in(-\infty,A)$, then there exists  a unique solution $u_{c,\ka}\in\boX(\R)$, up to phase shift, to \eqref{TWc} satisfying 
 $$\mathcal{Z}(u_{c,\ka})=\{0\} \quad \text{ and }\eta_{c,\ka}(0)=\eta_0.$$ 
  Also, $u\in\boN\boX(\R)$, $\eta_0$ is the global minimum of $\eta_{c,\ka}$, and 
 $\card(\boN(u_{c,\ka}))=2$.
In addition,  if  $(c,\ka)\in\tilde{\D}$, then  $\Gamma(u_{c,\ka})=\boN(u_{c,\ka})$, 
  whereas if $(c,\ka)\in\boC$, then $u_{c,\ka}$ is a compacton (i.e.\ $\eta_{c,\ka}$ is compactly supported).
 	\end{proposition}
 The left panel in   Figure~\ref{fig:composit} displays a numerical approximation of the solution $\eta_{c,\ka}$ to \eqref{eta2bintro}--\eqref{eta1bintro}, with $\eta_0=-10$, $c=1$, $\ka=1/2$ and $K_0=c\eta_0/(4-4\eta_0)$ (and $K_1=0$), given by Proposition~\ref{prop:compos}.
On the other hand, in the critical case $\ka=1/2$, we can also give
  a  family of explicit compactons by solving \eqref{eta2bintro}--\eqref{eta1bintro}, with $\eta_0=1-c^2/2$, $K_0=(2-c^2)^2/8$ (and $K_1=0$). In fact, with this choice of parameters, equation \eqref{eta1bintro} becomes \eqref{EDO:intro} with $\ka=1/2$, so we recover the explicit solution computed in Proposition~\ref{prop:k1/2} with $\gb_{c,\ka}^0=\pi/\sqrt{2}$, as follows.
\begin{proposition}[Compactons] 
\label{prop:compacton}
Let  $(c,\ka)\in\boC$ with $c\ne\sqrt{2}$. For $j\geq1$ an odd integer, set the interval $I_j=(-j\pi/\sqrt{2},j\pi/\sqrt{2})$. If $c=0$, we define
 $$u^{(j)}_{c,1/2}(x)=\sin(x/\sqrt{2}),\quad  \text{ for all } I_j,$$
whereas if $c> 0$ with $c\ne\sqrt{2}$, we define,  for all $x\in I_j$,
$$u^{(j)}_{c,1/2}(x)=\sqrt{1-\frac{(2-c^2)}{2}\cos^2\Big(\frac{x}{\sqrt{2}}\Big)}\, e^{i\theta(x)},\text{ with }
\theta(y)=\frac{\pi}2+k\pi-\frac{c y}2  -
 							\atan\Big(\frac c{\sqrt 2}\cot\Big(\frac y{\sqrt 2}\Big) \Big),
$$
for $y\in (k\sqrt 2\pi, (k+1)\sqrt 2\pi)\cap I_j$, for all $k\in\Z$.
 In both cases, we extend $u^{(j)}_{c,1/2}$ to $\R$ as a continuous function, which is constant outside $I_j$. Then  $u^{(j)}_{c,1/2}\in\boX(\R)$ and is a weak solution to \eqref{TWc}.		    
\end{proposition}    
Notice that the family of compactons $(u^{(j)})_{j\in\N^*}$ given by Proposition~\ref{prop:compacton},
 satisfies $\mathcal{Z}(u_{c,\ka}^{(j)})=\{k\sqrt{2}\pi~|~ k\in\Z,~|k|<|j|\}$ and $\boN(u_{c,\ka})=\emptyset$, yet $(u^{(j)})_{j\in\N^*}$ is not a family of $\boC^2(\R)$-solutions because the second order derivative is discontinuous at $x=\pm j\pi/\sqrt{2}$. 
In Figure~\ref{fig:composit}, we also plot the intensity profile of the compacton $u^{3}_{1,1/2}$ in the center panel, and its phase in the right panel.
\begin{figure}[ht!]
\begin{tabular}{ccc}
	  \resizebox{0.3\textwidth}{!}{
		\begin{overpic}
			[scale=0.5,clip]{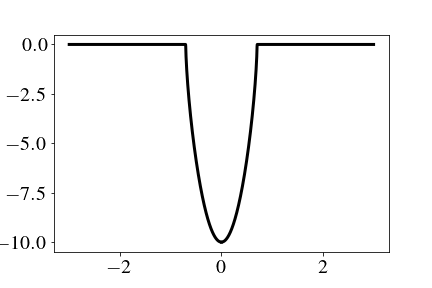}	
		\end{overpic}}
	&
				\resizebox{0.3\textwidth}{!}{
			\begin{overpic}
				[scale=0.5,clip]{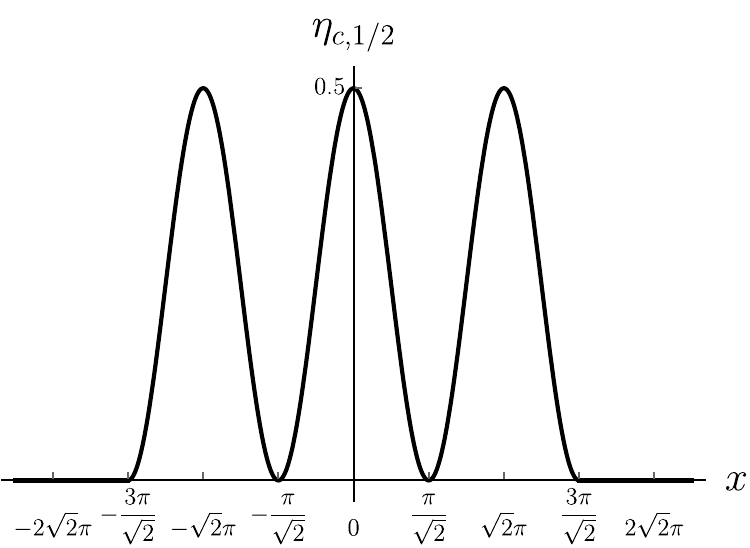}
			
			\end{overpic}
    }
		&\resizebox{0.3\textwidth}{!}{
			\begin{overpic}
				[scale=0.5,clip]{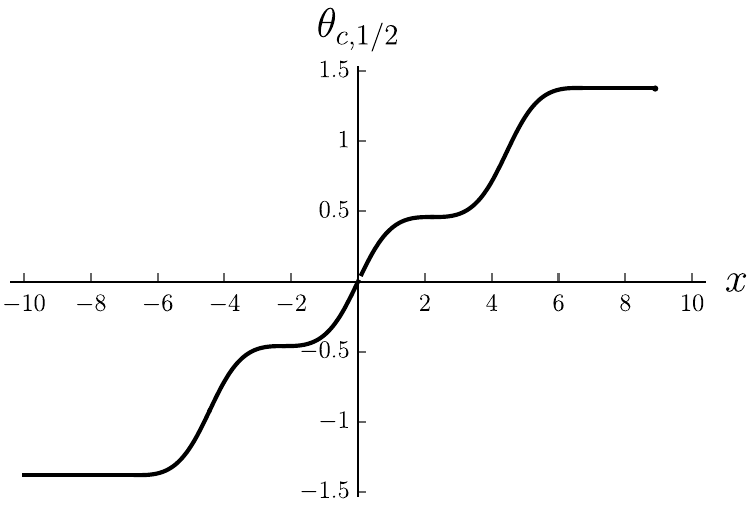}
			
			\end{overpic}	
		}

	\end{tabular}
\caption{
	On the left, the intensity profile of the compacton given by Proposition~\ref{prop:compos}  with $\eta_0=-10$, $c=1$ and $\kappa=1/2$.
The center and the right panel display, respectively, the intensity profile and the phase of the compacton $u_{1,1/2}^{(3)}$ in  Proposition~\ref{prop:compacton}.
}
\label{fig:composit}

\end{figure}

To our knowledge, only few results deal with dark solitons for \eqref{QGP} when $\kappa\neq 0$. A branch of explicit dark solitons was found in \cite{krolikowski2000} for $\kappa\in [0,1/2)$. Also, in the setting of Korteweg models, Benzoni-Gavage, Danchin, Descombes and Jamin~\cite{benzoniStab} obtained existence of smooth homoclinic and heteroclinic traveling waves by classical ODE phase portrait analysis. Their result applies to \eqref{QGP} and provides the existence of dark solitons for $c\in(0,\sqrt 2)$ and $\ka\leq0$, which corresponds to homoclinic waves in their setting. In view of the uniqueness stated in  Theorem~\ref{thm:classiftwregu}, these homoclinic waves are equal to our smooth dark solitons. However, it is not clear how to extend the method in \cite{benzoniStab} to classify dark solitons when $\ka>0$, because one needs to prevent the coefficient  $(1-2\kappa\abs{u}^2)$ in \eqref{EDO:intro} from vanishing. Therefore, the analysis presented above extends the results \cite{krolikowski2000,benzoniStab}  for all $\ka\in\R$, and provides a complete classification of finite energy localized solutions, including exotic solitons.

Concerning the higher dimensional case, the only result the authors are aware of is the work by Audiard~\cite{audiard2017}, proving existence of small energy dark solitons, for $\kappa<0$, in two dimensions by constrained minimization, in the spirit of \cite{bethuel}.

\subsection{Variational characterization and stability}
We explain now the strategy to study the orbital stability of solitons associated with parameters in $\boD_2$
by using a variational characterization.
In Lemma~\ref{lem:eulerlag}, we show that equation \eqref{TWc} can be recast in terms of the energy $E_\ka$ and the momentum $p$ as 
	\begin{equation}
 \label{EL}
	    \d E_\kappa(u)=c\, \d p(u),
	\end{equation}
so that it is natural to consider the following minimization problem 
	\begin{align}
		\label{min}
		\boE_{\kappa}(\gq)&=	\inf\{E_\ka(u) : u\in \boN\boX(\R),~p(u)=\gq\}.				\end{align} 
Indeed, if a minimizer exists, it must satisfy the Euler--Lagrange equation \eqref{EL}, and thus \eqref{TWc}, where $c\in\R$  appears as a Lagrange multiplier.
We will show in  Proposition~\ref{prop:Edegen}, that $\boE_{\kappa}(\gq)=-\infty$, if $\kappa>0$, so that we restrict ourselves to the case $\kappa\leq 0$. 
This minimization problem was studied for $\kappa=0$ in \cite{bethuel2008existence,deLGrSm1}, where the authors showed that the curve $\gq\mapsto \boE_0(\gq)$ is a well-defined function on $\R$, that is even and continuous, so it suffices to consider $\q\geq 0$. Moreover, for $\gq\in(0,\pi/2)$, the minimizer is attained, and corresponds, up to invariances, to the dark soliton $u_{c,0}$ in \eqref{sol:1D}, where the  speed $c\in(0,\sqrt 2) $ is given by the equation
$$\pi/2-\atan(c/\sqrt{2-c^2})-c\sqrt{2-c^2}/2=\gq.$$
Also, the energy of the solution is $\boE_0(\gq)=E(u_{c,0})=(2-c^2)^{3/2}/3$.
For $\gq>\pi/2$, the curve $\boE_0$ is constant, and the minimizers are not reached.

Therefore, we focus now on the case $\ka<0$. In view of Theorems~\ref{thm:classiftwregu} and \ref{thm:non-singular-sol}, if a minimizer of $\boE_{\kappa}$ is attained, then it is 
given, for $c\in(0,\sqrt2)$, by the dark soliton in \eqref{formule:soli}, 
i.e.\
\begin{align}\label{formule:soli:D2}
u_{c,\kappa}(x)=\sqrt{1- \boG_{c,\kappa}(x)}\, e^{i\theta_{c,\kappa}(x)},\  
\text{ with }\ \theta_{c,\kappa}(x)=\frac{c}2\int_0^x\frac{\boG_{c,\kappa}(y)}{1-\boG_{c,\kappa}(y)}dy, \text{ for all }x\in\R,
\end{align}
depicted in the center panel of Figure~\ref{fig:solitons}.

A complete study of the function $\boE_\kappa$ is done in Section~\ref{sec:minimization}. In particular, we will prove that $\boE_\ka$ is even and continuous
on $\R$, and also that it is nondecreasing, concave, and strictly subadditive on $\R^+$. These properties will enable us to apply the concentration-compactness argument to study the minimizers of $\boE_\kappa$.

On the other hand, the definition of $\boE_\kappa$ is not well-adapted to study solitons that vanish at some point, such as the black solitons. Indeed, their analysis is much more involved, as explained in \cite{blacksoliton,asymptstabblack,AlejoCorchoQuinticStab,DuanWangDarkAndBlackStab}. Hence, we only consider here nonvanishing solitons. Moreover, to prove the existence of minimizers, we need to guarantee that the limit of the minimizing sequences can be lifted, so that the momentum of the limit function is well-defined. For this purpose, we follow the 
method developed by de Laire and Mennuni in~\cite{delaire-mennuni}, by introducing the following the critical value for the momentum 
		\begin{align}
			\label{q_*}
			\gq_\ka^*=\sup\{\gq>0 : \forall v \in \mathcal{E}(\mathbb{R}),~ E_\ka(v)\leq \boE_\ka(\gq)\Rightarrow \underset{\mathbb{R}}\inf|v|>0\}.
		\end{align}
We show that  $\gq_\ka^*$ is related to the energy of the black soliton $u_{0,\ka},$ 
\begin{equation}
\label{def:k:blacksoliton}
u_{0,\ka}(x)=\pm\sqrt{1-\boG_{0,\ka}(x)},  \text{ for all }\pm x\geq0.
\end{equation}
 Indeed, it will be key to introduce the critical value for the speed 
 \begin{equation}
     \label{def:c*}
     c_\ka^*=\max\{c\in [0,\sqrt 2) : E_\ka(u_{c,\ka})=E_\ka(u_{0,\kappa})\},
 \end{equation}
that enables us to characterize $\gq_\ka^*$ in the following manner:
 \begin{equation}
\label{characterization-q}
\gq_\ka^*=p(u_{c^*_{\ka},\ka}), \ \text{if}\  c_\ka^*>0,
\quad \text{ and }\quad \gq_\ka^*=\pi/2,\  \text{ if }c_\ka^*=0.
\end{equation}
In Section~\ref{sec:energymoment}, we obtain formulas for the energy and momentum of solitons and cuspons, in particular for the energy and momentum of dark solitons $u_{c,\ka}$ with $\ka<0$ and $0\leq c<\sqrt{2}$. This enables us to show in Lemma~\ref{lem:allure} and Corollary~\ref{prop:difeo} that there exists a threshold  $\ka_0\approx-3.636$ such that we have
  $c_\ka^*=0$ if $\ka\in [\ka_0,0)$, and 
  $c_\ka^*\in(0,\sqrt 2)$ if $\ka<\ka_0.$
  Corollary~\ref{prop:difeo} also shows that the function  $p_\ka : [c_\ka^*,\sqrt{2}]\to [0,\gq_\ka^*]$, given by
 $p_\kappa(c)= p(u_{c,\ka})$, is continuous, bijective, and strictly decreasing, so its inverse  is well-defined, and we denote it by 
 $\gc_\kappa : [0,\gq_\ka^*] \to [c_\ka^*,\sqrt{2}].$ 
  Moreover, we introduce the number $\tilde{c}_\kappa\in[0, c^*_\ka]$ satisying
\begin{equation}\label{def:tildec}
    \frac{d}{d c}p(u_{c,\ka})\geq0,\quad\text{ for all }c\in(0,\tilde{c}_\kappa),\quad\text{ and, }\quad\frac{d}{d c} p(u_{c,\ka})<0,\quad\text{ for all }c\in(\tilde{c}_\kappa,\sqrt{2}).
\end{equation}
In Figure~\ref{fig:EP:kappa50}, we plot, for $\kappa=-50$, the functions $E_\kappa(u_{c,\ka})$ and $p(u_{c,\ka})$ depending on $c\in(0,\sqrt{2})$, where the quantities defined above are approximatively $c_\ka^*\approx0.75$,  $\tilde c_{\kappa}\approx0.51$, and $\mathfrak{q}_\ka^*\approx3.5$.

\begin{figure}[ht!]
\centering
	\begin{tabular}{ccc}
		\resizebox{0.45\textwidth}{!}{
			\begin{overpic}
				[scale=0.5,trim=0 0 0 24,clip]{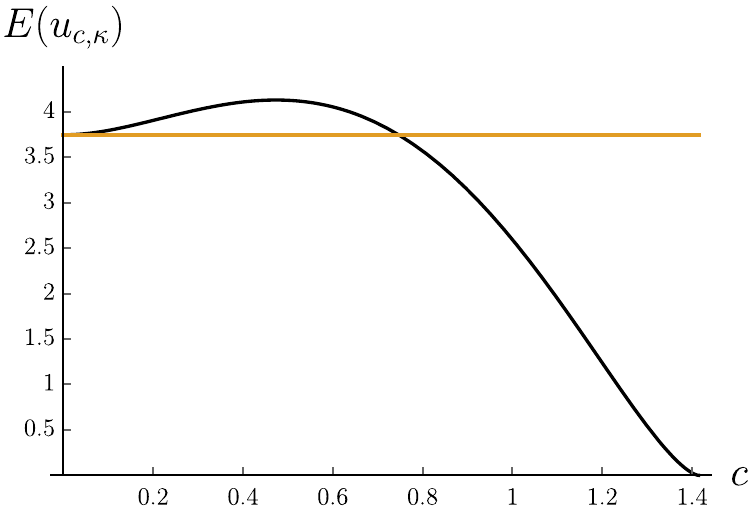}
				\dottedline{3}(52.5,7)(52.5,50)
                \dottedline{3}(38,7)(38,55)    
                \put(50,7){\tiny{$c^*_\kappa$}}
                \put(38,7){\tiny{$\tilde c_\kappa$}}
                \put(-4,62){\small{$E_\kappa(u_{c,\ka})$}}
			\end{overpic}
		}
		\resizebox{0.45\textwidth}{!}{
			\begin{overpic}
				[scale=0.5,trim=0 0 0 24,clip]{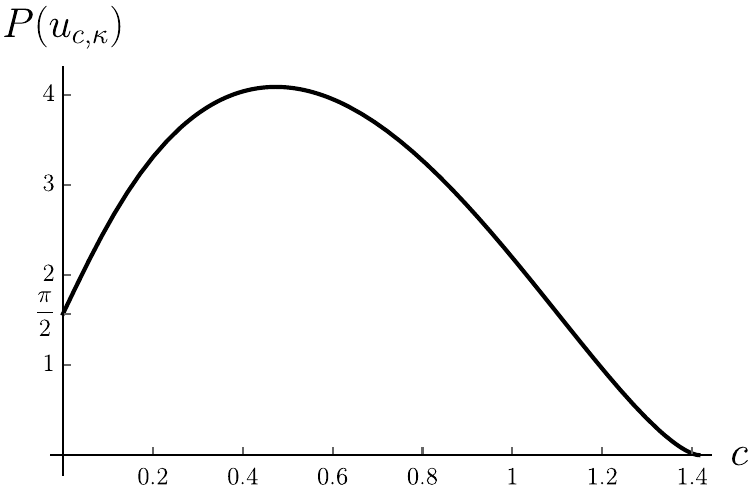}
                \dottedline{3}(38,7)(38,54)  
                \dottedline{3}(9,47)(52,47)  
                \dottedline{3}(53,7)(53,47)  
                \put(3.5,47){\tiny{$\gq_0$}}
                \put(-4,62){\small{$p(u_{c,\ka})$}}
                \put(38,7){\tiny{$\tilde c_\kappa$}}
             \put(53,7){\tiny{$c^*_\kappa$}}
			\end{overpic}
		}
	\end{tabular}	
	\caption{Plot of the energy $E_\ka(u_{c,\ka})$ and momentum $p(u_{c,\ka})$ of the dark solitons $u_{c,\kappa}$, for $(c,\kappa)\in \D_2$, with $\kappa=-50$. The left panel displays $E_\ka(u_{c,\ka})$ and the constant line crossing $E_\ka(u_{0,\ka})$ in orange. The value $\tilde c_\kappa\approx0.51$ is defined in \eqref{def:tildec}, $c^*_\kappa\approx0.75$ is the point defined in \eqref{def:c*}. The right panel depicts $p(u_{c,\ka})$, and the values of  $\Tilde{c}_\ka$ and $c_\ka^*$.} 
 \label{fig:EP:kappa50}
\end{figure}
There are   two possible behaviors for the energy-momentum diagram of dark solitons, 
depending on the critical value $\ka_0$.
For $\ka\in [\ka_0,0)$, the shape of the energy-momentum diagram is a concave increasing curve, as the one depicted in the left panel in Figure~\ref{fig:EPdiag} for $\kappa =-3$. 
For $\ka\in (-\infty,\ka_0)$, there is a cusp in the diagram, as seen in the right panel in Figure~\ref{fig:EPdiag} for $\kappa =-50$. To explain the cusp at $c=\tilde{c_\ka}$, recall the Hamiltonian group property (see Lemma~\ref{lem:hamiltongp}):
\begin{align}\label{eq:hamiltongp}
	\frac{d}{d c}E_\ka(u_{c,\ka})=c\frac{d}{d c}p(u_{c,\ka}).
\end{align}
Hence, due to the change in the monotonicity of the functions $c\mapsto p_{\ka}(u_{c,\ka})$ and $c\mapsto E_{\ka}(u_{c,\ka}$) at $c=\tilde{c}_\ka$, we observe the cusp in the right diagram of Figure~\ref{fig:EPdiag} because the tangent vector must jump from $(1,\tilde{c}_\ka^-)$ to $-(1,\Tilde{c}_\ka^+)$.

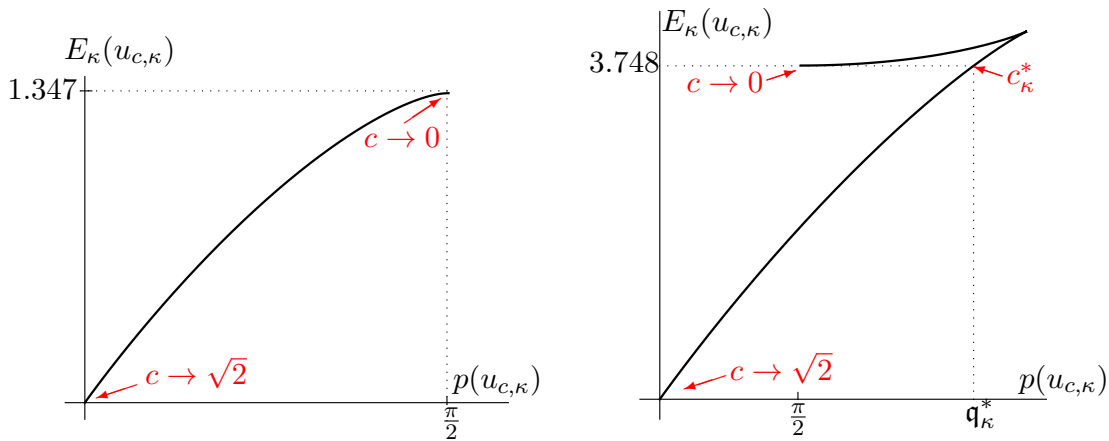
\begin{figure}[ht!]
\centering
\begin{tabular}{lr}
\resizebox{0.35\textwidth}{!}{
  			\begin{overpic}
				[scale=0.4,trim=0 0 0 0,clip]{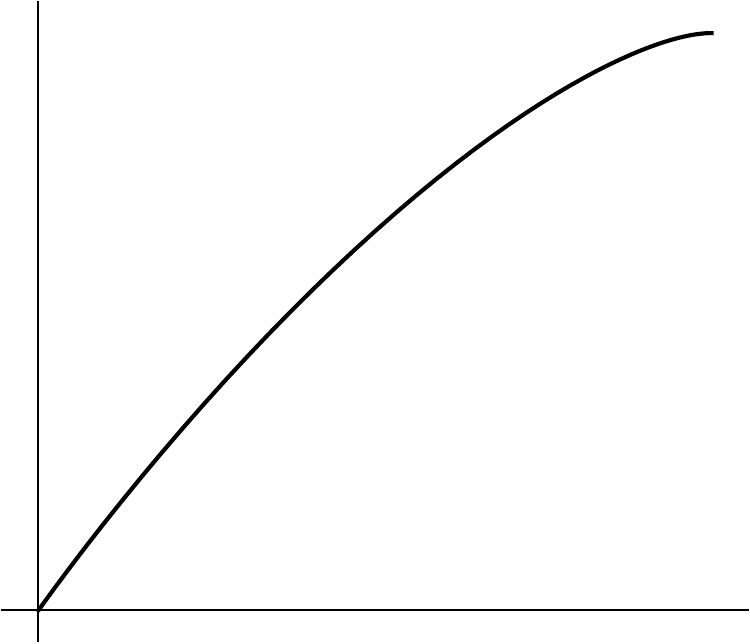}
                      \put(-14,80){1.347}
                       \put(4,82){\line(1,0){3}}
                \dottedline{2}(4,82)(94.5,82)
                \dottedline{2}(94.5,0)(94.5,82)      
				\put(84,74){\color{red}\vector(3,2){9}}
				 \put(74,68){\color{red}{$c\to0$}}
				\put(18,9){\color{red}\vector(-3,-1){10}}
				\put(20,9){\color{red}{$c\to\sqrt{2}$}}
                \put(96,8){$p(u_{c,\ka})$}
                \put(0,90){$E_\ka(u_{c,\ka})$}
                \put(93,-2){$\frac{\pi}{2}$}
                \put(94.5,4){\line(0,1){2}}
                \put(99.8,4.45){\line(1,0){10}}
			\end{overpic}
}
  &
  \qquad \qquad
\resizebox{0.35\textwidth}{!}{
			\begin{overpic}
				[scale=0.4,trim=0 0 0 0,clip]{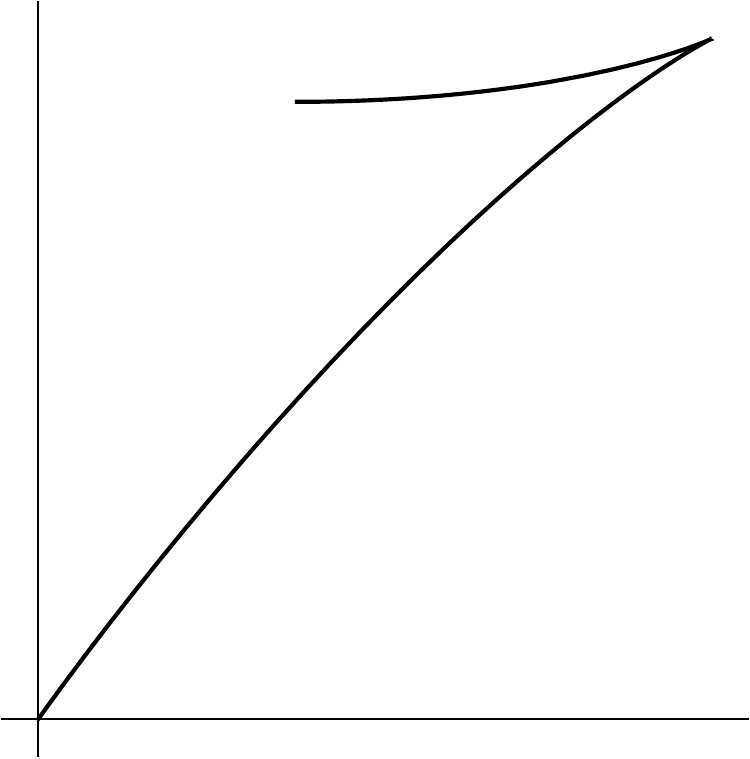}
				\put(31,83){\color{red}\vector(3,1){7}}
				\put(11.5,80){\color{red}{$c\to0$}}
				\put(21,11){\color{red}\vector(-3,-1){10}}
				\put(22,10){\color{red}{$c\to\sqrt{2}$}}
				\put(88,84){\color{red}\vector(-3,1){7}}
				\put(89,82){\color{red}{$c_\ka^*$}}
				\dottedline{2}(81,86.5)(81,5.2)
                \put(79,0){$\gq_\ka^*$}
                \dottedline{2}(5,86.5)(81,86.5)
                \put(-12,85){3.748}
                 \put(92,8){$p(u_{c,\ka})$}
                \put(5.5,95){$E_\ka(u_{c,\ka})$}
                 \put(36,-1){$\frac{\pi}{2}$}
                \put(38.5,5){\line(0,1){2}}
				\put(95,94){\color{red}{$\tilde{c}_\ka$}}
		\end{overpic} 
  }
     \end{tabular}
   \caption{Energy-momentum diagram of dark solitons with parameters in $\boD_2$, with  $\ka=-3$ (left) and $\ka=-50$  (right).}
   \label{fig:EPdiag}
   \end{figure}

Using $\gq_\ka^*$ in \eqref{q_*} and the function $\mathfrak{c}_\ka$, we can state our main result concerning the variational characterization of the dark solitons in the region $\boD_2$ as follows. 

\begin{theorem}\label{thm:min}
Let $\ka\leq0$ and $\gq\in (0,\gq_\ka^*)$. Then the infimum for the minimization problem \eqref{min} is attained at
		 the dark soliton $u_{\gc(\gq),\kappa}$ in \eqref{formule:soli:D2}, i.e.\
$\boE_\kappa(\gq)=E_\kappa(u_{\gc(\gq),\kappa}),$
and is the only minimizer, up to invariances. Moreover, $\boE_\ka(\gq)=E_\ka(u_{0,\ka})$ for all $\gq>\gq_\ka^*$ and this infimum is not attained. 
\end{theorem}
\begin{remark}\label{rem:notmin}
   Theorem~\ref{thm:min} establishes that the curve $\boE_\ka$ coincides the energy of the dark solitons $u_{c,\ka}$ with moment equal to $\gq\in(0,\gq_\ka^*)$ and speed $c\in(c_\ka^*,\sqrt{2})$. Moreover, the curve $\boE_\ka$ is constant equal to $E_\ka(u_{0,\ka})$ in  $(\gq_\ka^*,\infty)$. In particular, if $\ka\in (-\infty,\ka_0),$ so that $c_\ka^*\in (0,\sqrt{2})$, we see  that the energy of the soliton $u_{c,\ka}$ with $c\in(0,c_\ka^*)$ lies above the curve $\boE_\kappa$,  as illustrated in the right panel of Figure~\ref{fig:EPdiag}. Therefore, solitons with speed $c\in (0,c_\kappa^*)$ are not minimizers for $\boE_\ka.$
\end{remark}

We will see in Section~\ref{section:cauchy} that the Cauchy problem for \eqref{quasilin}, with $k<0$, is locally well-posed in $u_{c,\kappa}+H^s(\R)$, for $s> 5/2$ as stated in Corollary~\ref{coro:lwp},
including the conservation of energy and momentum by the flow.
Therefore,  by using the Cazenave--Lions argument \cite{cazlions},
and endowing $\boX(\R)$ with 
the distance: $$d(u_1,u_2)=\norm{u_1'-u_2'}_{L^2(\R)}+\norm{|u_1|-|u_2|}_{L^2(\R)}+|u_1(0)-u_2(0)|,$$
 we will deduce that the variational characterization leads to the orbital stability of the dark solitons, as follows.
\begin{theorem}\label{thm:stability}
Let $\ka<0$ and $c\in (c_\ka^*,\sqrt{2})$, then the dark soliton $u_{c,\kappa}$ in \eqref{formule:soli:D2} is orbitally stable in $(\boX(\R),d)$, in the following sense.
For all $\varepsilon>0,$ there exists $\delta >0$ such that if 
 $\Psi_0\in u_{c,\ka}+H^s(\R)$, $s>5/2$, satisfies $\inf_\R|\Psi_0|>0$ and 
  $$ d(\Psi_0,u_{c,\kappa})\leq \delta,$$
then,  for all $ t\in [0,T_{\Psi_0})$,
 $$ \inf_{(y,\phi)\in \R^2}d(\Psi(\cdot,t),e^{i\phi}u_{c,\kappa}(\cdot -y))\leq \varepsilon,$$
 where  $\Psi\in\boC([0,T_{\Psi_0});u_{c,\ka}+H^s(\R)$ is the solution in Corollary~\ref{coro:lwp}, with initial condition $\Psi_0$.
	\end{theorem}

\begin{remark}
Equation \eqref{quasilin0} can be written in a more general form as
		\begin{equation}\label{quasilin}
			i\ptl_t\Phi=\Delta\Phi+V(x)\Phi+\gs\Phi(f(|\Phi|^2)\Phi+\ka h'(\abs{\Phi}^2)\Delta h(|\Phi|^2)),\quad\text{ in }\R^d\times\R,
		\end{equation} 
with $f(y)=h(y)=y$ for all $y\in\R$ and $V\equiv0$. 
Considering  vanishing conditions at infinity with focusing nonlinearities $\gs=1$ and $\ka\leq0$, the existence and orbital stability of bright solitons for equation \eqref{quasilin} has been addressed in \cite{COLINdual,Colin2}  
for $V\equiv0$ , and in \cite{LIUgroundstate,Ruizgroundstate} for wide classes of $V\in\boC(\R,\R)$ bounded from below. 
After a phase shift, these solutions are real-valued, which simplifies the problem.
This fact is key to using the duality method in \cite{COLINdual,Colin2}, 
 which enables the transformation of the elliptic quasilinear problem into a semilinear one.
As explained by Selvitella \cite{AlessandroDual}, the dual method does not work for complex-valued functions, thus it is not well-suited for the study of \eqref{TWc} for $c\neq 0$.

\end{remark}

\begin{remark}
\label{rem:nonlocal}
The global well-posedness in the energy space and the properties of dark soliton for \eqref{eq:nonloc} have been addressed by the first author in \cite{delaire,delaire-mennuni}. Assuming that $\wh \W$ is {\em bounded}, with $\widehat{\boW}(\xi)\geq(1-\ka\xi^2)^+$, for all $\xi\in\R$, with $\kappa\in [0,1/2],$ and some additional conditions, it was shown that there exists a branch of dark solitons to \eqref{eq:nonloc}.
Although there is no analytical formula for these dark solitons, it was proven in \cite{delaire-mennuni} that some of them can be obtained by minimization at fixed momentum,
using the value $\gq_\ka^*$ in \eqref{q_*}. However, obtaining good estimates for $\gq_\ka^*$ in the nonlocal case remains an open problem.
\end{remark}

As discussed in Remark~\ref{rem:notmin}, some dark solitons of the branch are not global minimizers, with fixed momentum.
The analysis of the orbital stability of the other dark solitons with parameters in $\boD_1$ and $\boD_2$ is done in the companion paper~\cite{lequiniou2024stability}. More precisely, in \cite{lequiniou2024stability} the author shows that 
\begin{enumerate}
    \item[(a)]\label{item:stabD1} for $(c,\kappa)\in \D_1$, $c\neq 0$, dark solitons $u_{c,\ka}$ are orbitally stable.
    \item[(b)]\label{item:stabD2} for $(c,\kappa)\in \D_2$  dark solitons are stable whenever $c\in(\tilde{c}_\ka,\sqrt 2)$ and unstable if $c\in(0,\tilde{c}_\ka)$.
\end{enumerate}

The idea relies on the general Grillakis--Shatah--Strauss theory \cite{grillakisshatah}, which  reduces the stability of solitons to  
the study of the second derivative of the action $d(c)= E_\kappa(u_{c,\kappa})-c p(u_{c,\kappa}),$
in addition to some spectral conditions. From \eqref{eq:hamiltongp}, we get 
$d''(c)=-\frac{d}{dc}p(u_{c,\kappa})$, so that  $d''(c)>0$  is equivalent to the  so-called Vakhitov--Kolokolov stability criterion: 
$\frac{d}{dc}p(u_{c,\kappa})<0.$
In view of \eqref{def:tildec}, this criterion explains the result in \hyperref[item:stabD2]{b}.

The proofs in \cite{lequiniou2024stability} use the explicit formulas for the momentum in Section~\ref{sec:energymoment}, and the arguments introduced in  \cite{benzoniStab,audiard2017}  in the context of the Euler--Korteweg system. We refer to \cite{lequiniou2024stability} for more details.

To conclude this introduction, we provide two curves of the energy of traveling waves in $\boD_1$ and $\boD_3$, using the formulas in Section~\ref{sec:energymoment}.
Figure~\ref{fig:Energicompa} displays on the same plot the energy of dark solitons $E_\ka(c)$ (in black) and the energy of cuspons $\tilde E_\ka(c)$ (in orange) given by \eqref{def:Eck}--\eqref{def:tildeEck}. 
 The left panel corresponds dark solitons and antidark cuspons in $\boD_1$, with  $\kappa=0.4$, and the right panel illustrates the behavior
of antidark solitons and dark cuspons  in $\boD_3$, with  $\kappa=0.6$.
We see that in both cases the least energy solution does not remain on the same branch of solutions as $c$ varies. In view of \eqref{eq:hamiltongp}, the energy and the momentum have the same monotony, as functions of $c$. 
Thus, as stated in (\hyperref[item:stabD1]{a}), the whole branch of dark solitons in $\boD_1$ is stable, which is illustrated by the decrease of the energy in the left panel. The stability or instability of the other branches
is an open problem.

\begin{figure}[ht!]
\centering
 \begin{tabular}[b]{cc}
\resizebox{0.4\textwidth}{!}{
      \begin{overpic}
				[scale=0.5,trim=0 0 0 50,clip]{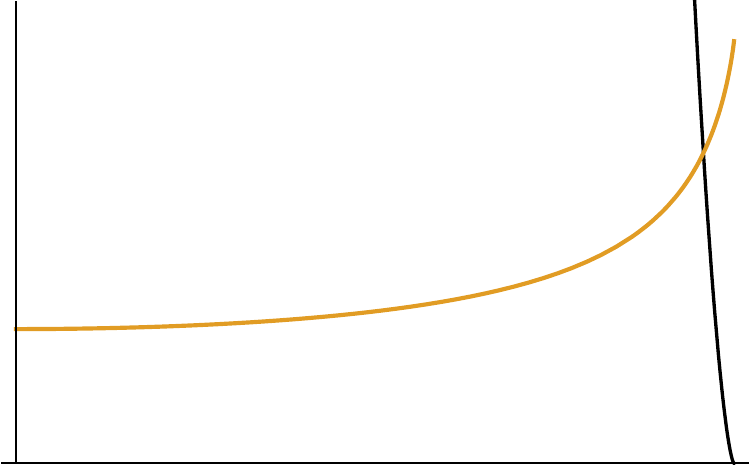}
   \put(98,0){\line(0,1){2}}
    \put(94,-6){$\sqrt{2}$}	 
     \put(101,0){$c$}

      \put(-12,16){$0.005$}
       \put(-12,48){$0.017$}
			\end{overpic}
}  &
  \qquad \
\resizebox{0.4\textwidth}{!}{
      \begin{overpic}
				[scale=0.5,trim=0 0 0 0,clip]{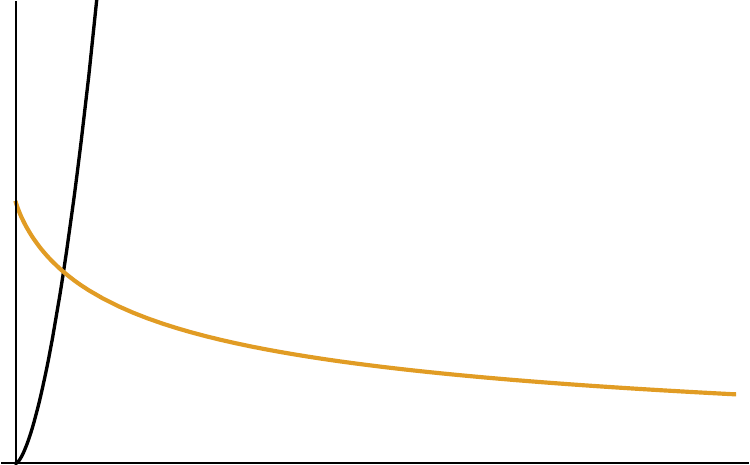}
				\put(101,0){$c$}
     \put(-12,31){$0.008$}
       \put(-12,60){$0.015$}
			\end{overpic}
} 
\end{tabular}
\caption{Plot of $E_\ka(c)$ in black and $\tilde{E}_\ka(c)$ in orange. On the left panel, we take parameters in $\D_1$,
with $\ka=0.4$ and $0\leq c<\sqrt{2}$. On the right panel,  the parameters belong to $\D_3$, with $\ka=0.6$ and $\sqrt{2}<c\leq 2$.}
\label{fig:Energicompa}
  \end{figure}




%
The outline of this paper is the following.  In Section~\ref{sec:properties}, 
we deduce the ODEs for the intensity profile.
Section~\ref{sec:construction:smooth} 
is devoted to the classification and construction of smooth solutions for the ODEs, while the analysis of weak solutions is done in Section~\ref{sec:construction:weak}.
In Section~\ref{sec:energymoment}, we provide some formulas for the energy and momentum 
of traveling waves. We study the minimization of the energy at fixed momentum in Section~\ref{sec:minimization}.
 in Section~\ref{section:cauchy}, we show the local well-posedness of \eqref{QGP} and the stability of dark solitons.
\paragraph{Notations.}
		The usual Lebesgue and Sobolev spaces of real-valued functions will be denoted, respectively, by $L^p(\R)$ and $W^{k,p}(\R)$, for $p\in [1,\infty]$ and $k\in\N$. Moreover, $W^{k,2}(\R)=H^k(\R)$. If $\Omega\subset\R$ is an open interval, then $H^1_0(\Omega)$  denotes the closure of $\boC_0^\infty(\Omega)$ in $H^1(\Omega)$. The notation for the Lebesgue spaces of complex-valued functions will be $L^p(\R;\C)$, and analogously for the Sobolev spaces of complex-valued functions, or simply $L^p(\R)$, if there is no ambiguity. For $k\geq1$, we introduce the homogeneous space $\dot{H}^{k}(\Omega)=\{u\in L^1_{\loc}(\R):u'\in H^{k-1}(\Omega)\}$ where $L^1_{\loc}(\R)$ is the Lebesgue space of functions integrable on every compact subset of $\R$. 
  Given a function $f$, $f(a^+)$ and $f(a^-)$ denote 
   the lateral limits of $f(x)$, as $x\to a^+$ and 
$x\to a^-$, respectively. We denote by  $\inner{~}{~}$ the real scalar product on $\C$:
		\[\inner{z_1}{z_2}=\Re(z_1\bar{z_2}).\]
For a complex-valued function $u$ (typically a function in $\boX(\R)$), 
 we define its intensity profile as the real-valued function $\eta=\eta_u=1-|u|^2$.
  In addition, we use the notation $u_{c,\ka}$  for a solution to \eqref{TWc} and 
 $\eta_{c,\ka}$ for its intensity profile, but 
  we will remove the subscripts ${c,\ka}$ in the proofs, when there is no ambiguity.

\section{Equations for the intensity profile}

\label{sec:properties}
We start by recalling properties satisfied by functions in $\boX(\R)$.
		\begin{lemma}
			\label{lem:finiteenergyassum}
			Let $u\in\boX(\R)$. Then $u$ belongs to $\boC^{1/2}(\R)\cap L^\infty(\R)$, and its intensity profile
   $\eta_u=1-|u|^2$ lies in $H^1(\R)$, with,
			\begin{equation}
				\label{eq:limeta}
	\lim_{\abs{x}\to \infty}\eta_u=0.
			\end{equation} 			
		\end{lemma}
		\begin{proof}
The fact that $u$ is bounded and 1/2-H\"older continuous follows from the Morrey inequality. Then, using that $\eta_u'=-2\inner{u'}{u}$, we conclude that  $\eta_u'\in L^2(\R)$, so that $\eta_u\in H^1(\R)$ and therefore  \eqref{eq:limeta}  holds.
		\end{proof}
		
For notational simplicity, we omit from now on the subscript $u$ in the profile intensity $\eta_u$, if there is no ambiguity.

For the sake of completeness, we also recall the following result concerning the lifting of smooth functions.
\begin{lemma}\label{lem:phaseregu}
If $u\in\boC^k(\R;\C)$, for some $k\in\N^*$,  satisfies that $u(x)\ne0$, for all $x\in\R$, then there exists a phase $\theta\in\boC^k(\R,\R)$, such that the lifting $u=\rho e^{i\theta}$ holds in $\R$, where 
$\rho\in\boC^k(\R;\R)$ is given by $\rho={\abs{u}}$.
Similarly, if $u\in\boN\boX(\R)$,  then there is $\theta\in\boC(\R;\R)\cap \dot{H}^1(\R)$ such that  $u=\rho e^{i\theta}$ in $\R$, and also 
$\rho=\abs{u}\in\boC(\R;\R)\cap \dot{H}^1(\R)$.

\end{lemma} 
\begin{proof}
If $u\in\boC^k(\R;\C)$ does not vanish, it is immediate that  $\rho=\sqrt{u\bar{u}}$ is of class $\boC^k.$ Now, let $w(x)=u(x)/\abs{u(x)}$ and $\tilde {\theta}=\int_0^x-iw'(s)\bar{w}(s)ds=\int_0^x-\inner{iu'(s)}{u(s)}/|u|^2(s)ds$, then we have $
	(w(x)e^{-i\tilde{\theta}(x)})'=e^{-i\tilde{\theta}}(w'-ww'\bar{w})=0,$
	since $w\bar{w}\equiv1$. Hence, $w(x)=w(0)e^{i\tilde{\theta}(x)}$, and we the conclusion follows by setting $\theta(x)=\tilde{\theta}(x)+\mathop{\mathrm{arg}}(w(0))$, where $\mathop{\mathrm{arg}}$ is any continuous determination of the argument defined near $w(0)$.

 The same argument holds for $u\in\boN\boX(\R)$, showing that $\theta$ is continuous. Also, since $u\in\boX(\R)$, we have
			\begin{equation}\label{eq:uprimpolaire}
				u'=\rho'e^{i\theta}+i\rho\theta'e^{i\theta}\in L^2(\R),
			\end{equation} 
			which implies that  $\rho', \rho\theta'\in L^2(\R)$. Since $\inf_{\R} \rho>0$, we conclude that $\theta'\in L^2(\R)$.
 \end{proof}
 
For a  nonvanishing solution $u$ to \eqref{TWc},  the lifting  $u=\rho e^{i\theta}$  is commonly used to derive a system of equations satisfied by $\rho$ and $\eta$, which are related to the hydrodynamical formulation for the
Gross--Pitaevskii equation \cite{CarlesDanSaut}. The equivalence between the hydrodynamical formulation and the traveling wave equation has been obtained for nonlocal Gross--Pitaevskii equations in  \cite{dLMar2022}. However, they use that the solutions are smooth. We verify now that this hydrodynamical formulation also holds for weak solutions to \eqref{TWc}.

    	\begin{corollary}
			\label{coro:eqpolaire}
			Let $u=\rho e^{i\theta}\in\boN\boX(\R)$, with  $\rho,\theta \in \boC(\R)\cap \dot{H}^1(\R)$ be weak solution to  \eqref{TWc}.
   Then $\theta'\in H^1(\R)$ and $(\rho,\theta)$ satisfies the following system:
			\begin{align}\label{TW:weaksystpol1}
				\theta'&=\frac{c(1-\rho^2)}{2\rho^2}.\\
				\label{TW:weaksystpol2}
				\Big(\rho'(1-2\ka\rho^2)\Big)'&=-2\ka\rho(\rho')^2-\rho(1-\rho^2)+\frac{c^2(1-\rho^2)^2}{4\rho^3}+\frac{c^2(1-\rho^2)}{2\rho}.
			\end{align}

Conversely, if there exists a function $\rho\in \boC(\R)$ such that $\inf_\R \rho>0$ and $1-\rho^2\in H^1(\R)$,  satisfying \eqref{TW:weaksystpol2},  then there exists a unique (up to a constant) function  $\theta\in\boC^1(\R)\cap \dot{H}^2(\R)$ satisfying \eqref{TW:weaksystpol1}. Moreover, the function defined by $u=\rho e^{i\theta}$ belongs to  $\boN\boX(\R)$ and is a solution to \eqref{TWc}. 
		\end{corollary}
		\begin{proof}
	Let $u=\rho e^{i\theta}\in\boN\boX(\R)$.	To show \eqref{TW:weaksystpol1}--\eqref{TW:weaksystpol2}, we take $\varphi\in H^1(\R)$ and write \eqref{TW:weak} with $\phi=e^{i\theta}\varphi$.
			We have $\phi\in H^1(\R)$ and the equation reads, using that $\inner{z_1e^{i\theta}}{z_2e^{i\theta}}=\inner{z_1}{z_2}$, for all $z_1,$ $z_2\in\C$ and all $\theta\in\R$,
			\begin{align}\label{TW:weakpol}
				\int_\R\inner{ic\rho'-c\theta'\rho+\rho(1-\rho^2)}{\varphi}-\inner{\rho'+i\rho\theta'}{i\theta'\varphi+\varphi'}+2\ka\inner{\rho}{\rho'+i\theta'\rho}\inner{\rho}{\varphi}'=0.
			\end{align}
			Taking $i\varphi_1$ and then $\varphi_1$ for any $\varphi_1\in H^1(\R;\R)$ in place of $\varphi$ in \eqref{TW:weakpol}, we deduce:
			\begin{align}\label{TW:weaksystpol0}
				&	\int_\R c\rho'\varphi_1+\rho'\theta'\varphi_1-\theta'\rho\varphi_1'=0,\\
				\label{TW:weaksystpol05}
				&\int_\R(-c\theta'\rho+\rho(1-\rho^2))\varphi_1-(\theta')^2\rho\varphi_1-\rho'\varphi_1'+2\ka\rho\rho'(\rho\varphi_1)'=0.
			\end{align}
			Setting $\varphi_1=\rho\varphi_2\in H^1(\R;\R)$ in \eqref{TW:weaksystpol0}, we obtain  $\int_\R-c(1-\rho^2)'\varphi_2/2-\theta'\rho^2\varphi_2'=0$, it follows, by integration by parts of the first term, that $(-c(1-\rho^2)/2+\theta'\rho^2)'=0$ in the distributional sense. We conclude from the integrability at infinity  of  $1-\rho^2$ and  $\theta'^2$ that equation \eqref{TW:weaksystpol1} is satisfied in $L^2(\R)$, which in turn implies that $\theta'\in H^1(\R)$. In particular, \eqref{TW:weaksystpol1} is satisfied pointwisely. To deduce equation \eqref{TW:weaksystpol2}, we just replace $\theta'$ by $c(1-\rho^2)/(2\rho^2)$ in \eqref{TW:weaksystpol05}.

			We now prove the converse; it follows from the assumptions that $\rho$ and $1/\rho$ are essentially bounded, with weak derivatives in $L^2(\R)$. Similarly to what has been done to obtain \eqref{TW:weaksystpol1}--\eqref{TW:weaksystpol2} from  \eqref{TWc}, choosing appropriate test functions in $H^1(\R)$, we deduce that $u=\rho e^{i\theta}$ belongs to $\boN\boX(\R)$ and satisfies equation \eqref{TWc} in the weak sense.
		\end{proof}
		\begin{remark}\label{rem:localeq}
We also deduce from the proof of Corollary~\ref{coro:eqpolaire} that  if $u\in \boX(\R)$ is a solution to  \eqref{TWc}, with 
$\inf_{[a,b]}\abs{u}>0$, for some interval $(a,b)$, so that  $u=\rho e^{i\theta}$ in $[a,b]$, then 	 there exists a constant $K\in\R$ such that 
\begin{equation}
   \theta'=\frac{c(1-\rho^2)}{2\rho^2}+K, \quad  \text{ in }(a,b).
\end{equation}
		\end{remark}
The next result shows that equation \eqref{TWc} can be recast as two equations for the (real-valued) intensity profile $\eta$, which is key for our classification results, in the same spirit of  \cite{bethuel2008existence,dLMar2022,deLaire4,deLaireLopez2}.
  \begin{proposition}
  \label{prop:eqeta}
		Let $u_{c,\ka}\in\boC^2(\R)\cap \boX(\R)$ be a solution  to \eqref{TWc}. Then $\eta_{c,\ka}=1-|u_{c,\ka}|^2,$ satisfies 
			\begin{align}
				\label{eta2}
    			(1- 
 2\ka+2\ka\eta_{c,\ka})\eta_{c,\ka}''+\ka(\eta_{c,\ka}')^2=-3\eta_{c,\ka}^2+(2-c^2)\eta_{c,\ka}, \quad \text{ in }\R.\\
				\label{eta1}
				(1-2\ka+2\ka\eta_{c,\ka})(\eta_{c,\ka}')^2=\eta_{c,\ka}^2(2-c^2 -2\eta_{c,\ka}), \quad \text{ in }\R.
			\end{align}
			In particular, assuming without loss of generality that 
			$\abs{\eta_{c,\ka}}$  reaches a global maximum at the origin, we obtain either
			\begin{align}
				\label{eq:maxeta}
				\eta(x)=0,\text{ for all }x\in\R,\quad\text{ or }\quad\eta(0)=1-\frac{c^2}{2}.
			\end{align}
		\end{proposition}
		\begin{proof}
			Let $u=u_1+iu_2$, writing the equations satisfied by $u_1$ and $u_2$, we obtain
			\begin{equation}\label{eq:Reu}
				u_1''-cu_2'+u_1(\eta+\ka\eta'')=0,\quad \text{in }\R,
			\end{equation}
			\begin{equation}\label{eq:Imu}
				u_2''+cu_1'+u_2(\eta+\ka\eta'')=0,\quad \text{in }\R.
			\end{equation}
			Multiplying \eqref{eq:Reu} by $-u_2$, and \eqref{eq:Imu} by $u_1$, and adding these equations, we get
			\begin{align*}
				(u_1u_2'-u_1'u_2)'=\frac{c}{2}\eta'.
			\end{align*}
			Since $u'\in L^2(\R)\cap\boC(\R)$, there exists a sequence $(R_n)_{n\in\N}$ such that $\lim_{n\rightarrow \infty}R_n=\infty$ and $u'(R_n)=u_1'(R_n)+iu_2'(R_n)\rightarrow0\text{ as }n\rightarrow \infty.$
			Integrating from $x\in\R$ to $R_n$, we obtain 
			\begin{equation}\label{eq:phase1}
				(u_1u_2'-u_1'u_2)(x)=\frac{c}{2}\eta(x)+K_n,\quad\text{ for all }x\in\R,
			\end{equation}
			where $K_n= (u_1u_2'-u_1'u_2)(R_n)-\frac{c}{2}\eta(R_n)\rightarrow0$, as $n\rightarrow \infty$.
			Thus equation \eqref{eq:phase1} reads
			\begin{equation}\label{eq:phase2}
				(u_1u_2'-u_1'u_2)=\frac{c}{2}\eta,\quad\text{in }\R.
			\end{equation}
			On the other hand, multiplying \eqref{eq:Reu} by $u_1'$ and \eqref{eq:Imu} by $u_2'$, and adding these equations, we have
			\begin{equation*}
				\frac{1}{2}\Big((u_1')^2+(u_2')^2\Big)'=\frac{\eta'}{2}(\eta+\ka\eta''),
			\end{equation*}
			so integrating this relation from $x$ to $R_n$, and taking the limit as before,  we obtain
			\begin{align}\label{eq:quadratic}
				\abs{u'}^2&=\frac{1}{2}(\eta^2+\ka(\eta')^2).
			\end{align}
			In addition, multiplying \eqref{eq:Reu} by $u_1$, \eqref{eq:Imu} by $u_2$ and adding these equations, we have
			\begin{align*}
				c(u_1u_2'-u_1'u_2)=u_1u_1''+u_2u_2''+|u|^2(\eta+\ka\eta'').
			\end{align*}
		We are now in a position to deduce \eqref{eta2}. Indeed, since
			$\eta''=-2(|u'|^2+u_1u_1''+u_2u_2'')$, using \eqref{eq:phase2}--\eqref{eq:quadratic}, we get
			\begin{align*}
				\eta''&=-(\eta^2+\ka(\eta')^2)-{c^2}\eta+2\abs{u}^2(\eta+\ka\eta''),
			\end{align*}
			which is exactly \eqref{eta2}.
			To obtain \eqref{eta1}, we notice  that \begin{align}\label{eq:eta21}
				\Big(\frac{(\eta')^2}{2}(1-2\ka+2\ka\eta)\Big)'=\eta''\eta'(1-2\ka+2\ka\eta)+\ka(\eta')^3,
			\end{align}
			so that  multiplying \eqref{eta2} by $\eta'$, integrating from $x$ to $R_n$ and taking the limit as $n\to\infty$, we finally deduce \eqref{eta1}. Since $\eta$ satisfies \eqref{eq:limeta}, it must reach a global extremum at some $x_0\in\R$. We fix $x_0=0$ and evaluate \eqref{eta1} at $0$ to obtain \eqref{eq:maxeta}.
		\end{proof}
		\begin{remark}\label{rem:etaeq}
			Notice that if $\eta\in\boC^2(\R)$ satisfies \eqref{eta1}, and if $\eta'$ only vanishes on a set of measure zero, then $\eta$ satisfies \eqref{eta2}. Conversely, if $\eta\in\boC^2(\R)$ satisfies \eqref{eta2} and if $\eta\in H^1(\R)$, then $\eta$ satisfies \eqref{eta1}.  This remark motivates our choice to study solutions to both \eqref{eta2}--\eqref{eta1} to shorten the statements of the results.
   \end{remark}
		Equation \eqref{eta2}--\eqref{eta1} provides a simpler formulation to \eqref{TWc} given by two real-valued
		equations. Combining Proposition~\ref{prop:eqeta} with the following result, we deduce the equivalence of these problems for regular solutions with nonzero speed $c$. 
		
		\begin{proposition}\label{prop:exiu}
			If $c>0$ and $\eta_{c,\ka}\in \boC^2(\R)$ is a solution to  \eqref{eta2}--\eqref{eta1}, then, $\eta_{c,\ka}<1$ on $\R$, and the function  
			\begin{equation}\label{eq:deftheta}
				\theta_{c,\ka}(x)=\frac{c}{2}\int_a^x\Big(\frac{\eta_{c,\ka}(y)}{1-\eta_{c,\ka}(y)}\Big)dy,\quad\text{ for all } x\in\R,
			\end{equation} is well-defined in $\R$, for all $a\in\R$.
			Moreover, the function $u_{c,\ka}=\sqrt{1-\eta_{c,\ka}} e^{i\theta_{c,\ka}}$ belongs to $\boC^2(\R;\C)$ and satisfies \eqref{TWc}. Also, if $\eta_{c,\ka}\in  H^1(\R)$, then $u_{c,\ka}\in\boN\boX(\R)$.
		\end{proposition}
		\begin{proof}
			Let $u=\rho e^{i\theta}$, with $\rho=\sqrt{1-\eta}$,  we show first that $\rho(x)>0,$ for all $x\in\R$. Suppose that $\rho(x_0)=0$, for some $x_0\in\R$, then $\eta(x_0)=1$ and is necessarily a global maximum of $\eta$. Equation \eqref{eq:maxeta} implies that $c=0$, contradicting our assumptions.
				We conclude that $\theta$ is well-defined and belongs to $\boC^3(\R)$.	In view of \eqref{eq:deftheta}, we have
			\begin{equation}\label{eq:uprim mod}
				u'(x)=-\frac{e^{i\theta(x)}\eta'(x)}{2\sqrt{1-\eta(x)}}+ie^{i\theta(x)}\frac{c\eta(x)}{2\sqrt{1-\eta(x)}}.
			\end{equation} Computing the left-hand side of \eqref{TWc} using \eqref{eq:uprim mod}, one can check that the imaginary part is zero. Using \eqref{eta2}--\eqref{eta1}, computations yield that the real part is zero too, thus $u$  satisfies \eqref{TWc}. Now suppose additionally that $\eta\in H^1(\R)$. Using \eqref{eq:uprim mod} we get \begin{equation}\label{eq:L2 uprim polar}
				\int_\R |u'|^2=\int_\R \frac{\eta'^2}{4(1-\eta)}+\frac{c^2\eta^2}{4(1-\eta)},
			\end{equation}
			Then $|u'|^2$ is integrable, since $\eta\in H^1(\R)$ and $\inf_\R(1-\eta)>0$. Therefore $u\in\boN\boX(\R)$.
		\end{proof}
		\begin{remark}\label{rem:equiv}
			
   If $u=\sqrt{1-\eta}e^{i\theta}\in\boN\boX(\R)$ is a solution to \eqref{TWc}, 
     then the formulas for  $E_\ka(u)$ and $p(u)$ in \eqref{eq:Eupol} and \eqref{def:mom2},
can be simplified as    
			\begin{equation}
				\label{eq:Euregu}E_\ka(u)=\frac{1}{2}\int_\R\eta^2\quad 		\text{ and }\quad
				p(u)=\frac{c}{4}\int_\R\frac{\eta^2}{1-\eta},
			\end{equation}
		by using respectively \eqref{eq:quadratic} and \eqref{TW:weaksystpol1}. 
		\end{remark}
\section{Classification of smooth traveling waves}
\label{sec:construction:smooth}
		
Equation \eqref{eta1}  yields necessary conditions on $(c,\ka)$ for the existence of nontrivial traveling waves. For instance, 
	taking $x$ large enough, we expect that 
		\begin{equation}
			\label{condition}
			0<(2-c^2)(1-2\ka), \text{ or equivalently }(c,\ka)\in \D, 
		\end{equation}
		is a necessary condition for the existence of nontrivial solutions.
	 Corollary~\ref{coro:crittrivial} and Proposition~\ref{prop:k1/2} aim to provide a rigorous proof of this fact.
		\begin{corollary}
			\label{coro:crittrivial}	
Let $c\geq 0$ and $\ka\in \R$, with $\ka\ne1/2$. Assume that $(c,\kappa)\notin \D$ and that  $u\in\boC^2(\R)\cap\boX(\R)$ is a solution to \eqref{TWc}.
Then $u$ is a trivial solution, i.e., there is a constant $\phi \in \R$ such that 
$u(x)=e^{i\phi}$, for all $x\in \R$.
	\end{corollary}
\begin{proof}
Let $\eta\in H^1(\R)$ be the solution to \eqref{eta2}--\eqref{eta1} given by Proposition \ref{prop:eqeta}. We distinguish three cases for the parameters:
(i) $c=\sqrt{2}$ and $\ka\in\R$, (ii) $c>\sqrt{2}$ and $\ka<1/2$, (iii) $0\leq c<\sqrt{2}$ and $\ka>1/2$. We show now that in each case we can conclude that $\eta\equiv0$  in $\R$.

If (i) is satisfied, then \eqref{eq:maxeta} implies that $0$ is the only possible global extremum of $\eta$, thus $\eta\equiv0$.
In the case (ii), suppose by contradiction that $\eta$ is not the zero function, then, using \eqref{eq:maxeta},  $\eta(0)=1-c^2/2<0$. By \eqref{eq:limeta} and,  using the intermediate value theorem, we infer that there exists $x_1\in\R$ such that \[\eta(x_1)=\mathrm{max}\Big\{\frac{2-c^2}{4},\frac{2\ka-1}{4\ka}\Big\} \text{ if }\ka>0,\text{ and }\eta(x_1)=\frac{2-c^2}{4} \text{ if }\ka\leq0.\]
Hence $\eta(x_1)^2(2-c^2-2\eta(x_1))<0$ and $1-2\ka+2\ka\eta(x_1)>0$. Therefore, computing the sign of both sides of equation \eqref{eta1}, we obtain a contradiction. We conclude that $\eta\equiv0$.
	Case (iii) can be treated analogously to case (ii), by taking $\eta(x_1)=\mathrm{min}\Big\{\frac{2-c^2}{4},\frac{2\ka-1}{4\ka}\Big\}.$
\end{proof}
Corollary~\ref{coro:crittrivial} provides a first nonexistence result for nontrivial finite energy solutions to \eqref{TWc}, where the case $\ka=1/2$ was excluded.
We will handle the case $\ka=1/2$, finding the explicit smooth solutions to \eqref{TWc}, and checking that this solution does not belong to the energy space. This explicit solution will also enable us to construct the dark compactons in Subsection~\ref{subsec:buildw}.

\begin{proposition}\label{prop:k1/2}
	Let $\ka=1/2$ and $c\geq0$. Assume that $\eta \in C^2(\R)$ is a  nonzero solution to  \eqref{eta2}--\eqref{eta1}, satisfying \eqref{eq:maxeta}. Then for all $x\in\R$, we have
\begin{equation}\label{eq:formuleeta05}
	\eta(x)=\frac{(2-c^2)}{2}\cos^2\Big(\frac{x}{\sqrt{2}}\Big).
\end{equation}
In particular, there is no nontrivial  solution to \eqref{TWc}
with $\ka=1/2$ in $\boC^2(\R)\cap\boX(\R)$. 
\end{proposition}
\begin{proof}
	If $c=\sqrt{2}$, then \eqref{eq:maxeta} shows that 
	$\eta\equiv 0$ is the only solution, so \eqref{eq:formuleeta05} is trivially satisfied.

Let us now verify formula \eqref{eq:formuleeta05} in the case $0\leq c<\sqrt{2}$. Notice first that \eqref{eta2} reads 
	\begin{align}
	\label{eta3}
	\eta\eta''+(\eta')^2/2=-3\eta^2+(2-c^2)\eta, \quad \text{ in }\R.
\end{align}
By \eqref{eq:maxeta}, we can assume that $\abs \eta$ reaches a global maximum at $x=0$ with $\eta(0)=1-c^2/2$,
hence $\eta'(0)=0$ and \eqref{eta3} yields
 $\eta''(0)<0$. Since $\eta$ is strictly concave at $x=0$,  $\eta(0)$ is the global maximum of $\eta$ and letting $R=\sup\{x>0 : \eta(s)>0,~\text{ for all } s\in(0,x)\}$, we infer that $\eta'<0$ on $(0,R)$. Indeed, otherwise setting  $x_1=\inf\{x>0,~\eta'(x)=0\},$
 we would have $x_1<R$ and $0<x_1$ by concavity. Then, evaluating \eqref{eta1} at $x_1$ yields $\eta(x_1)=1-c^2/2$ since $\eta(x_1)>0$ by definition of $R$.  Using Rolle's theorem, the latter identity contradicts the minimality of $x_1$. We deduce  again from \eqref{eta1} that 
	\begin{align}\label{eq:impli05}
		\frac{\eta'}{\sqrt{\eta(2-c^2-2\eta)}}=-1,\quad\text{ in }(0,R).
	\end{align}
By direct integration, we obtain $R=\pi/\sqrt{2}$ and
	\begin{equation}
		\label{eq:eta05}
		\eta(x)=\frac{2-c^2}{2\tan^2\big(x/\sqrt{2}\big)+2}=\frac{(2-c^2)}{2}\cos^2
			\big(x/\sqrt{2}\big),\quad x\in\Big[0,\frac{\pi}{\sqrt{2}}\Big).
	\end{equation}
The same argument shows that the formula \eqref{eq:eta05} remains valid for $-{\pi}/{\sqrt{2}}<x<0$, so \eqref{eq:formuleeta05} holds on  $(-{\pi}/{\sqrt{2}},{\pi}/{\sqrt{2}})$. It follows that, 
\begin{align}\label{eq:etap05}
	\eta'(x)&=-\frac{2-c^2}{\sqrt{2}}\sin\Big(\frac{x}{\sqrt{2}}\Big)\cos\Big(\frac{x}{\sqrt{2}}\Big),\quad\text{ for all }x\in\Big(-\frac{\pi}{\sqrt{2}},\frac{\pi}{\sqrt{2}}\Big),
\\
\label{eq:etapp05}
	\eta''(x)&=-\frac{2-c^2}{2}\Big(\cos^2\Big(\frac{x}{\sqrt{2}}\Big)-\sin^2\Big(\frac{x}{\sqrt{2}}\Big)\Big),\quad\text{ for all }x\in\Big(-\frac{\pi}{\sqrt{2}},\frac{\pi}{\sqrt{2}}\Big).
\end{align}
Since $\eta\in \boC^2(\R)$, we conclude from these explicit formulas that 
 $\eta(\pm\pi/\sqrt{2})=0$,  $\eta'(\pm\pi/\sqrt{2})=0$, and  $\eta''(\pm\pi/\sqrt{2})=(2-c^2)/2.$
Because $\eta$ is strictly convex at $x=-\pi/\sqrt{2}$, letting $$R_2=\inf\{x<-\pi/\sqrt{2}:\eta(s)<1-c^2/2, \text{ for all }s\in(x,-\pi/\sqrt{2})\},$$ we deduce that \eqref{eq:impli05} still holds replacing $(0,R)$ by $(R_2,-\pi/\sqrt{2})$,  arguing as before.
Integrating \eqref{eq:impli05}, we obtain $R_2=-\sqrt{2}\pi$, and \eqref{eq:eta05} remains true in $(-\sqrt{2}\pi,-\pi/\sqrt{2}]$. From the same arguments and an induction procedure, we infer that $\eta$ satisfies \eqref{eq:formuleeta05}. 
In the case $c>\sqrt{2}$, similar ideas allow us to deduce that \eqref{eq:formuleeta05} also holds.

Finally, suppose by contradiction that $u\in\boC^2(\R)\cap\boX(\R)$ satisfies \eqref{TWc} with $\ka=1/2$ and is nontrivial. Then $\eta=1-\abs{u}^2$ satisfies \eqref{eta2}--\eqref{eta1} and \eqref{eq:maxeta}, by Proposition~\ref{prop:eqeta}. Hence \eqref{eq:formuleeta05} holds, and therefore $\eta\notin L^2(\R)$, contradicting  $u\in\boX(\R)$. Additionally, we can check that \eqref{eq:Euregu} also holds in that case, so that we have $E_\ka(u)= \infty$. 
\end{proof}

\begin{remark}\label{rem:sol05}
	Using \eqref{eq:formuleeta05} and \eqref{eq:deftheta}, we deduce that the only nontrivial $\boC^2(\R)$-solution $u$ to (TW$(c,1/2))$ such that $\eta=1-\abs{u}^2$ satisfies \eqref{eq:maxeta} is given, up to invariances and for any $c>0$ by
	$$u(x)=\sqrt{{1-\frac{(2-c^2)}{2}\cos^2\Big(\frac{x}{\sqrt{2}}\Big)}}e^{i\theta(x)},
	$$
	where $\theta$ is a smooth odd function satisfying, for all $x\geq 0$  and all $k\in \N$,
	$$
	\theta(x)=-\frac{cx}2  -
	\Bigg(
	\atan\Big( \frac c{\sqrt 2}\cot\Big(\frac x{\sqrt 2}\Big) \Big)
	-\frac{\pi}2-k\pi
	\Bigg)	 , \quad \text { if } x\in \big(k\sqrt 2\pi, (k+1)\sqrt 2\pi\big).
	$$
In the case $c=0$, this solution degenerates to the periodic  solution to (TW$(0,1/2))$ given by
$u(x)=i\sin(x/\sqrt{2})$.
\end{remark}
  We continue the study of the equations \eqref{eta2}--\eqref{eta1} in Proposition~\ref{prop:eqeta} when the parameters $(c,\ka)$ belong to $\D$, using the regions $\D_1$, $\D_2$ and $\D_3$, defined in\eqref{def:D1}--\eqref{def:D3}.

\begin{corollary}\label{coro:propeta}
	Let $(c,\ka)\in \D$ and assume that 
$\eta\in\boC^2(\R)$ is a nonzero solution to \eqref{eta2}--\eqref{eta1} satisfying \eqref{eq:limeta}. 
Then up to translation,  $\eta$ is even,  reaching a  global extremum at the origin with $\eta(0)=1-c^2/2$. 
Moreover,  we have $\eta>0$ and $\eta'<0$ on $(0,\infty)$ if $(c,\ka)\in \D_1\cup \D_2$, while  $\eta<0$ and $\eta'>0$ on $(0,\infty)$ if $(c,\ka)\in \D_3$.
Additionally, the function $\eta$ belongs to $\boC^\infty(\R)$ and  is exponentially decaying, as well as all its derivatives, i.e.\ for every $j\in\N$ there exists positive constants $C_0$ and $C$ such that   
\[\abs{D^j\eta(x)}\leq C_0e^{-C(\abs{x}-1)},\quad\text{ for all }\abs{x}\geq1.\]
\end{corollary}
\begin{proof}
We recall that $\D=\D_1\cup\D_2\cup \D_3$.	Let us treat case first the case $(c,\ka)\in \D_1\cup \D_2$. Using \eqref{eq:maxeta}, we assume that $\eta$ reaches a {\it positive} global  maximum $1-c^2/2$  at $x=0$.
	Notice that $\eta>0$, indeed, otherwise we would have $\eta(x_1)=0$, for some $x_1\in\R$. Then, equation \eqref{eta1} yields $\eta'(x_1)=0$. Using $\eqref{eta2}$, we infer that $(\eta,\eta')^T$ satisfies \begin{equation}
		\label{eq:odeetaord2}
		\begin{pmatrix}
			\eta\\\eta'
		\end{pmatrix}'(x)=\begin{pmatrix}
			\eta'(x)\\\frac{-3\eta^2(x)+(2-c^2)\eta(x)-\ka(\eta'(x))^2}{1-2\ka+2\ka\eta(x)}
		\end{pmatrix}=F((\eta,\eta')^T),
	\end{equation} in a neighborhood of $x_1$. Since $F$ is well-defined and locally Lipschitz in $(\R\backslash\{1-1/(2\ka)\})\times\R$, by Cauchy--Lipschitz theorem, for any initial condition in $(\R\backslash\{1-1/(2\ka)\})\times\R$, there is a unique maximal solution of \eqref{eq:odeetaord2}. In particular, since $(\eta(x_1),\eta'(x_1))^T=(0,0)^T$, is an equilibrium, we deduce that  $\eta\equiv0$ in $\R$, which is a contradiction. 
 
 The fact that $\eta$ is even also follows from the Cauchy--Lipschitz theorem.
 Indeed, setting $\tilde{\eta}(x)=\eta(-x)$, for all $x\in \R$,  we deduce  that 
$(\tilde{\eta},\tilde{\eta}')^T$ satisfies  \eqref{eq:odeetaord2} with initial condition $(\tilde{\eta}(0),\tilde{\eta}'(0))^T=(1-c^2/2,0)^T$. Thus $\eta_2=\eta$. 

We show now that $\eta'(x)<0$, for all $x\in(0,\infty)$. Evaluating equation \eqref{eta2} at $x=0$, yields  $\eta''(0)<0$, hence $\eta'$ is decreasing in the vicinity of $0$. This implies that $\eta'(x)<0$, for $x>0$ near $0$. Now suppose by contradiction that $\eta'(x_1)=0$, for some $x_1>0$, and let $x_2=\inf\{x>0 : \eta'(x)=0\}$. Since $\eta(0)=1-c^2/2$, we infer from Rolle's theorem that we cannot have $\eta(x_2)=1-c^2/2$. Hence, evaluating \eqref{eta1} at $x=x_2$, we obtain $\eta(x_2)=0$, which contradicts the positivity of $\eta$. Thus $\eta'$ does not change sign and remains negative in $(0,\infty)$. 

To prove the remainder properties, we use that $\eta>0$ on $\R$, thus 
 \[1-2\ka+2\ka\eta(x)>1-2\ka>0,\quad\text{ for all }x\in\R,\]
 and therefore $\eta$ remains in the domain of $F$, so it is a global solution to \eqref{eq:odeetaord2}. Thus, the smoothness of $\eta$ follows by an induction argument on the ODE \eqref{eq:odeetaord2}. For the decay estimates, since $\eta$ is decreasing in $(0,\infty)$, we get from \eqref{eta1}
\begin{align}
	\label{eq:odeeta1}
	\eta'(x)=-\eta(x)\sqrt{\frac{-2\eta(x)+(2-c^2)}{1-2\ka+2\ka \eta(x)}},\quad\text{ for all }x>0.
\end{align} 
Thus $\eta'\leq-C\eta(x)$, for all $x\geq1$, where $C=\inf_{x\geq1}\Big\{\sqrt{\frac{-2\eta(x)+(2-c^2)}{1-2\ka+2\ka \eta(x)}}\Big\}>0.$
Integrating this differential inequality, we get  $\eta(x)\leq\eta(1)e^{-C(x-1)}$. The decay of higher other derivatives is then obtained by differentiating \eqref{eq:odeeta1}, together with an induction argument.

Finally, the case $(c,\kappa)\in \D_3$ can be treated similarly,  noticing  that $\eta$ also satisfies \eqref{eq:odeeta1}, since  $\eta<0$ and $\eta'>0$ on $(0, \infty)$.
 \end{proof}
Now we prove the main result of this subsection: the existence and uniqueness (up to translation)  of $\eta_{c,\ka}\in\boC^2(\R)\cap H^1(\R)$ satisfying \eqref{eta2}--\eqref{eta1}, for any $(c,\ka)\in \D$. Then, in the case $c>0$, Theorem \ref{thm:classiftwregu} is deduced using the equivalence stated in Proposition \ref{prop:exiu}. 
We prove beforehand, that the functions \eqref{eq:FD1}--\eqref{eq:FD3} are well-defined, and injective so that we can define their inverse. The idea behind \eqref{eq:FD1}--\eqref{eq:FD3} is to integrate explicitly the 
ODE \eqref{eq:odeeta1} taking into account the properties of the solutions stated in Corollary~\ref{coro:propeta}.
\begin{lemma}\label{lem:etafunctions}
	The functions  $F_{c,\ka}$, $G_{c,\ka}$ and $H_{c,\ka}$ are well-defined. Moreover,\\
\textup{(i)}		 $F_{c,\ka}$ is decreasing, belongs to $\boC^\infty(\boI_c^\circ)\cap \boC(\boI_c)$, and 
		\begin{equation}\label{eq:impliFprim}
			\lim_{y\to 0^+}F_{c,\ka}(y)=\infty, \
			F_{c,\ka}(1-c^2/2)=0, \quad F_{c,\ka}'(y)=-\frac{1}{y}\sqrt{\frac{1-2\ka+2\ka y}{2-c^2-2y}},\ \text{ for }y \in \boI_c^\circ.
		\end{equation}
\textup{(ii)} $H_{c,\ka}$ is increasing, belongs to $\boC^\infty(\boJ_c^\circ)\cap \boC(\boJ_c)$, and 
		\begin{equation}\label{eq:impliHprim}
			H_{c,\ka}(1-c^2/2)=0,\  	\lim_{y\to 0^-}H_{c,\ka}(y) =\infty, 
			\quad H_{c,\ka}'(y)=-\frac{1}{y}\sqrt{\frac{1-2\ka+2\ka y}{2-c^2-2y}},\ \text{ for }y \in \boJ_c^\circ.
		\end{equation}
\textup{(iii)} $G_{c,\ka}$ is decreasing, belongs to $\boC^\infty(\boI_c^\circ)\cap \boC(\boI_c)$, and 
		\begin{equation}\label{eq:impliGprim}
			\lim_{y\to 0^+}G_{c,\ka}(y)=\infty, \
			G_{c,\ka}(1-c^2/2)=0, \quad G_{c,\ka}'(y)=-\frac{1}{y}\sqrt{\frac{1-2\ka+2\ka y}{2-c^2-2y}},\ \text{ for }y \in \boI_c^\circ.
	\end{equation}
\end{lemma}
\begin{proof}
		 It is immediate to check that 
	\begin{gather}
		\label{denom1}
		(\sigma,(c,\ka)) \in \boI_c\times \D_1
		\Longrightarrow 1-2\ka+2\ka \sigma \geq 1-2\ka>0,\\
		\label{denom2}
		(\sigma,(c,\ka)) \in \boI_c\times  \D_2
		\Longrightarrow 1-2\ka+2\ka \sigma \geq 1-\ka c^2>0,
		\\
		\label{denom3}
		(\sigma,(c,\ka)) \in \boJ_c\times \D_3 \
		\Longrightarrow 1-2\ka+2\ka \sigma < 1-2\ka<0,
	\end{gather}
	so the square roots in \eqref{eq:FD1}--\eqref{eq:FD2} and in \eqref{eq:impliFprim}--\eqref{eq:impliGprim} are well-defined.	Since the domain of definition of $\atanh$ is $(-1,1)$, in case (i),
	it remains to check that  $(1-2\ka)(2-c^2-2y)<(2-c^2)(1-2\ka+2\ka y)$, which is equivalent to the condition 
	\begin{equation}
		\label{cond:atanh}
		y(c^2\ka-1)<0,
	\end{equation} which is satisfied in this case, since $y>0$. Thus, $F_{c,\ka}$ belongs to $\boC^\infty(\boI_c^\circ)\cap \boC(\boI_c)$, and 
	one gets immediately the values of $F_{c,\ka}$ at $y=0^+$ and  $y=1-c^2/2$ in \eqref{eq:impliFprim}.
	Finally, simple computations give
	$$\frac{d}{d y}
	\Bigg( \atan\Big(\sqrt{\ka}\sqrt{\frac{ 2-c^2-2y}{1-2 \ka +2\ka y}}\Big)
	\Bigg)=-\frac{\sqrt{\ka}}{\sqrt{(2-c^2-2y)(1-2 \ka +2\ka y)}},
	$$
	and 
	$$\frac{d}{d y}
	\Bigg(
	\atanh\Big(\sqrt{\frac{(1-2 \ka) (2-c^2-2y)}{(2-c^2)(1-2\ka+2\ka y)}}\Big)
	\Bigg)=
	-\frac{\sqrt{(1-2\ka)(2-c^2)}}{2y\sqrt{(2-c^2-2y)(1-2 \ka +2\ka y)}}, 
	$$
	so that we obtain the derivative of $F_{c,\ka}$ in \eqref{eq:impliFprim}.
	
	The case (ii) follows using the same computations, noticing that 
 \eqref{cond:atanh} is still satisfied in this case, since $y<0$.
	
	In case (iii), we check that for all $y\in \boI_c$,
	$$0\leq-\ka(2-c^2-2y)<1-2 \ka +2\ka y,$$
	and \eqref{cond:atanh} is  still satisfied, since $\ka<0$ and $y>0$.
	Therefore, both $\atanh$ in \eqref{eq:FD2} are well-defined. Noticing that
	$$\frac{d}{d y}
	\Bigg( \atanh\Big(\sqrt{-\ka}\sqrt{\frac{ 2-c^2-2y}{1-2 \ka +2\ka y}}\Big)
	\Bigg)=-\frac{\sqrt{-\ka}}{\sqrt{(2-c^2-2y)(1-2 \ka +2\ka y)}},
	$$
	the conclusion follows as in the previous cases.
\end{proof}

In view of Lemma~\ref{lem:etafunctions}, the inverse functions of $F_{c,\ka}$,
$H_{c,\ka}$ and $G_{c,\ka}$ are well-defined in $(0, \infty)$:
\begin{equation}
	\label{def:inverses2}
	\boF_{c,\ka}=F^{-1}_{c,\ka}, \text{ for }  (c,\ka)\in \D_1;  
	\quad 
	\boG_{c,\ka}=G^{-1}_{c,\ka}, \text{ for }  (c,\ka)\in \D_2;\quad 
	\boH_{c,\ka}=H^{-1}_{c,\ka}, \text{ for }  (c,\ka)\in \D_3. 
\end{equation}
We show that these functions can be smoothly extended to $\R$.

\begin{lemma}
	\label{lem:etafunctions2}
	Let $(c,\ka)\in \D$ and $\boF_{c,\ka}$, $\boG_{c,\ka}$		and $\boH_{c,\ka}$	
	be defined as in \eqref{def:inverses2}.
	Then they extend to $\R$ as even functions, that we still  denote by  
	$\boF_{c,\ka}$, $\boG_{c,\ka}$		and $\boH_{c,\ka}$, respectively, 
	and satisfy
	\begin{equation}
		\Im(\boF_{c,\ka})=(0,1-c^2/2], \quad 
		\Im(\boG_{c,\ka})=(0,1-c^2/2], \quad 
		\Im(\boH_{c,\ka})=[1-c^2/2,0). 
	\end{equation}
	In addition, 	these extensions belong to $\boC^\infty(\R)$.
\end{lemma}
\begin{proof}
We prove the statements only for $\boF_{c,\ka}$, since the proofs for $\boG_{c,\ka}$ and $\boH_{c,\ka}$ are similar. Since $F'_{c,\ka}\in\boC^\infty(\boI_c^\circ;(0, \infty))$ is negative valued, we have $\boF_{c,\ka}=F_{c,\ka}^{-1}\in\boC^1((0,\infty))$
with 
\begin{align}\label{eq:boFp}
	\boF_{c,\ka}'(x)=\frac{1}{F_{c,\ka}'(\boF_{c,\ka}(x))}=-\boF_{c,\ka}(x)\sqrt{\frac{2-c^2-2\boF_{c,\ka}(x)}{1-2\ka+2\ka\boF_{c,\ka}(x)}},
\end{align}
in fact, this yields $\boF_{c,\ka}\in\boC^\infty((0, \infty);\boI_c^\circ).$
Using \eqref{eq:boFp} we know the monotonicity of $\boF_{c,\ka}$,
and we conclude that  $\boF_{c,\ka}\in\boC([0,\infty);\boI_c)$, with $\boF_{c,\ka}(0)=1-1/c^2$. Symmetrizing $\boF_{c,\ka}$ with respect to the ordinate axis, and still denoting by $\boF_{c,\ka}$ the extension, we deduce that $\boF_{c,\ka}$ defines an even function and belongs to $\boC(\R)\cap\boC^\infty(\R\backslash\{0\})$. It remains to show that $\boF_{c,\ka}$ is smooth at the origin. For this purpose, let $h\in\R\setminus\{0\}$, so that, by the mean value theorem, there exists $0<|x_h|<|h|$ such that
$$
	\boF_{c,\ka}(h)-\boF_{c,\ka}(0)=h\boF_{c,\ka}'(x_h).
$$
Regarding \eqref{eq:boFp}, we have $\boF_{c,\ka}'(x_h)\rightarrow 0$ as $h\rightarrow0$, therefore $\boF_{c,\ka}\in\boC^1(\R;\boI_c)$ and $\boF_{c,\ka}'(0)=0$. Differentiating \eqref{eq:boFp} and proceeding by induction,  we infer that $\boF_{c,\ka}\in\boC^\infty(\R;\boI_c)$,  using the same  arguments.
\end{proof}

We are in a position to show that the function  $\eta_{c,\ka}$
defined in \eqref{def:eta:c} is the nonzero smooth solution to \eqref{eta2}--\eqref{eta1}.
\begin{proposition}\label{prop:globaleta}
Let $(c,\ka)\in \D$ and  $\eta_{c,\ka}$ defined in~\eqref{def:eta:c}.
Then $\eta_{c,\ka}$ belongs to $\boC^\infty(\R)\cap H^1(\R)$ and is a nonzero to solution to~\eqref{eta2}--\eqref{eta1}. Moreover, 
	it is, up to a translation,  the unique nonzero solution to~\eqref{eta2}--\eqref{eta1} in $\boC^2(\R)\cap H^1(\R)$.
\end{proposition}
\begin{proof}
We show the result only for $(c,\ka)\in \D_1$, so that $\eta_{c,\ka}=\boF_{c,\ka}$, since the other cases are 
	analogous. For the sake of simplicity, we omit the subscripts $c$ and $\ka$. By definition, $\eta$ is even, with $\eta(0)=1-c^2/2$,  $\eta'(0)=0$, $\eta>0$ on $\R$, and 
	\begin{equation}
		\label{dem:der:eta}
		F_{c,\ka}(\eta(x))=x, \quad \text{ for } x>0.
	\end{equation}
	In particular, $\eta$ satisfies~\eqref{eta1} for $x=0$.
	Bearing in mind that $\eta(x)\in \boI_c$, for all $x>0$, we can invoke Lemma~\ref{lem:etafunctions} to differentiate~\eqref{dem:der:eta}, using~\eqref{eq:impliFprim}, to obtain 
	\begin{equation}
		\label{eq:firstint}
		\frac{\eta'(x)}{\eta(x)}\sqrt{\frac{1-2\ka+2\ka\eta(x)}{2-c^2-2\eta(x)}}=-1,\quad\text{ for all }x>0,
	\end{equation}
	so that $\eta$ satisfies \eqref{eta1} for all $x\geq0$. Since $\eta$ is even, we conclude that 
	$\eta$ satisfies \eqref{eta1} in $\R.$ Also, by differentiating~\eqref{eta1}, we deduce that $\eta$ also solves \eqref{eta2}, using Remark~\ref{rem:etaeq}.
	
	The property $\eta\in H^1(\R)$ is a consequence of the decay of global solutions to \eqref{eta2}--\eqref{eta1}, stated in Corollary~\ref{coro:propeta}.

	Now, for the uniqueness, suppose that there exists another function $\tilde \eta\in\boC^2(\R)\cap H^1(\R)$, a nonzero solution to \eqref{eta2}--\eqref{eta1}. Then \eqref{eq:maxeta} applies, so that we can assume, up to a translation, that $\tilde\eta$ reaches a global extremum at the origin, with $\tilde \eta(0)=1-c^2/2$. Thus $(\tilde\eta,\tilde\eta')$ is also a global solution to the ODE~\eqref{eq:odeetaord2} with initial condition $(\tilde\eta,\tilde\eta')(0)=(1-c^2/2,0).$ 	Therefore, by Cauchy--Lipschitz theorem, we conclude that  $\eta=\tilde \eta$.
\end{proof}
Now we can prove Theorem \ref{thm:classiftwregu}.
\begin{proof}[Proof of Theorem~\ref{thm:classiftwregu}]
The assertion~\ref{thm:item:regunoexi} is a direct consequence of Corollary~\ref{coro:crittrivial} and Proposition~\ref{prop:k1/2}.

To prove \ref{thm:item:reguexi}, let $(c,\ka)\in\D1$, let $\eta$ be the unique (up to a translation) global solution to \eqref{eta2}--\eqref{eta1} given by Proposition~\ref{prop:globaleta}. 
 In the case $c>0$, from Proposition~\ref{prop:exiu}, $u=\sqrt{1-\eta}e^{i\theta}$ with $\theta$ given by \eqref{eq:deftheta} is up to invariances the unique finite energy solution of \eqref{TWc}. The smoothness and decay of $u$ is a consequence of the smoothness and the decay of $\eta$.

 It remains to analyze the case $c=0$, where $\eta(0)=1$ and $0<\eta<1$ on $\R\setminus\{0\}$, and 
 $u$ is the odd real function defined in \eqref{black-soliton}, so that $u(0)=0$. 
Hence,  we need to modify the arguments given above. 
It is clear that $u$ is smooth for $x\neq 0$. To avoid tedious computations, 
we use the following argument to verify that $u$ is a smooth solution to \eqref{TWc}.

Let us define $\tilde u\in\boC^\infty((-R,R))$ as the local real solution, given  by the Cauchy--Lipschitz theorem, of 
\begin{align}\label{eq:pbreel}
&(1-2\ka \tilde u^2)\tilde u''+\tilde u(1-\tilde u^2-2\ka (\tilde u')^2)=0,\\
&\label{eq:pbreel2}	\tilde u(0)=0,~
\tilde u'(0)=1/{\sqrt{2}},
\end{align}
for some $R>0$, since $1-2\ka y^2>0$, for $y$ near $0$. In this manner, $\tilde u$ satisfies 
\eqref{TWc} with $c=0$ on $(-R,R)$. Notice that we chose the value of $\tilde u'(0)$ to have compatibility with the identity in \eqref{eq:quadratic}, since $\eta(0)=1$ and $\eta'(0)=0$.
Let $\tilde\eta=1-\tilde u^2$. Since $\tilde\eta\leq 1$ in $(-R,R)$, with 
 $\tilde\eta(0)=1$, we infer that $\tilde\eta$ reaches a global maximum at the origin, so that $\tilde\eta'(0)=0$.
 Arguing as in the proof of Proposition~\ref{prop:eqeta}, with $u_2\equiv 0$ and $c=0$, but integrating between $0$ and $x$, instead of $x$ and $R_n$, we conclude that  $\tilde\eta$  satisfies equation \eqref{eta2} with $c=0$ in $(-R,R)$. Thus, $\tilde\eta$ and $\eta$ satisfy the same ODE problem \eqref{eq:odeetaord2}, with the same initial condition. By Cauchy--Lipschitz theorem, we have $\tilde\eta=\eta$, in $(-R,R)$, i.e.
 \begin{equation}
      \label{eq:black}
 \abs{\tilde u(x)}=\abs{u(x)}, \quad \text{ for all }x\in(-R,R).
 \end{equation}
 Recall that $u>0$ in $(0,\infty)$, and that $u<0$ in $(-\infty,0)$. Since  $\tilde u'(0)>0$, we deduce that there is some $R_0 \in(0,R]$ such that 
 $\tilde u> 0$ in $(0,R_0)$ and $\tilde u<0$ in $(0, R_0)$. 
We conclude from \eqref{eq:black} that $\tilde u=u$ in $(-R_0,\tilde R_0)$, so that 
$u\in\boC^\infty(\R)$. This also implies that $\tilde u$
is a global solution to \eqref{eq:pbreel}--\eqref{eq:pbreel2}, and that $\tilde u=u$ in $\R$, so that $u$
is solution to \eqref{TWc}, with $c=0$.
	
Finally, we need to prove the uniqueness of the black soliton $u$ in \eqref{black-soliton}, up to invariances. Let us assume that  $\check u \in\boC^\infty(\R;\C)\cap\boX(\R)$ is another solution to \eqref{TWc} with $c=0$. Setting  $\check \eta =1-|\check u|^2$, 
  we deduce that $\check\eta$ satisfies \eqref{eta2}--\eqref{eta1} by Proposition~\ref{prop:eqeta}, so that
  Proposition~\ref{prop:globaleta} implies that 
 $\check\eta =\eta(\cdot-x_0)$ for some $x_0\in\R$.
  Up to a translation, we assume that $x_0=0$.
Also, by Corollary \eqref{coro:propeta}, we deduce that  
\begin{equation}\label{born:check:u}
\abs{\check{u}}\leq 1 \text{ on }\R. 
\end{equation}
 In this manner, we have $\check\eta(0)=1$ and  $\check\eta'(0)=0$, so that 
$\check u(0)=0$. By using \eqref{eq:quadratic}, we also get $|\check u'|(x_0)=1/\sqrt 2,$ thus $\check u'(0)=e^{i \phi}/\sqrt 2$, for some $\phi\in\R$.

To conclude, we consider the function  
 $v= e^{-i\phi}\check u$. It is clear that $v$ satisfies \eqref{TWc} with $c=0$. 
 Therefore, the real-valued function $V=(\Re(v),\Im(v))$
satisfies the ODE system in \eqref{eq:systtw}, with $c=0$, that we rewrite as 
\begin{equation}\label{eq:GV}
    V''=G(V,V'),
\end{equation}
with initial condition 
\begin{equation}
\label{IC:V}
    V(0)=(0,0), \quad V(0)=(1/\sqrt 2,0).
\end{equation}
Notice that $G(V,V')$ includes the multiplication by $A(V)^{-1}$, which is well-defined on $\R$, in view of \eqref{born:check:u}, since $2\kappa<1$.
Because the black soliton $u$ satisfies \eqref{eq:pbreel}--\eqref{eq:pbreel2},  
the function  $U=(u,0)$ is also  a solution to \eqref{eq:GV}--\eqref{IC:V}.
Therefore, by the Cauchy--Lipschitz theorem, we conclude that $U=V$, i.e.\ that 
$\check u=e^{i\phi }u$ on $\R$, completing the proof.
\end{proof}
We end this section with a word on the regularity of $u_{c,\ka}$ with respect to $(c,\ka)$. Because $\eta_{c,\ka}$ satisfies the autonomous ODE \eqref{eq:odeetaord2} and since the right-hand side is a smooth vector field $F=F(c,\ka,\eta,\eta_2)$, the regularity of the flow with respect to the initial condition (see e.g \cite{hartman1982ordinary}) yields $\eta\in\boC^\infty(\D\times\R)$, where $\eta(c,\ka,x)=\eta_{c,\ka}(x).$ We prove  by induction that $\eta\in\boC^\infty(\D,H^k(\R))$ for all $k\in\N$.
\begin{proposition}\label{prop:reguparam1}
	For every multi-index $\alpha=(\alpha_1,\alpha_2,j)\in\N^3$, $(c,\ka)\in \D$, there exists $R$, 
 $C_0$, $C$ positive constants depending continuously on $(c,\ka)$ but not on $x\geq R$, such that for all $\abs{x}\geq R$
	\begin{align}\label{eq:decayck}
	|D^\alpha_{c,\ka,x}\eta_{c,\ka}(x)|\leq C_0 e^{-C(|x|-R)},
	\end{align} where $D^\alpha_{c,\ka,x}=\partial_c^{\alpha_1}\partial_\ka^{\alpha_2}\partial_x^j$.
	Moreover, $\eta\in\boC^\infty(\D;H^j(\R))$ where $\eta(c,\ka,x)=\eta_{c,\ka}(x)$.
\end{proposition}
\begin{proof}
 By simplicity, we prove \eqref{eq:decayck} only for $\alpha_2=0$, since the case $\alpha_2\geq 1$ follows by induction on $\alpha_2$.

 By induction on $\alpha_1$, the case $\alpha_1=0$ is just the exponential decay of $\eta$ and all its derivatives in $x$. Since $\eta_{c,\ka}$ satisfies \eqref{eta2}--\eqref{eta1}, from Corollary~\ref{coro:propeta} and the regularity of the flow, we have for all $\alpha_1\geq0$ and all $x>0$
	\begin{align}\label{eq:flowreg}
		\partial_x\partial_c^{\alpha_1}\eta_{c,\ka}(x)=-\frac{d^{\alpha_1}}{(d c)^{\alpha_1}} F(c,\ka,\eta_{c,\ka}(x)),\quad\text{ where }F(c,\ka,y)=y\sqrt{\frac{2-c^2-2y}{1-2\ka+2\ka y}}.
	\end{align} We infer that using the induction hypothesis, there exists $R>1$ such that for all $x\geq R$
\begin{align}\label{eq:majgronvck}
	|\partial_x\partial_c^{\alpha_1}\eta_{c,\ka}(x)+\partial_yF(c,\ka,\eta_{c,\ka}(x))\partial_c^{\alpha_1}\eta_{c,\ka}(x)|\leq C_1e^{-C_2(x-R)}\text{, for some } C_1,C_2>0.
\end{align}
For instance, with $\alpha_1=1,$ \eqref{eq:flowreg} yields
\begin{align}\label{eq:edoflotc}
    \partial_x\partial_c\eta_{c,\ka}(x)=-\frac{d}{d c} F(c,\ka,\eta_{c,\ka}(x))=-\partial_c\eta_{c,\ka}(x)\partial_yF(c,\ka,\eta_{c,\ka})-\eta_{c,\ka}\frac{\partial_cF(c,\ka,\eta_{c,\ka})}{\eta_{c,\ka}},
\end{align} Using $|\eta(x)|\leq|\eta(1)|<|1-c^2/2|$ for all $x\geq 1$, we can bound $(\partial_cF)/\eta$ independently of $x\geq 1$, we can conclude that \eqref{eq:edoflotc} implies \eqref{eq:majgronvck}. On the other hand, we have
\begin{align}
\partial_yF(c,\ka,\eta_{c,\ka})=\frac{(1-2\ka+2\ka \eta_{c,\ka})(2-c^2-2\eta_{c,\ka})+\eta_{c,\ka}(\ka c^2-1)}{\sqrt{(1-2\ka+2\ka\eta_{c,\ka})^3(2-c^2-2\eta_{c,\ka})}}.
\end{align}
Since $\partial_yF(c,\ka,\eta_{c,\ka}(x))$ tends to $\sqrt{(2-c^2)/(1-2\ka)}$, as $x\to\infty$, we get  $C_3<\partial_yF(c,\ka,\eta_{c,\ka}(x))$ for all $x\geq R$ and some $C_3>0$ and $R>1$. From \eqref{eq:majgronvck}, we deduce the decay estimate as follows. Let $G(x)=\int_R^x\partial_yF(c,\ka,\eta_{c,\ka}(s))ds$, then, multiply \eqref{eq:majgronvck} by $e^{G(x)}$ and integrate from $R$ to $w\geq R$. By integration by part of $\int_R^w e^{G(x)}\partial_x\partial_c^{\alpha_1}\eta_{c,\ka}(x)dx$, we obtain
\begin{equation}\label{eq:decay05}-C_1 \int_R^w e^{G(x)-C_2(x-R)}dx\leq e^{G(w)}\partial_c^{\alpha_1}\eta_{c,\ka}(w)-\partial_c^{\alpha_1}\eta_{c,\ka}(R)\leq C_1 \int_R^w e^{G(x)-C_2(x-R)}dx.
\end{equation}
Moreover, we can check that for all $R\leq x\leq w$, we have $G(x)-G(w)\leq -C_3(w-x)$.
Adding $\partial_{c,\ka}\eta_{(c,\ka)}(R)$ and multiplying by $e^{-G(w)}$ every side of \eqref{eq:decay05} we get, using the latter estimate on $G$
\begin{equation}\label{eq:decay051}
\partial_{c,\ka}\eta_{c,\ka}(R)e^{-G(w)}-C_1 \int_R^w e^{-C_3(w-x)-C_2(x-R)}dx\leq\partial_c^{\alpha_1}\eta_{c,\ka}(w),
\end{equation}
and also the upper bound on $\partial_c^{\alpha_1}\eta_{c,\ka}$
\begin{equation}\label{eq:decay0512}
\partial_c^{\alpha_1}\eta_{c,\ka}(w)\leq \partial_c^{\alpha_1}\eta_{c,\ka}(R)e^{-G(w)}+C_1\int_R^w e^{-C_3(w-x)-C_2(x-R)}dx.
\end{equation}
If $C_2=C_3$ we obtain using \eqref{eq:decay051}--\eqref{eq:decay0512}
\begin{equation}\label{eq:decay0513}
	|\partial_c^{\alpha_1}\eta_{c,\ka}(w)|\leq( |\partial_c^{\alpha_1}\eta_{c,\ka}(R)|+C_1(w-R))e^{-C_2(w-R)},
\end{equation}
while if $C_2\ne C_3$, computing the integrals in \eqref{eq:decay051}--\eqref{eq:decay0512} yields
\begin{equation}\label{eq:decay0514}
	|\partial_c^{\alpha_1}\eta_{c,\ka}(w)|\leq |\partial_c^{\alpha_1}\eta_{c,\ka}(R)|e^{-C_3(w-R)}+ \frac{C_1}{|C_3-C_2|}(e^{-C_2(w-R)}+e^{-C_3(w-R)})
\end{equation}
From \eqref{eq:decay0513}--\eqref{eq:decay0514} follows the exponential decay estimate \eqref{eq:decayck} for $j=0$. Using the same method, we get \eqref{eq:decayck} for $x\leq-R$. The case $j\geq1$ is then obtained by induction, differentiating \eqref{eq:flowreg} $j-1$ times with respect to $x$.

Since $(\eta_{c+h,\ka}(x)-\eta_{c,\ka}(x))/h-\partial_c\eta_{c,\ka}(x)$ converges to $0$ for every $x\in\R$ as $h\rightarrow0$ and $\sup_{c_h\in(c-h,c+h)}|\partial_c\eta_{c_h,\ka}(\cdot)|$ is a continuous exponentially decaying function uniformly in $(c,\ka)$ since the constant in \eqref{eq:decayck} are continuous with respect to $(c,\ka)\in\D$. We also get the convergence to 0 in the $L^2$-norm using the dominated convergence theorem. Induction of this argument ensures that $\eta\in\boC^\infty(\D;H^k(\R))$.
\end{proof}
\section{Classification of singular traveling waves}
\label{sec:construction:weak}
In the following subsection, we prove Theorems~\ref{thm:non-singular-sol}, \ref{thm:classifcuspon}, \ref{thm:symmetry}, and Propositions~\ref{prop:nonvanishsingu} and~\ref{prop:constant} about qualitative properties on the weak solutions to \eqref{TWc}.}
\subsection{Properties of singular solutions}\label{subseq:propw}
Let $(c,\ka)\in[0,\infty)\times\R$, throughout this subsection, we assume that $u_{c,\ka}\in\boX(\R)$ is a solution to \eqref{TWc}.  
Let  $u_1=\Re(u_{c,\ka})$, $u_2=\Im(u_{c,\ka})$ and  $\eta_{c,\ka}=1-\abs{u_{c,\ka}}^2$. 
Taking $\phi=\phi_1\in H^1(\R;\R)$ and $\phi=\phi_2\in H^1(\R;\R)$, 
in \eqref{TW:weak}, we see that \eqref{TW:weak} is equivalent to the system of two real equations in \eqref{eq:systtw}, satisfied in the weak sense.
Therefore, since $u_{c,\ka}$ is continuous, we expect that solutions to  \eqref{TWc} 
are smooth on the open set 
\begin{equation}\label{eq:defomega}
	\Omega(u_{c,\ka}):=\Big\{x\in\R : \abs{u_{c,\ka}}^2\ne\frac{1}{2\ka}\Big\}=\Gamma(u_{c,\ka})^c. 
\end{equation}
This is exactly the conclusion of the result below.
Note that if $ \ka\leq0$, we trivially have $\Omega(u_{c,\ka})=\R$. Also,
in the case $\ka\ne1/2$, bearing in mind \eqref{eq:limeta}, we can find $R>0$ such that $-\infty<-R\leq\ga_{c,\ka}\leq\gb_{c,\ka}\leq R<\infty$ and $(-\infty,-R)\cup (R,\infty)\subset \Omega(u_{c,\ka})$, where $\ga_{c,\ka}$, $\gb_{c,\ka}$ are given by \eqref{def:supcrit}.

\begin{lemma}
	\label{lem:omegaregu}
	Let $(c,\ka)\in[0, \infty)\times\R$. If $u_{c,\ka}\in\boX(\R)$ satisfies \eqref{TWc}, then $u_{c,\ka}\in \boC^\infty(\Omega(u_{c,\ka}))$. In particular, if $\ka\ne1/2$ so that  $-\infty<\ga_{c,\ka}\leq\gb_{c,\ka}<\infty$, then $u_{c,\ka}$ is smooth in $(-\infty,\ga_{c,\ka})\cup(\gb_{c,\ka},\infty)$.
\end{lemma} 
\begin{proof}
As usual, let  $u=u_{c,\ka}$ and write $u=u_1+iu_2$,  so that $U=(u_1,u_2)$ satisfies the system \eqref{eq:systtw}, that we recast as 
	$$A(u)U''=F(U,U'),$$ in the distributional sense,  		where $F(U,U')$ denotes the right-hand side in \eqref{eq:systtw}.
	Because $F(U,U')$ belongs to $L^1_{\loc}( \R)$, we conclude that 	$A(u)U''$	
	also belongs to $L^1_{\loc}( \R).$ Since the determinant \eqref{determinant} in nonzero on $\Omega(u)$, we deduce that $(A(u))^{-1}\in L^\infty_{\loc}( \Omega(u))$, and thus 
	$ U''=(A(u))^{-1}F(U,U') \in L^1_{\loc}( \Omega(u))$.
	Therefore, $U'\in  W^{1,1}_\mathrm{loc}(\Omega(u))$.
	Using that, by the Sobolev embedding theorem, the functions $W_{\rm{loc}}^{1,1}(\Omega(u))$ are continuous, we conclude that  $u'$ is in $\boC(\Omega(u))$. This implies that $F(U,U')\in\boC(\Omega(u))$, so $u$ is twice continuously differentiable in $\Omega(u)$. The smoothness follows using a  bootstrap argument. 
\end{proof}

If $ \abs{u_{c,\ka}(x)}^2\ne1/(2\ka)$, for all $x\in\R$, then by Lemma~\ref{lem:omegaregu}, $u_{c,\ka}\in\boC^2(\R)$ so we can refer to Theorem \ref{thm:classiftwregu} to determine $u_{c,\ka}$. Hence, from now on, we always assume $\Omega(u_{c,\ka})\ne\R$ (or equivalently $\Gamma(u_{c,\ka})\ne\emptyset$) in an attempt to describe singular solutions to \eqref{TWc}.
To analyze the behavior at infinity of singular solutions (see Theorem~\ref{thm:classifcuspon}), we introduce the following variant of equations \eqref{eta2}--\eqref{eta1}. Let $u_{c,\ka}\in\boX(\R)$ be a singular solution to \eqref{TWc} such that $\Gamma(u_{c,\ka})$ is bounded, and define the real numbers $\ga_{c,\ka}$, $\gb_{c,\ka}$ according to \eqref{def:supcrit}. Repeating the arguments of Proposition~\ref{prop:eqeta}, we deduce that $\eta_{c,\ka}=1-\abs{u_{c,\ka}}^2$ satisfies the following ODEs:
	\begin{align}
	\label{eta2l}
	(1-2\ka+2\ka\eta_{c,\ka})\eta_{c,\ka}''+\ka(\eta_{c,\ka}')^2=-3\eta_{c,\ka}^2+(2-c^2)\eta_{c,\ka}, \quad \text{ in }(-\infty,\ga_{c,\ka})\cup(\gb_{c,\ka}, \infty),\\
	\label{eta1l}
	(1-2\ka+2\ka\eta_{c,\ka})(\eta_{c,\ka}')^2=\eta_{c,\ka}^2(2-c^2 -2\eta_{c,\ka}), \quad \text{ in }(-\infty,\ga_{c,\ka})\cup(\gb_{c,\ka}, \infty).
\end{align}
In particular, the sign constraints in equation \eqref{eta1l} imply some conditions on $(c,\ka)$, which we discuss in the following lemma.
\begin{lemma}\label{lem:nonexiw}
	Let $(c,\ka)\in[0, \infty)\times\R$ and assume that $\ka\ne1/2$. If $u_{c,\ka}\in\boX(\R)$ satisfies \eqref{TWc} and $\Gamma(u_{c,\ka})\ne\emptyset$, then we necessarily have $(c,\ka)\in\tilde{\D}$.
\end{lemma}
\begin{proof}
Let us remove the subscripts of $u_{c,\ka}$.
	We saw above that since $\ka\ne1/2$, $\Gamma(u)$ is bounded, so that the function $\eta=1-\abs{u}^2$ satisfies \eqref{eta2l}--\eqref{eta1l}. By contradiction, assume that $(c,\ka)\in
	([0, \infty)\times\R)\backslash\tilde{\D}$. Then, since $\ka\ne1/2$, we have either $(c,\ka)\in\D_2$ or $(c,\ka)$ meet the assumption of Corollary~\ref{coro:crittrivial}. 
 By definition of $\Gamma(u)$, if $\ka\leq0$, then $\Gamma(u)=\emptyset$, therefore $(c,\ka)$ meet the assumption of Corollary~\ref{coro:crittrivial}.
Using ideas along the same lines as in Corollary~\ref{coro:crittrivial}, we infer that there exists $x_0>\gb$ such that $(2-c^2-2\eta(x_0))/(1-2\ka+2\ka\eta(x_0))<0$, which further implies that $(\eta'(x_0))^2<0$ in \eqref{eta1l}, a contradiction. We conclude that $(c,\ka)\in\tilde{\D}$.
\end{proof}
We are now in a position to prove Theorem~\ref{thm:non-singular-sol}, using the previous lemmas.
\begin{proof}[Proof of Theorem~\ref{thm:non-singular-sol}]
Let $(c,\ka)\notin(\tilde{D}\cup\boC)$ and assume by contradiction that $\Gamma(u_{c,\ka})\ne\emptyset$. Since $(c,\ka)\notin\boC$, we can apply Lemma~\ref{lem:nonexiw} to deduce that $(c,\ka)\in\tilde{\D}$, contradicting the assumptions on $(c,\ka)$.
Thus $\Gamma(u_{c,\ka})=\emptyset$, and using Lemma~\ref{lem:omegaregu}, we deduce that $u_{c,\ka}\in\boC^2(\R)$, so that the conclusion follows from  Theorem~\ref{thm:classiftwregu}.
\end{proof}

Let $(c,\ka)\in\tilde{\D}$, and $\ga\leq\gb$ be given real numbers. We discuss the behavior of any solution $\eta$ to \eqref{eta2l}--\eqref{eta1l} with $\ga_{c,\ka}=\ga$ and $\gb_{c,\ka}=\gb$ reaching the critical value at the boundaries $\eta(\ga^-)=\eta(\gb^+)=1-1/(2\ka)$. We will show later,  in the proof of Theorem~\ref{thm:classifcuspon},  that such a solution is uniquely determined using \eqref{eq:etaw} and satisfies $\eta(x)=\eta(\ga-x+\gb)$, for all $x<\ga$. Therefore, in the next result, we consider only the properties of the function on  $(\gb, \infty)$. Most of the ideas in this result come from Corollary~\ref{coro:propeta}.  
\begin{lemma}\label{lem:propetaw}
	Let $(c,\ka)\in\tilde{\D}$ and assume that one of the following conditions holds:
	 \begin{center}
		\textup{($i$)}	$(c,\ka)\in \D_1\cup \B_-$\quad or \quad \textup{($ii$)} $(c,\ka)\in \D_3\cup \B_+$.
	\end{center} 
Suppose that $\eta\in\boC^2((-\infty,\ga)\cup(\gb,\infty))$  satisfies \eqref{eq:limeta}, \eqref{eta2l}--\eqref{eta1l}, and the boundary conditions 
$$\eta(\ga^-)=\eta(\gb^+)=1-1/(2\ka).$$
Then $\eta\in\boC^\infty((\gb,\infty))$,   we have $\eta<0$ and $\eta'>0$ on $(\gb,\infty)$ in case $(i)$, while in case (ii), $\eta>0$ and $\eta'<0$ on $(\gb,\infty)$.
In addition, if  $(c,\ka)\in \D_1\cup \D_3$, for all $j\in\N$, there exist $C, C_0>0$ such that   $$	|\partial_x^j\eta(x)|\leq C_0e^{-C_1(|x|-\gb-1)},\quad\text{ for all }\abs{x}\geq\gb+1,$$
Finally,  if $(c,\ka)\in \B_-\cup \B_+$, for all $j\in\N$, there exist $C_0$, $C_1>0$ such that 
\begin{align}\label{eq:decaycrit}
	|\partial_x^j\eta(x)|\leq \frac{C_0}{(x-\gb-1+C_1)^{2+j}},\quad\text{ for all }\abs{x}\geq\gb+1.
\end{align}
\end{lemma}
\begin{proof}
	We place ourselves in case (i). We prove that the function $\eta$ does not vanish in $(\gb, \infty)$ using the same ODE argument as in Corollary~\ref{coro:propeta}. Then, since $\eta(\gb^+)<0$, we get $\eta<0$ in $(\gb, \infty)$. Using condition \eqref{eq:limeta}, there must exist $R>\gb$ such that $1-1/(2\ka)<\eta(R)<0$, hence by the mean value theorem, there exists $\gb<x_0<R$ such that $\eta'(x_0)(R-\gb)=\eta(R)-1+1/(2\ka)$. We conclude that $\eta'>0$ is the neighborhood of $x_0$. Now suppose that $\eta'(\tilde x)=0$, for some $\tilde x\in(\gb, \infty)$. From \eqref{eta1l}, we infer that $\eta(\tilde x)=0$ or $\eta(\tilde  x)=1-c^2/2\geq0$, which is absurd since $\eta$ cannot vanish in $(\gb,\infty)$. Therefore, $\eta'>0$ in $(\gb, \infty)$. Proceeding along the same lines as in Corollary~\ref{coro:propeta} we recover the exponential decay of $\eta$ if $c\ne\sqrt{2}$.
 It remains to prove \eqref{eq:decaycrit} for $(\sqrt{2},\ka)\in\B_-$. In this case, since $\eta<0$ and $\eta'>0$ in $(\gb,\infty)$, we obtain using \eqref{eta1l}
	\begin{align}\label{eq:decayeq1}\eta'(x)=\frac{\sqrt{-2\eta^3(x)}}{\sqrt{1-2\ka+2\ka\eta(x)}},\quad\text{ for all }x>\gb.
	\end{align}
    Thus $ \eta'(x)\geq \sqrt{-2\eta^3(x)}/\sqrt{1-2\ka}$, for all $x\geq\gb+1$, and we obtain \eqref{eq:decaycrit} for $j=0$ by integrating this inequality. We get \eqref{eq:decaycrit} for higher order derivatives using \eqref{eq:decayeq1} for $j=1$, and differentiating  \eqref{eta2l} with respect to $x$, together with an induction argument for $j\geq2$. Case (ii) is analogous, bearing in mind that $\eta(\gb^+)=1-1/(2\ka)>0$, so that $\eta>0$ and $\eta'<0$ in $(\gb,\infty)$.
\end{proof}
We now prove an adapted version of Lemma~\ref{lem:etafunctions} to the functions \eqref{eq:FD1W}--\eqref{eq:FD5W}, using the same arguments.
\begin{lemma}\label{lem:etawfunctions}
The functions  $f_{c,\ka}$, $g_{\ka}$, $\tilde{g}_{\ka}$ and  $h_{c,\ka}$ are well-defined. Moreover, 
 \begin{enumerate}
 [leftmargin=4.5ex,topsep=1pt,itemsep=-5pt,partopsep=0ex,parsep=1ex]
		\item\label{lem:item:etaf} $f_{c,\ka}$ is increasing, belongs to $\boC^\infty(\mathfrak{T}_\ka^\circ)\cap \boC(\mathfrak{T}_\ka)$, and 
		\begin{equation}\label{eq:implifwprim}
		f_{c,\ka}(1-1/(2\ka))=0, \	\lim_{y\to 0^-}f_{c,\ka}(y)=\infty, \ f_{c,\ka}'(y)=-\frac{1}{y}\sqrt{\frac{1-2\ka+2\ka y}{2-c^2-2y}},\ \text{ for }y \in \mathfrak{T}^\circ_\ka.
		\end{equation}
  	\item\label{lem:item:etag} $g_{\ka}$ is increasing, belongs to $\boC^\infty(\mathfrak{T}_\ka^\circ)\cap \boC(\mathfrak{T}^\circ_\ka)$, and 
		\begin{equation}\label{eq:impligwprim}
			g_{\ka}(1-1/(2\ka))=0,\  	\lim_{y\to 0^-}g_{\ka}(y) =\infty, 
			\ g_{\ka}'(y)=-\frac{1}{y}\sqrt{\frac{1-2\ka+2\ka y}{-2y}},\ \text{ for }y \in \mathfrak{T}^\circ_\ka.
		\end{equation}
  \item\label{lem:item:etagt} $\tilde{g}_{\ka}$ is decreasing, belongs to $\boC^\infty(\mathfrak{J}_\ka^\circ)\cap \boC(\mathfrak{J}_\ka)$, and 
	\begin{equation}\label{eq:impligjwprim}
			\lim_{y\to 0^+}\tilde{g}_{\ka}(y) =\infty,\  \tilde{g}_{\ka}(1-1/(2\ka))=0, 
		\ \tilde{g}_{\ka}'(y)=-\frac{1}{y}\sqrt{\frac{1-2\ka+2\ka y}{-2y}},\ \text{ for }y \in \mathfrak{J}_\ka.
	\end{equation}
		\item\label{lem:item:etah} $h_{c,\ka}$ is decreasing, belongs to $\boC^\infty(\mathfrak{J}_\ka^\circ)\cap \boC(\mathfrak{J}^\circ_\ka)$, and 
	\begin{equation}\label{eq:implihwprim}
		 	\lim_{y\to 0^+}h_{c,\ka}(y)=\infty,\ h_{c,\ka}(1-1/(2\ka))=0, \ h_{c,\ka}'(y)=-\frac{1}{y}\sqrt{\frac{1-2\ka+2\ka y}{2-c^2-2y}},\ \text{ for }y \in \mathfrak{J}^\circ_\ka.
	\end{equation}

		\end{enumerate} Consequently, the functions in \eqref{eq:FD1W}--\eqref{eq:FD5W} are injective and their inverse lie in $\boC^\infty(\mathfrak{T}^\circ_\ka)\cap\boC(\mathfrak{T}_\ka)$ in cases \ref{lem:item:etaf}--\ref{lem:item:etag} (respectively in $\boC^\infty(\mathfrak{J}^\circ_\ka)\cap\boC(\mathfrak{J}_\ka)$ in cases \ref{lem:item:etagt}--\ref{lem:item:etah}).
\end{lemma}
\begin{proof}
	For case \ref{lem:item:etaf}, it can be proved just as in Lemma~\ref{lem:etafunctions} that the square roots are well-defined for all $y\in\mathfrak{T}_\ka^\circ$. The only difference between the formula of $f_{c,\ka}$ and $F_{c,\ka}$ given by \eqref{eq:FD1}, is the inversion of the argument in $\atanh$. Here, we have $y(c^2\ka-1)>0$, which is equivalent to the condition $(1-2\ka)(2-c^2-2y)>(2-c^2)(1-2\ka+2\ka y)>0$. Therefore, the 
	$\atanh$ in $f_{c,\ka}$ is well-defined.
	Moreover, we have for all $y\in\mathfrak{T}_\ka^\circ$,
	$$\frac{d}{d y}
	\Bigg(
	\atanh\Big(\sqrt{\frac{(2-c^2)(1-2\ka+2\ka y)}{(1-2 \ka) (2-c^2-2y)}}\Big)
	\Bigg)=
	-\frac{\sqrt{(1-2\ka)(2-c^2)}}{2y\sqrt{(2-c^2-2y)(1-2 \ka +2\ka y)}}, 
	$$
	We refer to Lemma~\ref{lem:etafunctions} for the formula of the derivative of the other terms in $f_ {c,\ka}$.
	The other cases are analogous to \ref{lem:item:etaf}.
\end{proof} 
We are in a position to prove Theorem~\ref{thm:classifcuspon}.
\begin{proof}[Proof of Theorem~\ref{thm:classifcuspon}]
    Recall that $u_{c,\ka}\in\boC^\infty((-\infty,\ga_{c,\ka})\cap(\gb_{c,\ka},\infty))$ with $-\infty<\ga_{c,\ka}\leq\gb_{c,\ka}<\infty$ and that $\eta_{c,\ka}$ satisfies \eqref{eta2l}--\eqref{eta1l}. Without loss of generality, we assume that $\gb_{c,\ka}=0$ and we set $\ga=\ga_{c,\ka}$.
	Let us treat the case $(c,\ka)\in \D_1$, using Lemma~\ref{lem:propetaw}, we get $1-2\ka<\eta_{c,\ka}(x)<0,$ (i.e.\ $\eta_{c,\ka}(x)\in\mathfrak{T}^\circ_\ka$) for all $x>0$ and that $\eta_{c,\ka}'>0$ for all $x>0$. From this and equation \eqref{eta1l}, we have in the fashion of \eqref{eq:odeeta1} 
	\begin{align}\label{eq:intweak}
		\frac{\eta_{c,\ka}'(x)}{\eta_{c,\ka}(x)}\sqrt{\frac{1-2\ka+2\ka\eta_{c,\ka}(x)}{2-c^2-2\eta_{c,\ka}(x)}}=-1,\quad\text{ for all }x>0.
	\end{align}
	Integrating \eqref{eq:intweak} from $\varepsilon$ to $x$ and taking the limit as $\varepsilon\to0$, we obtain from \eqref{eq:implifwprim} that  $f_{c,\ka}(\eta_{c,\ka}(x))=x$ for all $x\geq0$ and we obtain \eqref{eq:etaw} by applying  $f_{c,\ka}^{-1}$ to both sides. Additionally, the symmetry formula $\eta_{c,\ka}(x)=\eta_{c,\ka}(\ga-x)$ holds for all $x\geq0$. Indeed, using ideas from Lemma~\ref{lem:propetaw}, we can show that $\eta_{c,\ka}<0$ and $\eta_{c,\ka}'<0$ in $(-\infty,\ga)$. Thus, $\eta_{c,\ka}(\ga-x)\in\mathfrak{T}_{\ka}$, for all $x\geq0$, and also satisfies \eqref{eq:intweak}. Then, we infer that
 $$f_{c,\ka}(\eta_{c,\ka}(\ga-x))=f_{c,\ka}(\eta(x))=x,\quad\text{ for all }x\geq0,$$ and we conclude by applying $f_{c,\ka}^{-1}$ to the previous equation. Since $1-1/(2\ka)<\eta_{c,\ka}(x)<0,$ we have $\eta_{c,\ka}(x)<1$ for all $x\in(-\infty,\ga)\cup(0, \infty)$. Hence, we recover $u_{c,\ka}=\sqrt{1-\eta_{c,\ka}}e^{i\theta_{c,\ka}}$ with $\theta_{c,\ka}'=c\eta_{c,\ka}/(2-2\eta_{c,\ka})$ in $(-\infty,\ga)\cup(0,\infty)$ using ideas similar to the ones in Corollary~\ref{coro:eqpolaire} adapted to intervals of the form $(-\infty,\ga)\cup(0,\infty)$. On the other hand, the limits of $\eta_{c,\ka}'$ as $x\to\ga_{c,\ka}^-$ and $x\to0^+$ are obtained taking the limit in \eqref{eq:intweak} and using $\eta_{c,\ka}(x)=\eta_{c,\ka}(\ga-x)$ for all $x\leq\ga$.
 The general result for $(c,\ka)\in\tilde{D}$ follows from the same arguments because we can show that the formula \eqref{eq:intweak} still holds in all cases.
 Notice that formula \eqref{eq:etaw} does not contradict the finite energy assumption on $u_{c,\ka}.$ Indeed, using ideas along the same line as in the proof of Proposition~\ref{prop:exiu}, we deduce that the condition $u'_{c,\ka}\in L^2((-\infty,\ga)\cup((0,\infty))$ reduces to $\eta_{c,\ka}\in H^1((-\infty,\ga)\cup((0,\infty))$. From the symmetry formula, this question further simplifies to whether or not $\eta_{c,\ka}\in H^1((0,\infty))$. For all $\ve>0$ due to the exponential decay estimates in Lemma~\ref{lem:propetaw}, we have $\eta_{c,\ka}\in H^1((\ve,\infty))$. Since $\eta_{c,\ka}$ in continuous in $(0,\infty)$, we have $\eta_{c,\ka}\in L^2((0,\ve))$, and, concerning $\eta'_{c,\ka}$, which is unbounded near 0, we get using  equation \eqref{eq:intweak},
	\begin{align}\label{eq:intcriti}
		\int_{0}^{\varepsilon}\eta'(x)^2dx=-\int_{0}^{\varepsilon}\eta'(x)\eta(x)\sqrt{\frac{2-c^2-2\eta(x)}{1-2\ka+2\ka\eta(x)}}dx.
	\end{align}
Performing the change of variable $y=\eta(x)$, the integrability problem in \eqref{eq:intcriti} simplifies to the integrability of $s\mapsto(\sqrt{s})^{-1}$ near 0. Therefore, $\eta_{c,\ka}\in H^1((-\infty,\ga)\cup((0,\infty))$ and $u'_{c,\ka}\in L^2((-\infty,\ga)\cup((0,\infty))$ follows.
\end{proof}

At this stage, we could prove Corollary~\ref{cor:cuspon}, which provides a weak solution to \eqref{TWc}. However, we postpone its proof to Subsection~\ref{subsec:buildw},  where we establish a more general result to glue local pointwise solutions to form a global weak solution to \eqref{TWc} (see  Lemma~\ref{lem:buildw}).
\subsubsection*{Properties of solution with two or more singular points}
We are now interested in the behavior of a singular solution $u_{c,\ka}\in\boX(\R)$ in the interval $(\ga_{c,\ka},\gb_{c,\ka})$. We assume that this interval is nonempty, i.e.\  that $\card(\Gamma(u_{c,\ka}))\geq2.$ 
If $(c,\ka)\in\tilde{\D}$ so that $-\infty<\ga_{c,\ka}<\gb_{c,\ka}<\infty$, then from the variations of $\eta_{c,\ka}$ in $(-\infty,\ga_{c,\ka})\cup(\gb_{c,\ka},\infty)$ (see the proof of Theorem~\ref{thm:classifcuspon}) we have \begin{equation}\label{eq:Zinclu}
    \mathcal{Z}(u_{c,\ka})\subset(\ga_{c,\ka},\gb_{c,\ka}),\quad\text{ and }\quad-\infty<\ga_{c,\ka}^0<x_0<\gb_{c,\ka}^0<\infty,
    \end{equation}
   where $\ga_{c,\ka}^0$ and $\gb_{c,\ka}^0$ are the closest singular points to $x_0\in\mathcal{Z}(u_{c,\ka})$ defined in \eqref{def:ab0}. Notice that we omit the dependence of  $\ga_{c,\ka}^0$ and $\gb_{c,\ka}^0$ on $x_0$ for notational simplicity. 
The following lemma establishes that if $\ka=1/2$, then $\ga_{c,1/2}=-\infty$ and $\gb_{c,1/2}=\infty$, so that there are infinitely many points in $\Gamma(u_{c,1/2})$ and \eqref{eq:Zinclu} still holds unless $u_{c,1/2}$ is trivial. 

 Whether the set $\Gamma(u_{c,1/2})$ contains countably or uncountably many points is nontrivial and out of the scope of this work. For instance, a solution $u_{c,1/2}\in\boX(\R)$ consisting of countably many bubbles with a smaller and smaller amplitude, and such that the set $\Gamma(u_{c,1/2})$ countable, may exist.
  
\begin{lemma}\label{lem:Z1/2}
	Let $(c,\ka)\in\boC$. If $u\in\boX(\R)$ satisfies \eqref{TWc} with $\ka=1/2$, then $\ga_{c,1/2}=-\infty$ and $\gb_{c,1/2}= \infty$. In particular $|u(x)|=1$ for an infinite number of $x\in\R$.
\end{lemma}  
\begin{proof}
	If $\Gamma(u)=\emptyset$, then from Lemma~\ref{lem:omegaregu}, we have $u\in\boC^2(\R)$, and Proposition~\ref{prop:k1/2} yields that $u$ is the trivial solution. Now assume that $\Gamma(u)\ne\emptyset$ so that $\ga_{c,1/2}\leq\gb_{c,1/2}$. By contradiction, if $-\infty<\ga_{c,1/2}\leq\gb_{c,1/2}<\infty$ , then $\eta=1-\abs{u}^2$ satisfies \eqref{eta2l}--\eqref{eta1l}. In view of \eqref{eq:limeta}, the definition of $\gb_{c,1/2}$ implies that $\eta$ must reach a nonzero extremum  at some $x_0>\gb_{c,1/2}$ . Hence, ideas along the same lines as in Proposition~\ref{prop:k1/2} yield $\eta(x)=(1-c^2/2)\cos^2((x-x_0)/\sqrt{2})$, for all $x>\gb_{c,1/2}$; in particular, $\eta(x_0+\pi/\sqrt{2})=0$. This implies that $\gb\geq x_0+\pi\sqrt{2}$, which is absurd. This proves that $\gb_{c,1/2}= \infty$, and the proof of $\ga_{c,1/2}=-\infty$ is analogous.
\end{proof}

In the fashion of Proposition~\ref{prop:eqeta}, we obtain near any local extremum $x_0\in\mathcal{Z}(u_{c,\ka})$ a system of equations similar to \eqref{eta2}--\eqref{eta1} satisfied by $\eta=1-\abs{u_{c,\ka}}^2$ in $(\ga_{c,\ka}^0,\gb_{c,\ka}^0)$, making the analysis tractable.
\begin{lemma}\label{lem:buleq}
Let $(c,\ka)\in\tilde{\D}\cup\boC$. Consider  $u=u_{c,\ka}\in\boX(\R)$ a nontrivial solution to \eqref{TWc} with  $\card(\Gamma(u_{c,\ka}))\geq2$, so that there exists $x_0\in\mathcal{Z}(u_{c,\ka})$ satisfying   $-\infty<\ga^0_{c,\ka}<x_0<\gb^0_{c,\ka}< \infty$. Then, denoting $\eta_0=\eta(x_0)$ and $K_0=|u'(x_0)|^2$, we have $u\in\boC^\infty((\ga^0_{c,\ka},\gb^0_{c,\ka}))$,  and there exists $K_1\in\R$ such that $\eta=1-\abs{u}^2$ satisfies in $(\ga^0_{c,\ka},\gb^0_{c,\ka})$
	\begin{align}
		2(1-2\ka+2\ka\eta)\eta''+2\ka(\eta')^2&=P(\eta)+(\eta-\eta_0)P'(\eta)\label{eta2b},\\
		\label{eta1b}
		(1-2\ka+2\ka\eta)(\eta')^2&=(\eta-\eta_0)P(\eta),
	\end{align}
 where $P(y)=-2y^2+(2-c^2-2\eta_0)y+(2-c^2)\eta_0-4K_0-4cK_1$, for all $y\in\R$.
\end{lemma}
\begin{proof}
	 By simplicity, we fix $x_0=0$ and
	 we adapt arguments from Proposition~\ref{prop:eqeta} to obtain \eqref{eta2b}--\eqref{eta1b}, taking this time $R_n=x_0=0$, for all $n\in\N$. Writing $u_{c,\ka}=u_1+iu_2$, we obtain similarly to \eqref{eq:phase1}
	 \begin{equation}\label{eq:polairebu}
	 	(u_1u_2'-u_1'u_2)(x)=\frac{c}{2}\eta(x)+K_1,\quad\text{ in }(\ga_{c,\ka}^0,\gb_{c,\ka}^0),
	 \end{equation}
	 where $K_1=(u_1u_2'-u_1'u_2)(0)-c\eta_0/2$. We also get in the fashion of \eqref{eq:quadratic},
	 \begin{align}
	 	2|u'|^2&=\eta^2+\ka(\eta')^2-\eta_0^2+2K_0,\quad\text{ in }(\ga_{c,\ka}^0,\gb_{c,\ka}^0).
	 \end{align} 
	 From $u_1u_1''+u_2u_2''=c(u_1u_2'-u_1'u_2)-\abs{u}^2(\eta+\ka\eta'')$ and $\eta''=-2(\abs{u'}^2+u_1u_1''+u_2u_2'')$, we recover
	 $\eta''=-(\eta^2+\ka(\eta')^2-\eta_0^2+2K_0 )-c^2\eta-2cK_1+2\abs{u}^2(\eta+\ka\eta'')$, which can be recast as 
  \begin{align}\label{eq:eta2b}
      (1-2\ka+2\ka\eta)\eta''+\ka(\eta')^2&=-3\eta^2 +(2-c^2)\eta+\eta_0^2-2K_0-2cK_1\quad\text{ in }(\ga_{c,\ka}^0,\gb_{c,\ka}^0).
  \end{align}
 Multiplying \eqref{eq:eta2b} by $\eta'$ and integrating from $0$ to $x$ we get, using \eqref{eq:eta21},
  $$(1-2\ka+2\ka\eta)(\eta')^2=\eta^2(2-c^2-2\eta) +(2\eta_0^2-4K_0-4cK_1)\eta-(2-c^2)\eta_0^2+4K_0\eta_0+4cK_1\eta_0,$$
  which can be \emph{algebraically} factorized into \eqref{eta1b} due to the choice of integration bounds. From this factorization, we deduce that \eqref{eq:eta2b} can be recast as \eqref{eta2b} by differentiation of \eqref{eta1b} with respect to $x$.
\end{proof}
 Equations \eqref{eta2b}--\eqref{eta1b} imply constraints on $P(\eta)$: By the intermediate value theorem, $\eta$ reaches at least once every number in $I=(\eta_0,1-1/(2\ka))$. Thus, equation \eqref{eta1b} yields 
 \begin{align}\label{eq:eta0alg2}
	\frac{(y-\eta_0)P(y)}{1-2\ka+2\ka y}\geq0,\quad\text{ for all }y\in I.
\end{align}
We are in a position to prove Theorem~\ref{thm:symmetry}. A consequence of this result is that the strict inequality in \eqref{condalgintro} holds.
\begin{proof}[Proof of Theorem~\ref{thm:symmetry}]
The symmetry property follows from ODEs arguments along the same lines as in Corollary~\ref{coro:propeta}. For the monotonicity of $\eta$ in $(x_0,\gb_{c,\ka}^0)$, we assume without loss of generality that $x_0=0$ and that $\eta_0<1-1/(2\ka)$. Suppose by contradiction that $\eta'(x)=0$ for some $x\in(0,\gb^0_{c,\ka})$ and let $x_1=\inf\{x>0:\eta'(x)=0\}$. If $x_1=0$, then there is an infinite sequence $(x_n)$ of zeros of $\eta'$ tending to 0. This yields $\eta''(0)=0$, further implying that $P(\eta_0)=0$ so that $\eta_0$ is an equilibrium of equation \eqref{eta2b}, a contradiction to $\eta(\gb_{c,\ka}^0)=1-1/(2\ka)$. Therefore, $x_1>0$ and we necessarily have $\eta(x_1)\ne\eta_0$, otherwise Rolle's theorem provides a contradiction to the minimality of $x_1$. Hence, evaluating equation \eqref{eta1b} at $x_1$, we get $P(\eta(x_1))=0$. If $\eta(x_1)$ is a zero of $P$ of multiplicity 2, notice that $\eta(x_1)$ is then an equilibrium point of equation \eqref{eta2b}. Therefore, we can apply ODE arguments along the same lines as in Corollary~\ref{coro:propeta} to deduce that $\eta\equiv\eta(x_1)$ in $(\ga^0_{c,\ka},\gb^0_{c,\ka})$, contradicting that $\eta_0\ne\eta(x_1)$. Thus $\eta(x_1)$ is a zero of $P$ of multiplicity 1, and we denote by $y_1$ the other root of $P$. By contradiction, if $\eta_0$ is a local maximum of $\eta$, then by definition of $x_1$, we have  $\eta'<0$ in $(0,x_1)$. In particular $(\eta-\eta_0)<0$ and $1-2\ka+2\ka\eta<1-2\ka+2\ka\eta_0<0$ in $(0,x_1)$, therefore the condition \eqref{eq:eta0alg2} yields $P(\eta)\geq0$ in $(0,x_1)$. We deduce by analyzing the sign of $P$ that $y_1>\eta(x)>\eta(x_1)$, for all $x\in[0,x_1)$. Notice that the inequality $y_1>\eta$ is strict because otherwise, $\eta_0$ would be a root of multiplicity 2 in the right-hand side of \eqref{eta1b}, thus an equilibrium point of \eqref{eta2b}, which contradicts that $\eta(\gb_{c,\ka}^0)\ne\eta_0$.

Since $\eta(\gb^0_{c,\ka})=1-1/(2\ka)>\eta_0$, using the intermediate value theorem, there exists $x_2\in(x_1,\gb_{c,\ka}^0)$ such that $\eta_0<\eta(x_2)<y_1$ so that $P(\eta(x_2))>0$. Since $x_2<\gb_{c,\ka}^0$, we have $1-2\ka+2\ka\eta(x_2)<0$, therefore \eqref{eq:eta0alg2} does not hold for $y=\eta(x_2)$. We deduce that $\eta_0$ is not a local maximum of $\eta$.
	Then, we can assume that $\eta_0$ is a local minimum. Analogously to the previous case, we have $\eta'>0$ and $\eta-\eta_0>0$ in $(0,x_1)$. By the definition of $\gb_{c,\ka}^0$, we have  $\eta(x_1)<1-1/(2\ka)$, and using the intermediate value theorem, we can find $x_2\in(x_1,\gb_{c,\ka}^0)$ satisfying $\eta(x_1)<\eta(x_2)<1-1/(2\ka)$. Since $\eta'>0$ in $(0,x_1)$, we obtain $1-2\ka+2\ka\eta<0$ in $(0,x_1)$. Condition \eqref{eq:eta0alg2} implies that $P(\eta(x))\leq0$, for all $x\in(0,x_1)$, and $P(\eta(x_2))\leq0$, therefore the analysis of sign of $P$ gives $\eta(x)\leq\eta(x_1)<y_1\leq\eta(x_2)$. Once again, by the intermediate value theorem, there exists $x_1\leq x_3\leq x_2$ satisfying $\eta(x_1)<\eta(x_3)<y_1$. 
This relation implies that $\eta(x_3)-\eta_0>0$, and, $\eta(x_3)<y_1\leq\eta(x_2)$ yields $1-2\ka+2\ka\eta(x_3)<0$. Since $P(\eta(x_3))>0$, condition \eqref{eq:eta0alg2} is not fulfilled for $y=\eta(x_3)$ yielding a contradiction. We conclude that $\eta'\ne0$ in $(0,\gb_{c,\ka}^0)$. By the mean value theorem, we deduce that $\eta'>0$ in $(0,\gb_{c,\ka}^0)$. 
 
 In the case  $\eta_0>1-1/(2\ka)$, the result follows using the same ideas.
\end{proof} 
By the monotonicity of $\eta_{c,\ka}$, we can check that if $u(x_0)=0$, then $|u_{c,\ka}|>0$ in $(x_0,\gb^0_{c,\ka})$, so that, using Remark~\ref{rem:localeq}, we can lift  $u_{c,\ka}$ in $(x_0,\gb^0_{c,\ka})$  as
\begin{equation}\label{eq:polaireloc}
   u_{c,\ka}=\sqrt{1-\eta_{c,\ka}}e^{i\theta_{c,\ka}}\quad\text{ with }\quad\theta_{c,\ka}'=\frac{c\eta_{c,\ka}}{2(1-\eta)}+K, 
\end{equation}
for some $K\in\R$.  We will use this fact in the
proof of Proposition~\ref{prop:nonvanishsingu}, as follows.
\begin{proof}[Proof of Proposition~\ref{prop:nonvanishsingu}]
    Assume by contradiction that  $u_{c,\ka}(0)=0$ and $c>0$. If $\card(\Gamma(u_{c,\ka}))<2$, this contradicts the explicit formula of $u_{c,\ka}$ given in Theorems~\ref{thm:classiftwregu} and~\ref{thm:classifcuspon}. Hence,  $\card(\Gamma(u_{c,\ka}))\geq2$ and we infer that $0\in\mathcal{Z}(u_{c,\ka})$. We deduce that \eqref{eq:polaireloc} holds, with $\eta=\eta_{c,\ka}$ solution to \eqref{eta2b}--\eqref{eta1b} in $(0,\gb_{c,\ka}^0)$. The contradiction comes from the fact that $u_{c,\ka}'$  cannot be square integrable near 0.
    Indeed, from \eqref{eq:polaireloc},
    $$u_{c,\ka}'=e^{i\theta}\Big(\frac{\eta'}{2\sqrt{1-\eta}}+i\frac{c\eta}{2\sqrt{1-\eta}}+iK\sqrt{1-\eta}\Big),\quad\text{ in }(0,\gb_{c,\ka}^0).$$
    Since $u_{c,\ka}'\in L^2(\R)$, we deduce that $c\eta/(2\sqrt{1-\eta})+K\sqrt{1-\eta}\in L^2((0,\gb_{c,\ka}^0))$, which implies that 
    \begin{equation}\label{eq:vanishl2}
    \int_0^{\gb_{c,\ka}^0}\frac{c^2\eta^2}{4(1-\eta)}<\infty.
    \end{equation}
    To compute \eqref{eq:vanishl2}, we proceed as in \eqref{eq:intcriti}. Precisely, using \eqref{eta1b} and inequality \eqref{condalgintro}, we have
    \begin{equation}\label{eq:firstintbu}
    \eta'(x)\sqrt{\frac{1-2\ka+2\ka\eta}{(\eta-1)P(\eta)}}=-1,
    \end{equation}
     so plugging  \eqref{eq:firstintbu} into \eqref{eq:vanishl2}, and performing the change of variable $y=\eta(x)$, we obtain
    $$\int_{1-1/(2\ka)}^1\frac{cy^2}{4(1-y)}\sqrt{\frac{1-2\ka+2\ka y}{(1-y)P(y)}}<\infty.$$
    However, this is a contradiction with the non-integrability of $(1-y)^{-3/2}$ near 1. We conclude that if $c>0$, then $\inf_{x\in\R}|u_{c,\ka}(x)|>0$. 
     
     Now suppose that $c=0$ and $u_{0,\ka}(0)=0$. We observe a weaker correlation between the two real equations in \eqref{eq:systtw}, indeed, if $u_{0,\ka}'(0)=re^{i\phi}$ for some $r\geq0$ and $\phi\in\R$, then the Cauchy--Lipschitz theorem yields that $\tilde{u}=u_{0,\ka}e^{-i\phi}$ must coincide with the local real solution to the simpler problem \eqref{eq:pbreel} with initial condition $(0,r)$. Thus $\Tilde{u}(x)\in\R$, for all $x\in(-R,R)$, for some $R>0$. Assume by contradiction that there exists $\tilde x>0$ such that $\Im(\Tilde{u})(\tilde x)\ne0$  and let $x_0=\sup\{\tau>0:\Im(\Tilde{u})=0,\text{ in }(0,\tau)\}$. By continuity, we infer that $x_0\leq \tilde x$ and that $\Im\Tilde{u}(x_0)=0$. If $\Tilde{u}(x_0)=0$, then the latter argument applies so that there exists $\phi_0\in\R$ such that $e^{-i\phi_0}\Tilde{u}\in\R$ in $(x_0-R_0,x_0+R_0)$, for some $R_0>0$. But then, since $\Tilde{u}(x)\in\R$, for all $0<x<x_0$, we conclude that $\phi_0=2l\pi$ for some $l\in\N$, contradicting the definition of $x_0$. We deduce that  $\Tilde{u}(x_0)\ne0$ so that we can lift $\Tilde{u}=\rho e^{i\theta}$ in a neighborhood of $x_0$ with $\theta'=K$ for some $K\in\R$ by Remark~\ref{rem:localeq}.
     Once again, since $\Tilde{u}(x)\in\R$ for all $0<x<x_0$, the constant $K$ must be 0, which further implies that $\theta$ is constant in the neighborhood of $x_0$. Therefore, $\Tilde{u}$ remains real-valued in the neighborhood of $x_0$, contradicting its definition. Consequently, there is no $\tilde{x}>0$ such that   $\Im\Tilde{u}(\Tilde{x})\ne0$. The same holds for all $x<0$ with similar arguments, which means that $\Tilde{u}$ is real-valued. If $c=0$, and $u_{c,\ka}\in\boN\boX(\R)$, the proof is easier since we get from Corollary~\ref{coro:eqpolaire} that we can lift $u_{c,\ka}=\rho e^{i\theta}$ with $\theta'\equiv0$ in $\R$.
    \end{proof}
    In the setting of Lemma~\ref{lem:buleq}, the next result establishes that the constant $K_1$ in \eqref{eta2b}--\eqref{eta1b} must be zero.
   \begin{lemma}\label{lem:K10}
       Let $(c,\ka)\in\tilde{\D}\cup\boC$ and $u=u_{c,\ka}\in\boX(\R)$ be a solution to \eqref{TWc} satisfying $\card(\Gamma(u_{c,\ka}))\geq2$ so that there exists $x_0\in\mathcal{Z}(u_{c,\ka})$ and Lemma~\ref{lem:buleq} applies.
       Then  $\eta=1-\abs{u}^2$ satisfies \eqref{eta2b}--\eqref{eta1b} near $x_0$ and the constant $K_1$ in those equations must be 0.
   \end{lemma}
   \begin{proof}
        If $u\in\boN\boX(\R)$, then we can lift $u=\sqrt{1-\eta}e^{i\theta}$ with $\theta'=c\eta/(2-2\eta)$ by Corollary~\ref{coro:eqpolaire}. Clearly the expression for $K_1=(u_1u_2'-u_1'u_2)(x_0)-c\eta(x_0)/2$ can be recast as \begin{equation}\label{eq:K10}
            K_1=(1-\eta(x_0))\theta'(x_0)-c\eta(x_0)/2,
        \end{equation}
       so that $K_1=0.$   
       On the other hand, if $u$ vanishes at some point, then by Proposition~\ref{prop:nonvanishsingu}, we must have $c=0$, and we can find $\phi\in\R$ so that $e^{i\phi}u(x)\in\R$ for all $x\in\R$. 
       If $u(x_0)=0$ then we get $K_1=0,$ whereas if $u(x_0)\ne0$ we can lift $u=\sqrt{1-\eta}e^{i\theta}$ near $x_0$, so that we can rewrite $K_1$ with \eqref{eq:K10} and $c=0$. We must have $\theta'=0$ in the neighborhood of $x_0$ otherwise contradicting the fact that $e^{i\phi}u(x)\in\R$ for all $x\in\R.$ This lets us conclude that $K_1=0$  
   \end{proof} 
   We deduce that \eqref{condalgintro} simplifies to
   \begin{align}\label{condalgintro2}
-2y^2+(2-c^2-2\eta_0)y+(2-c^2)\eta_0-4K_0<0,\text{ for all }y\in[\eta_0,1-1/(2\ka)).
\end{align}
 We now prove Proposition~\ref{prop:constant} stating that unless $\ka=1/2$, $u_{c,\ka}$  cannot be of constant intensity in subintervals $(a,b)\subset(\ga_{c,\ka},\gb_{c,\ka})$. Indeed,
assume that $|u_{c,\ka}|(x)=1/(2\ka)$ for all $x\in(a,b)$, then $u_{c,\ka}\ne0$ in $(a,b)$ so that we can apply the same reasoning as in Remark~\ref{rem:localeq} to recover that $u_{c,\ka}=\rho e^{i\theta}$ in $(a,b)$ with $\theta'\equiv\dot{\theta_0}\in\R$ and $\rho\equiv1/\sqrt{2\ka}$ satisfying
 \begin{equation}\label{eq:condcst}
 (\dot{\theta_0})^2+c\dot{\theta_0}+\rho^2-1=0.
 \end{equation}
Although one can build local solutions in $(a,b)$, choosing adequate $\dot{\theta_0}$, we show that global solutions add constraints on $\dot{\theta_0}$ preventing this behavior.
 \begin{proof}[Proof of Proposition~\ref{prop:constant}]
     On one hand, if $u_{c,\ka}\in\boN\boX(\R)$, then we can write $u=\rho e^{i\theta}$ with $\theta$ and $\rho$ satisfying \eqref{TW:weaksystpol1}--\eqref{TW:weaksystpol2}. Setting $\rho\equiv1/\sqrt{2\ka}$ in \eqref{TW:weaksystpol1} yields $\dot{\theta_0}=c(\ka-1/2)$, which replaced in \eqref{eq:condcst} implies that
     $(\ka-1/2)(c^2(\ka-1/2)+c^2-1/\ka)=0$. Since $(c^2(\ka-1/2)+c^2-1/\ka)\ne0$, for all $(c,\ka)\in\tilde{\D}$, we deduce that $\ka=1/2$.
     
     On the other hand, if $u_{c,\ka}(x_0)=0$ for some $x_0\in\R$, we deduce by Proposition~\ref{prop:nonvanishsingu} that $c=0$ so that  $\ka\leq1/2$. Indeed, if $\ka>1/2$ we deduce that $|u(x)|^2=1$ for all $x\in\R$ by Theorem~\ref{thm:non-singular-sol}, in particular $|u(x)|\ne1/(2\ka)$ for all $x\in\R.$
 Thus, equation \eqref{eq:condcst} writes
    $ (\dot{\theta_0})^2+1/(2\ka)-1=0$, which further implies that $\ka=1/2$ and $\dot{\theta_0}=0$.    
     \end{proof}
    
    \subsection{Construction of singular solutions}\label{subsec:buildw}
 We are now in a position to show the existence of singular solutions. 
We start with the case $\card(\Gamma(u_{c,\ka}))=1$, i.e. the cuspons 
in Corollary~\ref{cor:cuspon}, constructed using the functions in Theorem~\ref{thm:classifcuspon}. To show that the cuspons defined by parts is indeed a global weak solution to \eqref{TWc},  we use the following lemma, which provides sufficient conditions for gluing two strong solutions into a weak solution.


\begin{lemma}\label{lem:buildw}
	Let $c\geq 0$ and $k\in\R\setminus\{0\},$
	 and let $I_1, I_2$ be two nonempty disjoint open intervals such that $I_1$ and $I_2$ share one bound $a\in\R$, i.e.\  $\mathrm{cl}(I_1)\cap\mathrm{cl}(I_2)=\{a\}$, where $\mathrm{cl}$ is the usual closure  on $\R$, and set $I =I_1\cup  I_2\cup\{a\}$.
  Suppose that $u\in \boC(I)\cap \boC^2(I_1\cup I_2)$ is a solution to \eqref{TWc} on $I_1\cup I_2$, and assume that $|u(a)|^2=1/(2\ka)$, so that  we can write $u=\sqrt{1-\eta}e^{i\theta}$ with $\eta<1$ in $(a-2\delta,a+2\delta)$ for some $\delta>0$. If $\theta'$ has a well-defined limit as $x\to a$, and if the function $\eta$ belongs to  $H^1((a-\delta,a+\delta))$, with 
		\begin{equation}
		\label{cond:weak}
  \lim_{x\to a }\eta'(x)(1-2\ka+2\ka\eta(x))=0,
	\end{equation}
then $u$ is a weak solution to \eqref{TWc} on $I$, i.e.\ $u$ satisfies \eqref{TW:weak} for every $\phi\in\boC_0^\infty(I;\C)$.
\end{lemma}
\begin{remark}
It is enough to assume that the limit in \eqref{cond:weak} exists, but this limit is going to be equal to zero in all our applications.
\end{remark}
\begin{remark}
    This result can be adapted to handle strong solution in $\bigcup_{j\in J} I_j$ where $(I_j)_{j\in J}$ is a finite family of open intervals touching at one of their bounds and to recover a weak solution in $\mathrm{cl}(\bigcup_{j\in J} I_j)$
\end{remark}
\begin{proof}
Since we are assuming that $u$ lies in  $\boC(I)\cap \boC^2(I_1\cup I_2)$, we deduce that $\theta$ is also is continuous on $I$. In addition, 
since $\theta'$ has a limit as $x\to a$, we infer that $\theta'$ is bounded in 
$(a-\delta,a+\delta)$. Therefore, using that 
\begin{equation}
		\label{weak-der2}
		u'(x)=\frac{e^{i\theta(x)} }{2\sqrt{1-\eta(x)}} \big( -\eta'(x)+i (1-\eta) \theta'(x)\big),
	\end{equation}
we conclude that  $u'\in L^2((a-\delta,a+\delta)).$
 
To check that $u$ is a weak solution, let us take without loss of generality $a=0$, $I_1=(r_1,0)$ and $I_2=(0,r_2)$, for some $-\infty\leq r_1<0<r_2\leq \infty$, $0<\varepsilon<\delta$ small, and   $I_\ve =I\setminus (-\ve,\ve)$,
	so that $u\in \boC^2(\mathrm{cl}(I_\ve))$. Taking  $\varphi\in \boC_0^\infty(I;\C)$,  and integrating by parts, we deduce that
	\begin{align}
		\label{ipp-weak}
		\Re\int_{I_\ve}(icu'+u\eta)\bar\varphi-u'\bar\varphi'+2\ka\inner{u}{u'}(u\bar{\varphi})'=&\Re\int_{I_\ve}(icu'+u\eta+u''-2\ka u(|u'|^2+\inner{u}{u''}))\bar{\varphi}\notag\\
		&+\Re\Big[-u'\bar\varphi+ 2\ka\inner{u}{u'}(u\bar{\varphi})\Big]_{x=\ve}^{x=-\ve}
	\end{align}
	Since $u \in \dot{H}^1((-\ve,\ve))\cap C(I)$, we can use the dominated convergence theorem to conclude that the integral in the left-hand side of \eqref{ipp-weak} converges to the integral in $I$, as $\ve\to 0$.
	On the other hand, the integral in the right-hand side of \eqref{ipp-weak} vanishes because $u$ is a pointwise solution to \eqref{TWc} on $I_\ve$.
	
	We verify now that the last term in \eqref{ipp-weak}  goes to zero, as $\ve\to 0$. Using \eqref{weak-der2}, we have for all $x\in[-\ve,\ve]$ with $x\ne0$
	\begin{equation*}
		-u'(x)+ 2\ka\inner{u(x)}{u'(x)}u(x)=\frac{ e^{i\theta} }{ 2\sqrt{1-\eta(x)} }
		\Big( \eta'(x) (1-2\ka-2\ka\eta(x))-i (1-\eta(x))\theta'(x)\Big).
	\end{equation*}
	Finally, invoking \eqref{cond:weak}, the continuity of $\eta$ and $\theta$ and the fact that $\theta'$ has a limit as $x\to 0$, we conclude that the last term in \eqref{ipp-weak}  goes to zero, as $\ve\to 0$, which completes the proof.
\end{proof}
Applying this result we deduce Corollary~\ref{cor:cuspon}.
\begin{proof}[Proof of Corollary~\ref{cor:cuspon}]
 Since $\Gamma(u_{c,\ka})=\{0\}$, we deduce from Theorem~\ref{thm:classifcuspon}
 that $\ga_{c,\ka}=\gb_{c,\ka}=0$, and that, up to phase shift,
  $u_{c,\ka}$ must be given by $u_{c,\kappa}=\sqrt{1-\eta_{c,\ka}}e^{i\theta_{c,\ka}},$
where  $\eta_{c,\ka}$ is the function in \eqref{eq:etaw}, 
and $\theta_{c,\kappa}(x)=\frac{c}2\int_0^x\frac{\eta_{c,\kappa}(y)}{1-\eta_{c,\kappa}(y)}$.
 In the course of the proof of Theorem~\ref{thm:classifcuspon}, we established that $u_{c,\ka}\in\boX(\R)$,
 thus it remains to verify that $u_{c,\ka}$ is indeed a weak solution. Since $\eta_{c,\ka}$ satisfies \eqref{eta1l} in $\R\backslash\{0\}$, we can check that \eqref{cond:weak} holds multiplying \eqref{eta1l} by $1-2\ka+2\ka\eta_{c,\ka}$, and taking the square root of both sides of the equation.
	It is straightforward to check that $\theta_{c,\ka}$ is well-defined and that $u\in\boC(\R)\cap\boC^2(\R\backslash\{0\})$ satisfies \eqref{TWc} in $\R\backslash\{0\}$ using ideas similar to the ones in Proposition~\ref{prop:exiu}.
 Finally, 
	$\theta_{c,\ka}'$ tends to $c\ka(1-2\ka)$, as $x\to0$, therefore we can apply Lemma~\ref{lem:buildw} to conclude.
\end{proof}
We focus now on the construction of solutions satisfying $\card(\Gamma(u_{c,\ka}))\geq 2$. Thus,  
  we need to find solutions to \eqref{eta2b}--\eqref{eta1b} satisfying the boundary condition $\eta(\gb^0_{c,\ka})=1-1/(2\ka)$, for some $\gb_{c,\ka}^0\in\R$. We show now that such solutions exist when the numbers $\eta_0$ and $K_0$ satisfy \eqref{condalgintro2}.
\begin{proposition}\label{prop:exibu}
    Let $(c,\ka)\in\tilde{\D}\cup\boC$. For any   $\eta_0\leq1$, with $\eta_0\ne1-1/(2\ka)$, and any $K_0\geq0$ satisfying  \eqref{condalgintro2},  there exist a nonempty interval $(\ga^0_{c,\ka},\gb^0_{c,\ka})$ containing 0, and a unique solution $\eta$ to \eqref{eta2bintro}--\eqref{eta1bintro}, with $K_1=0$, on $(\ga^0_{c,\ka},\gb^0_{c,\ka})$, such that $\eta(0)=\eta_0$. Moreover $\ga^0_{c,\ka}=-\gb^0_{c,\ka}$, the function $\eta$ is even in $(\ga^0_{c,\ka},\gb^0_{c,\ka})$, and monotonous in $(0,\gb^0_{c,\ka})$. Also $\eta$ belongs to $H^1(\ga_{c,\ka}^0,\gb_{c,\ka}^0)) \cap \boC([\ga_{c,\ka}^0,\gb_{c,\ka}^0])$,  
    with 
    \begin{equation}
    \label{dem:limits}
        \eta(\ga^0_{c,\ka})=\eta(\gb^0_{c,\ka})=1-1/(2\ka).
    \end{equation}
    \end{proposition}
 \begin{proof}
 Let us assume first that $\eta_0<1-1/(2\ka)$.
 We look for a solution $\eta$ in $(0,\gb_{c,\ka}^0)$ with $\eta(0)=\eta_0$  and $\eta((\gb^0_{c,\ka})^-)=1-1/(2\ka)$, with  $\gb_{c,\ka}^0>0$ to be determined.
 Let $y\in[\eta_0,1-1/(2\ka)],$ and 
 \begin{equation*}
F(y)=\int_{\eta_0}^yf(s)ds,\text{ where }       f(s)=-\sqrt{\frac{1-2\ka+2\ka s}{(s-\eta_0)P(s)}},\text{ for all }s\in(\eta_0,1-1/(2\ka)).
 \end{equation*}
 Using condition \eqref{condalgintro2}, which can be recast as 
 $P(y)<0$ in $[\eta_0,1-1/(2\ka))$ with $P$ given by \eqref{def:P}, we see that $f$ is a well-defined smooth function in $(\eta_0,1-1/(2\ka)).$ Also, $f(s)={O}((s-1+1/(2\ka))^{-1/2})$, as $s\to1-1/(2\ka)$, and
 since $P(\eta_0)<0$, it follows that $f(s)={O}((s-\eta_0)^{-1/2})$ as $s\to\eta_0$, so that $f$ is locally integrable in $(\eta_0,1-1/(2\ka))$. We deduce that $F\in\boC^\infty((\eta_0,1-1/(2\ka)))$ is a well-defined function, which we can extend by continuity to $[\eta_0,1-1/(2\ka)]$, with $F(\eta_0)=0$ and $F(1-1/(2\ka)):=\gb_{c,\ka}^0$.
 
Since $F'(y)<0$, for all $y\in(\eta_0,1-1/(2\ka)),$ we deduce that
 $F$ is bijective,  and its inverse function $\boF=F^{-1}$ lies in $\boC([0,\gb_{c,\ka}^0])\cap\boC^\infty((0,\gb_{c,\ka}^0))$. In addition, $\boF$ can be extended by reflection to $[-\gb_{c,\ka}^0,\gb_{c,\ka}^0]$ as an even function, that we still denote by $\boF$. This extension satisfies $\mathrm{Im}(\boF)=[\eta_0,1-1/2\ka]$ with $\boF(0)=\eta_0$ and $\boF(\pm\gb_{c,\ka}^0)=1-1/(2\ka)$. Moreover, we can prove, using ideas similar to the ones in Lemma~\ref{lem:etafunctions2}, that $\boF\in\boC^\infty((-\gb_{c,\ka}^0,\gb_{c,\ka}^0))$. Letting $\eta=\boF$, we get $\eta'=1/f(\eta)$ in $(-\gb_{c,\ka}^0,\gb_{c,\ka}^0)$. We see that this imply that $\eta$ is not only a solution to \eqref{eta1b}, but also satisfies \eqref{eta2b} in $(-\gb_0,\gb_0)\backslash\{0\}$. Since $\eta\in\boC^\infty((-\gb_{c,\ka}^0,\gb_{c,\ka}^0))$, it must satisfy \eqref{eta2b} at $x=0$, by a continuity argument. It remains to check that $\eta\in H^1((\ga_{c,\ka}^0,\gb_{c,\ka}^0))$. Since $\eta$ is a continuous function, we just need to show that $\eta'$ is square integrable near $\gb_{c,\ka}^0$ to conclude. 
 Indeed, since
 \begin{equation}\label{eq:buh1}
     \int_0^{\gb_{c,\ka}}(\eta')^2=-\int_0^{\gb_{c,\ka}^0}\eta'\sqrt{\frac{(\eta-\eta_0)P(\eta)}{1-2\ka+2\ka\eta}},
 \end{equation}
 the change of variable $y=\eta$ in \eqref{eq:buh1}, yields
 \begin{align*}
    \int_0^{\gb_{c,\ka}}(\eta')^2\leq
    \Bigg|\int_{\eta_0}^{1-1/(2\ka)}\sqrt{\frac{(y-\eta_0)P(y)}{1-2\ka+2\ka y}}\Bigg|<\infty,
 \end{align*}
 where we used that $y\mapsto(1-2\ka+2\ka y)^{-1/2}$ is integrable near $1-1/(2\ka).$
 This concludes the existence part of the result.

 For the uniqueness, let $\check{\eta}$ be another solution in $(\ga_{c,\ka}^0,\gb_{c,\ka}^0)$ 
 such that $\check{\eta}(0)=\eta_0$. Then, by \eqref{eta1b} we obtain $(\check{\eta})'(0)=0$. Thus, it also satisfies \eqref{eta2b} with the same initial condition at $x=0$ as $\eta$. By the Cauchy--Lipschitz theorem, we deduce that $\check{\eta}(x)=\eta(x)$ for all $x\in(\ga_{c,\ka}^0,\gb_{c,\ka}^0)$
 The same ideas apply to treat the case $\eta_0>1-1/(2\ka)$. 
 \end{proof} 
 
Inspired by Corollary~\ref{prop:exiu}, we prove now that  getting a $C^2((\ga_{c,\ka}^0,\gb_{c,\ka}^0))$-solution to \eqref{TWc} in $(\ga_{c,\ka}^0,\gb_{c,\ka}^0)$ from $\eta$ solution to \eqref{eta2b}--\eqref{eta1b}, is equivalent to 
\begin{equation}
   \label{K0} 
   K_0=(c\eta_0)^2/(4-4\eta_0).
\end{equation} In this manner,  condition \eqref{condalgintro2} reduces to \eqref{condalgintro22}.
\begin{proposition}\label{prop:exibu2}
	Let $(c,\ka)\in\tilde{D}\cup\boC$. Assume that $\eta_0<1$, with $\eta_0\ne1-1/(2\ka)$, and $K_0\geq0$ fulfill \eqref{condalgintro2}. If $\gb_{c,\ka}^0>0$ and $\eta\in\boC^\infty((-\gb_{c,\ka}^0,\gb_{c,\ka}^0))$ are given by Proposition~\ref{prop:exibu}, then the function $u=\sqrt{1-\eta}e^{i\theta}$ defined in \eqref{localsol} satisfies \eqref{TWc} if and only if $K_0=c^2\eta_0^2/(4-4\eta_0)$. Moreover, this is, up to phase shift, the unique $\boC^2((-\gb_{c,\ka}^0,\gb_{c,\ka}^0))$-solution to \eqref{TWc} such that $1-|u|^2(0)=\eta_0$ and $\theta'(0)=c\eta_0/(2-2\eta_0)$.
\end{proposition}
\begin{proof}
	Using \eqref{weak-der2} together with the formula for $\theta'$ in \eqref{localsol}, we can compute the left-hand side of \eqref{TWc}.
	A direct verification yields that the imaginary part is zero. On the other hand, the real part writes
	\begin{align}\label{eq:LHSbu}
		-c\sqrt{1-\eta}\theta'-\frac{\eta''}{2\sqrt{1-\eta}}(1-2\ka+2\ka\eta)-\frac{(\eta')^2}{4(1-\eta)^{3/2}}-(\theta')^2\sqrt{1-\eta}-\eta\sqrt{1-\eta}.
	\end{align}
We can factorize $(-4(1-\eta)^{3/2})^{-1}$ in \eqref{eq:LHSbu} and use consecutively equations \eqref{eta2b}--\eqref{eta1b} to remove respectively the term in $\eta''$ and the terms in $(\eta')^2$. Then, we infer that, since $(1-\eta)$ does not vanish in $(-\gb_{c,\ka}^0,\gb_{c,\ka}^0)$, formula \eqref{eq:LHSbu} is zero if and only if
\begin{equation}\label{eq:poly3}
    2c^2\eta-c^2\eta^2+(1-\eta_0)P(\eta)+(\eta-\eta_0)(1-\eta)P'(\eta)-4\eta(1-\eta)^2=0,
\end{equation} where $P$ is given by \eqref{def:P} with $K_1=0$. 
In addition, when expanding all the terms in the left-hand side of \eqref{eq:poly3}, we see that \eqref{eq:poly3}
 reduces to the identity $(1-\eta_0)P(0)-\eta_0P'(0)=0$.
Thus, we check that \eqref{eq:LHSbu} is zero if and only if   
\eqref{K0} holds.

Let us prove the uniqueness. Assume that $\check{u}\in\boC^2((-\gb_{c,\ka}^0,\gb_{c,\ka}^0))$ is another solution that can be lifted near $0$, $\check{u}=\sqrt{1-\check{\eta}}\exp({i\check{\theta}})$,  with $\check{\eta}(0)=\eta_0$ and $\check{\theta}'(0)=c\eta_0/(2-2\eta_0)$. Then, by Lemma~\ref{lem:buleq} and the  condition \eqref{K0} for $K_0$, we deduce that $\eta$ and $\check{\eta}$ both satisfy \eqref{eta2b}--\eqref{eta1b}. From Proposition~\ref{prop:exibu}, we obtain $\check{\eta}=\eta$ in $(-\gb_{c,\ka}^0,\gb_{c,\ka}^0)$, thus we can lift $\check{u}$ in $(-\gb_{c,\ka}^0,\gb_{c,\ka}^0)$. Using the latter identity and Remark~\ref{rem:localeq}, (or equivalently using \eqref{eq:polairebu}), we check that $\check{\theta}'$ satisfies the same equation as $\theta'$ in \eqref{localsol} in $(\ga_{c,\ka}^0,\gb_{c,\ka}^0)$, therefore $\check{\theta}=\theta+\phi$, which completes the proof. 
\end{proof}
In the case  $c=0$, we can also obtain a local solution to \eqref{TWc} vanishing at the origin (so that  $\eta_0=1$), using ideas similar to the ones in the proof of Theorem~\ref{thm:classiftwregu}. 
\begin{proposition}\label{prop:exibu3}
	Let $c=0$ and $0<\ka\leq1/2$. Assume that $\eta_0=1$ and $K_0\geq0$ are real numbers
	so that \eqref{condalgintro2} holds. If $\gb_{0,\ka}^0>0$ and $\eta\in\boC^\infty((-\gb_{0,\ka}^0,\gb_{0,\ka}^0))$ are given by Proposition~\ref{prop:exibu}, then the real-valued function $u(x)=\pm\sqrt{1-\eta(x)}$ for all $\pm x\in[0,\gb_{0,\ka}^0)$ defined in \eqref{localsol0} is, up to phase shift, the unique $\boC^2((\gb_{0,\ka}^0,\gb_{0,\ka}^0))$-solution to \eqref{TWc} with $c=0$ such that $u(0)=0$ and $|u'(0)|^2=K_0$.
\end{proposition}
\begin{proof}
Using the Cauchy--Lipschitz theorem, define $\tilde{u}\in\boC^\infty((-R,R))$ to be the local real solution of equation \eqref{eq:pbreel}, with initial conditions $ \tilde{u}(0)=0$ and $\tilde{u}'(0)=\sqrt{K_0}$, for some $R>0$. Note that $\tilde{u}$ is a local solution to \eqref{TWc} with $c=0$ such that $0\in\mathcal{Z}(\tilde{u})$.

Then, following computations along the same lines as in Lemma~\ref{lem:buleq}, we infer that $\tilde{\eta}=1-\abs{\tilde{u}}^2$ satisfies \eqref{eta2b}--\eqref{eta1b} with $\eta_0=1$ and $K_0=(u'(0))^2$ (and $K_1=0$).
By Proposition~\ref{prop:exibu}, we obtain $\tilde{\eta}=\eta$ in $I=(-\gb_{0,\ka}^0,\gb_{0,\ka}^0)\cap(-R,R)$. We deduce that $\abs{\tilde{u}}=\sqrt{1-\eta}$ in $I$. Since $\tilde{u}'(0)>0$, we see that $\tilde{u}$ is negative and then becomes positive in the neighborhood of 0. Therefore, $\tilde{u}$ and $u$ coincide in $I$. This yields that $\tilde{u}$ and $\tilde{u}'$ are bounded and $\Tilde{\eta}$ does not reach $1-1/(2\ka)$ in $I$ if $R<\gb_{0,\ka}^0$. Thus, we can assume that the local solution $\tilde{u}$ exist in $(-R,R)$ with $R\geq\gb_{0,\ka}^0$,  which further implies that $u$ satisfies \eqref{TWc} in $(-\gb_{0,\ka}^0,\gb_{0,\ka}^0)$ with $c=0$.

To prove uniqueness,  suppose that $\check{u}$ is another solution in $(-\gb_{c,\ka}^0,\gb_{c,\ka}^0)$. Then by assumption, $\check{u}$ satisfies the initial value problem associated to \eqref{TWc} with $\check{u}(0)=0$ and $\check{u}'(0)=e^{i\phi}\sqrt{K_0}$, for some $\phi\in\R$. Hence $e^{-i\phi}\check{u}$ satisfies \eqref{TWc} with the same initial conditions as $u$, therefore $\check{u}=u$, by the Cauchy--Lipschitz theorem. 
\end{proof}
We can extend these local solutions to $\R$ by a cuspon-like solution outside  $(\ga_{c,\ka}^0,\gb_{c,\ka}^0)$ if $(c,\ka)\in\Tilde{\D}$, and by constants of modulus one if $\ka=1/2$. By Lemma~\ref{lem:buildw}, we recover a composite-wave (and a compacton if $\ka=1/2$) satisfying \eqref{TWc} weekly. For instance, this procedure yields the solutions in Proposition~\ref{prop:compos}.
To obtain Proposition~\ref{prop:compos}, we just need to show that \eqref{condalgintro22} holds for $\eta_0$ small enough
\begin{proof}[Proof of Proposition~\ref{prop:compos}]
Let $\eta_0<1-1/(2\ka)$ and $K_0=c^2\eta_0^2/(4-4\eta_0)$, then condition \eqref{condalgintro22} writes $P(y)<0$ for all $y\in(\eta_0,1-1/(2\ka))$ where $P$ is the quadratic polynomial in \eqref{def:P} with $K_1=0$. Computing the discriminant of $P(y)$ yields
\begin{equation}\label{eq:discricomposit}
    \Delta=(2-c^2-2\eta_0)^2+8\frac{(2-c^2-2\eta_0)\eta_0}{1-\eta_0}.
\end{equation}
We can check that $\Delta\to\infty$ as $\eta_0\to-\infty$, thus there exists $A_1(c,\ka)<0$ such that $\Delta>0$ for all $\eta_0<A_1$.
We show that $1-1/(2\ka)$ is smaller than the smallest root of $P$, $y_1=(2-c^2-2\eta_0-\sqrt{\Delta})/4$ for $\eta_0$ small enough. Indeed, $y_1=(2-c^2-2\eta_0)(1-\sqrt{1+8\eta_0/((2-c^2-2\eta_0)(1-\eta_0))})/4$ for $\eta_0<1-c^2/2$ and we can  check that $y_1\to1$ as $\eta_0\to-\infty$ so there exists $A(c,\ka)<A_1(c,\ka)$ such that $1-1/(2\ka)<y_1$  for all $\eta_0\leq A(c,\ka)$. This implies \eqref{condalgintro22}, so we can build the composite wave solution $u$ by extending the local solution in $(-\gb_{c,\ka}^0,\gb_{c,\ka}^0)$ of Proposition~\ref{prop:exibu2} with a cuspon like solution or a constant solution if $\ka=1/2$. The uniqueness follows from the uniqueness in Proposition~\ref{prop:exibu2} and from Theorem~\ref{thm:classifcuspon} ruling out the behavior at infinity of the solutions. Finally, since $P(1-1/(2\ka))<0$ we obtain $(1-\abs{u}^2)'\to\infty$ as $x\to\gb_{c,\ka}^0$  using \eqref{eta1bintro} so that $\boN(u)=\{-\gb_{c,\ka}^0,\gb_{c,\ka}^0\}$.
\end{proof}

Finally, the existence of explicit compactons is a simple consequence of Lemma~\ref{lem:buildw} and the explicit local solutions obtained in \eqref{prop:k1/2}. 
\begin{proof}[Proof of Proposition~\ref{prop:compacton}]
By construction $u_{c,1/2}^{(j)}$ is continuous $\R$, 
is a strong solution in $I_j$ by Proposition~\ref{prop:k1/2}, and is a constant function of modulus 1 outside $I_j$. Thus, it is simple to check that
$u_{c,1/2}^{(j)}$ is a global weak solution by invoking Lemma~\ref{lem:buildw}.
\end{proof}
This ends the classification in the case $\card(\mathcal{Z}(u_{c,\ka}))<\infty$. Moreover, the next result establishes that the classification of solutions with $2\leq\card(\Gamma(u_{c,\ka}))<\infty$ can be reduced to the cases already treated.
\begin{lemma}\label{lem:classifcard}
Let $(c,\ka)\in\Tilde{\D}\cup\boC$ and $u_{c,\ka}$ be a solution to \eqref{TWc}.
    If $2\leq\card(\Gamma(u_{c,\ka}))<\infty$, then $\card(\mathcal{Z}(u_{c,\ka}))=\card(\Gamma(u_{c,\ka}))-1$.
\end{lemma}
\begin{proof}
We show that between two elements of $\Gamma(u_{c,\ka})$ lies exactly one element of $\mathcal{Z}(u_{c,\ka})$.
    Let $a<b$ be two consecutive elements of $\Gamma(u_{c,\ka}).$ By Lemma~\ref{lem:omegaregu}, $u_{c,\ka}$ is smooth in $(a,b)$. Since the intensity profile $\eta_{c,\ka}=1-\abs{u_{c,\ka}}^2$ satisfies $\eta_{c,\ka}(a)=\eta_{c,\ka}(b)=1-1/(2\ka)$, we can apply  Rolle's theorem to deduce the existence of $x_0\in(a,b)$ such that $\eta'(x_0)=0$. Since $\Gamma(u_{c,\ka})\cap(a,b)=\emptyset,$ we deduce that $x_0\in\mathcal{Z}$ and that $a=\ga_{c,\ka}^0$ and $b=\gb_{c,\ka}^0$. Using the monotonicity of $\eta_{c,\ka}$ in $(a,b)$ established in Theorem~\ref{thm:symmetry}, the conclusion follows.
\end{proof}
In conclusion, we classified all the solutions with $\card(\Gamma(u_{c,\ka}))<\infty$. On the other hand, 
by Proposition~\ref{prop:constant}, if $\ka\ne1/2$ then $\Gamma(u_{c,\ka})$, is a closed bounded set of empty interior. Hence, it is possible that $\Gamma(u_{c,\ka})$ contains a countable number of points (for instance, $\Gamma(u_{c,\ka})=\{1/n:n\geq1\}$), or even uncountable many points (for instance, $\Gamma(u_{c,\ka})$ is the Cantor set). These cases are beyond the scope of this paper, and we refer to \cite{lenells2005traveling} treating these kinds of considerations in the context of  the 
Camassa--Holm equation. Similarly,  the question of whether  $\Gamma(u_{c,1/2})$ can be any closed subset of $\R$ is still open. 
Also, as explained in the proof of Proposition~\ref{prop:exibu}, in general, the value of $\gb_{c,\ka}^0$ with respect to the parameters is not explicit and relies on our ability to compute the integral from $\eta_0$ to $1-1/(2\ka)$ of a function $f_{c,\ka,\eta_0}$. 
It is an open problem to determine if there is a one-to-one correspondence between the maximal amplitude $\eta_0$ and the half-length $\gb_{c,\ka}^0$ of the interval of existence.


We end this section with a comment on the regularity and decay of the cuspon solutions to mirror Proposition~\ref{prop:reguparam1} about the dark solitons. Let $u_{c,\ka}$ be the cuspon solution given by Corollary~\ref{cor:cuspon}, then, from the regularity of the flow of ODEs, for all $x\ne0$, and all $j\in\N $, $\partial_{x}^j\eta(\cdot,\cdot,x)$ is smooth for all $(c,\ka)\in \tilde{\D}$ where $\eta(c,\ka,x)=\eta_{c,\ka}(x)$ (the derivatives in the $c$ direction are in the sense of Dini for $c=\sqrt{2},$ and $\ka\ne1/2$). At $x=0$, $\eta_{c,\ka}(0)=1-1/(2\ka)$ so $\eta(\cdot,\cdot,0)$ is also smooth in $\tilde{\D}$, however, $\eta_{c,\ka}(\cdot)$ is not continuously differentiable in 0.
To summarize, we have $\eta\in\boC^\infty(\tilde{\D}\times\R\backslash\{0\})$ and the following decay estimates. Since the proof involves the same ideas as the ones in Proposition~\ref{prop:reguparam1}, we omit it.
\begin{proposition}\label{prop:reguparam2}
	Let $(c,\ka)\in \D_1\cup \D_3$. Consider $u_{c,\ka}\in\boX(\R)$  the singular solution given by Corollary~\ref{cor:cuspon} and $\eta_{c,\ka}=1-\abs{u_{c,\ka}}^2$ be its intensity profile. Then, for every multi-index $\alpha=(\alpha_1,\alpha_2,k)\in\N^3$, there exist $R$, $C_0$, and $C$ positive numbers such that for all $\abs{x}\geq R$
	\begin{align}\label{eq:decayckw}
		|D^\alpha_{c,\ka,x}\eta_{c,\ka}(x)|\leq C_0 e^{-C(|x|-R)},\quad\text{ where } D^\alpha_{c,\ka,x}=\partial_c^{\alpha_1}\partial_\ka^{\alpha_2}\partial_x^k.
	\end{align} 
\end{proposition}
\section{Energy and momentum of traveling waves}
\label{sec:energymoment}
In this section, we compute explicit formulas for the energy and the momentum of the dark solitons and the cuspons given by Theorem~\ref{thm:classiftwregu} and Corollary~\ref{cor:cuspon} in $\boN\boX(\R)$, respectively. In this manner, we can check the stability condition 
\begin{equation}
\label{cond:der:P}
    \frac{d}{dc}p(u_{c,\kappa})<0.
\end{equation}
For the sake of simplicity, we use the notations $E_\kappa(c)$, $p_\kappa(c)$,
$\tilde E_\kappa(c)$ and $\tilde p_\kappa(c)$, corresponding to the set of parameters $\D$ and $\tilde \D$, defined by 
\begin{align}
\label{def:Eck}
E_\ka(c)&:=E_k(u_{c,k}) \text{ and } p_\ka(c):=p(u_{c,k}), \text{ if }
  (c,\ka)\in\boD, 
  \text{ with $u_{c,\kappa}$   given by  Theorem~\ref{thm:classiftwregu}},\\
  \label{def:tildeEck}
   \tilde E_\ka(c)&:=E_k(u_{c,k}) \text{ and } \tilde p_\ka(c):=p(u_{c,k}), \text{ if }
   (c,\ka)\in\tilde \boD, 
   \text{ with $u_{c,\kappa}$   given by  Corollary~\ref{cor:cuspon}}.
\end{align}
The starting point is the following elementary result.
\begin{lemma}\label{lem:Eexpli}
For all $(c,\kappa)\in \D$, we have
\begin{align}\label{eq:Eexpli}
E_\ka(c)=\int_{1-c^2/2}^0 -y\sqrt{\frac{1-2\kappa+2\kappa y}{2-c^2-2y}}dy,\qquad  p_\kappa(c)=\frac{c}{2}\int_{1-c^2/2}^0\frac{-y}{1-y}\sqrt{\frac{1-2\kappa+2\kappa y}{2-c^2-2y}}dy,
	\end{align}
where we exclude the value $c=0$ for momentum. Similarly, for all $(c,\kappa)\in \tilde\D$, we have 
\begin{align}\label{eq:EWexpli}
\tilde E_\ka(c)=\int_{1-1/(2\kappa)}^0 -y\sqrt{\frac{1-2\kappa+2\kappa y}{2-c^2-2y}}dy, \qquad   \tilde p_\ka(c)=\frac{c}{2}\int_{1-1/(2\kappa)}^0\frac{-y}{1-y}\sqrt{\frac{1-2\kappa+2\kappa y}{2-c^2-2y}}dy.
	\end{align}
 \end{lemma}
\begin{proof}
	Since $\eta$ is even, from \eqref{eq:Euregu} we obtain
	$E_\ka(u)=\int_0^\infty\eta^2$. Equation \eqref{eq:firstint} yields $E_\ka(c)=-\int_0^\infty \eta\eta'\sqrt{\frac{1-2\kappa+2\kappa \eta}{2-c^2-2\eta}}$, so that the formula for $E_\ka(u)$ in \eqref{eq:Eexpli} follows from a change of variables. Concerning $p_\kappa(c)$, for $c>0$ we have $1-\eta>0$, therefore the momentum in \eqref{eq:Euregu} is well-defined, and the same argument gives us the expression for $p_\kappa(c)$.

 Noticing that the formulas in \eqref{eq:Euregu} are still valid for the cuspons, 
we obtain similarly the expressions in \eqref{eq:EWexpli}, by using \eqref{eq:intweak}.
\end{proof}

It remains to compute the integrals in \eqref{eq:Eexpli}--\eqref{eq:EWexpli}.
To simplify our results, let us set the constants
\begin{align*}
	A_{c,\kappa}=3c^4\kappa^2-8c^2\kappa^2-2c^2\kappa+8\kappa-1,\	
	B_{c,\kappa}=3c^2\kappa-4\kappa-1,\\
	C_{c,\kappa}=c^2\kappa-4\kappa-1,\ 
	D_{c,\kappa}=\sqrt{(1-2\kappa)(2-c^2)},~
	L_{c,\kappa}=\sqrt{\frac{2-c^2}{1-2\kappa}}.
\end{align*}

\begin{proposition}\label{prop:Eexpli}
  \textup{(i)} Let $(c,\kappa)\in \D_1\cup \D_3$, and set $\gs=1$ if $(c,\kappa)\in \D_1$,	and $\gs=-1$ if $(c,\kappa)\in \D_3$.
		Then 
		\begin{align}\label{eq:ED1}
			E_\ka(c)&=\frac{1}{16 \kappa^{3/2}}\Big(A_{c,\kappa}\atan(\sqrt{\kappa}L_{c,\kappa})
			-\gs \sqrt{\kappa}B_{c,\kappa}D_{c,\kappa}\Big),\\
			p_\kappa(c)&=\frac{c}{4}\Big(\frac{C_{c,\kappa} }{\sqrt{\kappa}}\atan(\sqrt{\kappa} L_{c,\kappa})
			-\gs D_{c,\kappa}\Big)+ \atan\Big(\frac{L_{c,\kappa}}c\Big),\quad \label{eq:PD1}\text{ if }c>0.\\
			\label{eq:EWD1}
			\tilde E_\ka(c)&=\frac{1}{16 \kappa^{3/2}}\Big(A_{c,\kappa}
			\Big[\atan(\sqrt{\kappa}L_{c,\kappa})-\pi/2\Big]
			-\gs \sqrt{\kappa}B_{c,\kappa}D_{c,\kappa}\Big),\\
			\tilde p_\kappa(c)&=\frac{c}{4}\Big(\frac{C_{c,\kappa} }{\sqrt{\kappa}}
			\Big[	\atan(\sqrt{\kappa} L_{c,\kappa})-\pi/2\Big]
			-\gs D_{c,\kappa}\Big)+ \atan\Big(\frac{L_{c,\kappa}}c\Big)-\pi/2. \label{eq:PWD1}
		\end{align}
  \textup{(ii)} 
  For $(c,\kappa)\in \D_2$, with $\kappa<0$, we have 
		\begin{align}\label{eq:ED2}
			E_\ka(c)&=\frac{-1}{16 \abs{\kappa}^{3/2}}\Big(A_{c,\kappa}\atanh(\sqrt{ \abs{\kappa} }L_{c,\kappa})
			-\sqrt{\abs{\kappa}}B_{c,\kappa}D_{c,\kappa}\Big),\\
			p_\kappa(c)&=\frac{c}{4}\Big(\frac{C_{c,\kappa} }{\sqrt{\abs{\kappa}}}\atanh(\sqrt{\abs{\kappa}} L_{c,\kappa})
			-D_{c,\kappa}\Big)+ \atan\Big(\frac{L_{c,\kappa}}c\Big),\quad\text{ if }c>0. \label{eq:PD2}
		\end{align}
    \textup{(iii)} If $(c,\kappa)\in \B_+\cup \B_-$, we have
		\begin{align}\label{eq:ED4}
			\tilde E_\ka(c)&=\pi\frac{(2\kappa-1)^2}{32\kappa^{3/2}},\quad 
			\tilde p_\kappa(c)=\pi\frac{(\sqrt{2\kappa}-1)^2}{4\sqrt{2\kappa}}.
		\end{align}
	
\end{proposition}

\begin{proof}
Given Lemma~\ref{lem:Eexpli}, we use the following antiderivatives:  In case (I),   for all $y\in\boI_c\cup\mathfrak{T}_\kappa$,
\begin{multline*}
\int-y\sqrt{\frac{1-2\kappa+2\kappa y}{2-c^2-2y}}dy=\frac{1}{16 \kappa^{3/2}}\Big((3c^4\kappa^2-8c^2\kappa^2-2c^2\kappa+8\kappa-1)\atan\Big(\sqrt{\kappa}\sqrt{\frac{2-c^2-2y}{1-2\kappa+2\kappa y}}\Big)\\
\phantom{=\int-y\sqrt{\frac{1-2\kappa+2\kappa y}{2-c^2-2y}}dy}
-\sqrt{\kappa}\sqrt{(2-c^2-2y)(1-2\kappa+2\kappa y)}(3c^2\kappa-4\kappa y-4\kappa-1)\Big),    
\end{multline*}
 and     
\begin{multline*}
	   \frac{c}{2}\int\frac{-y}{1-y}\sqrt{\frac{1-2\kappa+2\kappa y}{2-c^2-2y}}dy	=\frac{c}{4}\Big(\frac{c^2\kappa-4\kappa-1}{\sqrt{\kappa}}\atan\Big(\sqrt{\kappa}\sqrt{\frac{2-c^2-2y}{1-2\kappa+2\kappa y }}\Big)\\
\phantom{=\frac{c}{2}\int\frac{-y}{1-y}\sqrt{\frac{1-2\kappa+2\kappa y}{2-c^2-2y}}dy	}
-\sqrt{(2-c^2-2y)(1-2\kappa+2\kappa y)}\ \Big)+\atan\Big(\frac{1}{c}\sqrt{\frac{2-c^2-2y}{1-2\kappa+2\kappa y}}\Big).
	 \end{multline*}

In case (ii) and for all $y\in\boI_c$, we take $\sqrt{-\kappa}$ instead of $\sqrt{\kappa}$ and $\atanh$ instead of the first $\atan$ in the formulas above, in the fashion of Lemma~\ref{lem:etafunctions}. 
	In case (iii), for all $y\in\boJ_c\cup\mathfrak{J}_\kappa$, 
 we take
	\begin{align*}
\int-y\sqrt{\frac{1-2\kappa+2\kappa y}{2-c^2-2y}}dy&=\frac{1}{16 \kappa^{3/2}}\Big((3c^4\kappa^2-8c^2\kappa^2-2c^2\kappa+8\kappa-1)\atan\Big(\sqrt{\kappa}\sqrt{\frac{2-c^2-2y}{1-2\kappa+2\kappa y}}\Big)\\\notag
		&\qquad\qquad+\sqrt{\kappa}\sqrt{(2-c^2-2y)(1-2\kappa+2\kappa y)}(3c^2\kappa-4\kappa y-4\kappa-1)\Big),\\
		  \frac{c}{2}\int\frac{-y}{1-y}\sqrt{\frac{1-2\kappa+2\kappa y}{2-c^2-2y}}dy&=\frac{c}{4}\Big(\frac{c^2\kappa-4\kappa-1}{\sqrt{\kappa}}\atan\Big(\sqrt{\kappa}\sqrt{\frac{2-c^2-2y}{1-2\kappa+2\kappa y }}\Big)\\\notag
		&\qquad\qquad+\sqrt{(2-c^2-2y)(1-2\kappa+2\kappa y)}\ \Big)+\atan\Big(\frac{1}{c}\sqrt{\frac{2-c^2-2y}{1-2\kappa+2\kappa y}}\Big).
	\end{align*}
	We verify that these are indeed the antiderivatives of the integrands in \eqref{eq:Eexpli}--\eqref{eq:EWexpli} from the derivatives computed in Lemma~\ref{lem:etafunctions}, and using that 
	$$\frac{\d}{\d y}
	\Bigg(\atan\Big(\sqrt{\frac{ 2-c^2-2y}{c^2(1-2 \kappa +2\kappa y)}}\Big)
	\Bigg)=\frac{-c}{2(1-y)\sqrt{(2-c^2-2y)(1-2 \kappa +2\kappa y)}},$$
	$$\frac{\d}{\d y}\sqrt{(2-c^2-2y)(1-2\kappa+2\kappa y)}= \frac{-4\kappa y+4\kappa-\kappa c^2-1}{\sqrt{(2-c^2-2y)(1-2\kappa+2\kappa y)}}.
	$$
	In the case (iii), by the dominated convergence theorem, the integrals \eqref{eq:EWexpli} are continuous in $(c,\ka)\in\tilde{\D}$. Therefore, it suffices  to take the limit as $c\rightarrow\sqrt{2}$ in \eqref{eq:EWD1} to obtain \eqref{eq:ED4}.
\end{proof}

As a consequence of Proposition~\ref{prop:Eexpli}, 
we can compute the limit cases stated in the following result. We omit the proof, since it only uses  Taylor expansions and the identity $\atan(x)+\atan(1/x)=\pi/2$, for all $x\neq0$. 
\begin{corollary}
For all $c>0$, as $\ka\rightarrow1/2$, we have 
		\begin{equation}
   \label{eq:energylim}
			(E_k(c),p_k(c))\to\Big(\pi\frac{3c^4/4-3c^2+3}{8\sqrt{2}} , \pi\Big( \frac{c^3/2-3c}{4\sqrt{2}} +1/2\Big)\Big)
   =E(u_{c,1/2}^{(1)}),p(u_{c,1/2}^{(1)}),   
		\end{equation}
where  $u_{c,1/2}^{(1)}$ is the dark compacton given by Proposition~\ref{prop:compacton}, and $(\tilde E_\kappa(c),\tilde p_\kappa(c))$ converges to $(0,0)$.
Also, for all $c>0$, as $\ka\rightarrow 0$, $$(E_\kappa(c),p_\kappa(c))\to\Big(\frac{(2-c^2)^{\frac{3}{2}}}{3},\frac{\pi}{2}-\atan\Big(\frac{c}{\sqrt{2-c^2}}\Big)-\frac{c}{2}\sqrt{2-c^2}\Big)=(E(u_{c,0}),p(u_{c,0})),$$
where $u_{c,0}$ is the dark soliton of the Gross--Pitaevskii equation, given in \eqref{sol:1D}.
Finally, for all $\ka\in\R$, as $c\rightarrow\sqrt{2}$, $(E_\ka(c),p_\kappa(c))\to 0$, while  $\tilde E_\ka(c)$
and $\tilde p_\kappa(c)$ converge to the values  for the energy and momentum in \eqref{eq:ED4}.
\end{corollary}
\begin{remark}
Notice that from \eqref{eq:PD1}, we deduce that $p_\kappa(0^+)=\pi/2$ for all $\ka<1/2$.

\end{remark}


We check now if the stability condition \eqref{cond:der:P} is satisfied for the smooth solitons given by Theorem~\ref{thm:classiftwregu}, according to regions $\D_1$, $\D_2$ and $\D_3$, for $c>0$. In this case the functions $c\mapsto p_\kappa(\cdot)$ are $c\mapsto E_\kappa(\cdot)$ are smooth with respect $c$, on $(0,\sqrt 2)$, so that we introduce the notations  
\begin{equation}
	\label{def:derp:derE}
	p'_\kappa(c)=\frac{d}{dc}p_\ka(c) \quad \text{ and }   \quad E'_\kappa(c)=\frac{d}{dc}E_\ka(c).
\end{equation}
In analogous manner, we define $\tilde E'_\kappa$  and $\tilde p'_\kappa$ 

To rigorously compute the Vakhitov--Kolokolov criterion, we need the Hamilton group property, which is just a consequence of the formulas in Proposition~\ref{prop:Eexpli}.
\begin{lemma}\label{lem:hamiltongp}
	If $(c,\ka)\in\D$ with $c>0$,
	then  ${E'_\ka(c)=c \, p'_\ka(c)}$.
 In the same manner, if $(c,\ka)\in\tilde{\D}$ with $c>0$, then 
	 $\tilde E'_\ka(c)=c \, \tilde p'_\ka(c)$.
\end{lemma}
Consequently, it is enough to compute $p'_\kappa$, as follows.

\begin{proposition}\label{prop:behavep}
Let $(c,\kappa)\in\D$ with $c>0$. Then 	
\begin{align}\label{eq:PCD1}
p_\kappa'(c)=&\frac{3c^2\ka-4\ka-1}{4\sqrt{\ka}}\atan\Big(\sqrt{\ka}\sqrt{\frac{2-c^2}{1-2\ka}}\Big)-\frac{3(2-c^2)}{4}\sqrt{\frac{1-2\ka}{2-c^2}},\ \text{ if }  (c,\ka)\in \D_1\cup \D_3,
   \\
\label{eq:PCD2}
			p_\kappa'(c)=&\frac{3c^2\ka-4\kappa-1}{4\sqrt{-\ka}}\atanh\Big(\sqrt{-\ka}\sqrt{\frac{2-c^2}{1-2\ka}}\Big)-\frac{3(2-c^2)}{4}\sqrt{\frac{1-2\ka}{2-c^2}}, \ \text{ if } (c,\ka)\in \D_2.
		\end{align}
Furthermore,  we have $p_\kappa'(c)<0$  if $(c,\ka)\in \D_1$, 
$p_\kappa'(c)>0$ if $(c,\ka)\in \D_3,$ and  
  \begin{align}\label{eq:ineqconvex}
    		p_\kappa'(c)<\max\{p_\kappa'(0^+),0\}, \quad \text{ if } (c,\ka)\in \D_2.
			\end{align}
\end{proposition}
\begin{proof}
The formulas in \eqref{eq:PCD1}--\eqref{eq:PCD2} follow directly by differentiating the expressions in the momentum in Proposition~\ref{prop:Eexpli}.

If $(c,\ka)\in \D_1$, then  $3c^2<6<4+1/\ka$, so that  $3c^2\ka-4\ka-1<0$ and we conclude that  $p'_\ka(c)>0$. Similarly, if 
 $(c,\ka)\in \D_3$, then $3c^2\ka-4\ka-1>0$, and we get $p'_\ka(c)>0$, since $c^2>2$.

Finally, we consider the case $(c,\ka)\in \D_2$. We remark from \eqref{eq:PCD2} that $p'_\kappa(c)$
can be seen as a function of $c^2$. Differentiating twice with respect to $c^2$, we obtain 
\begin{equation}
\label{d2p}    
\frac{d^2}{(d c^2)^2} \, p'_\kappa(c)=\frac{(1-2\ka)^2}{4(2-c^2)(1-\ka c^2)^2}\sqrt{\frac{1-2\ka}{2-c^2}}.
\end{equation}
Thus, $p'_\kappa(c)$ is a strictly convex function of $c^2$, for $c^2\in(0,2)$, so that it is below  its end points,
  $$p'_\kappa(c)<\max\{p'_\kappa(0^+),p'_\kappa(\sqrt 2^-)\}.$$
 Since $p'_\kappa(c)\to 0$, as $c\to\sqrt 2^-$, the inequality in \eqref{eq:ineqconvex} follows.
 \end{proof}

We focus now on the case $(c,\kappa)\in \D_2$, where the sign of $p'_k(c)$ is unclear. Indeed, it depends on $p'_\ka(0)$, which corresponds to the function in \eqref{eq:PCD2}, evaluated at $c=0$, i.e.
\begin{equation}
\label{def:w}
    w(k)=p'_\ka(0)=\frac{-4\kappa-1}{4\sqrt{-\ka}}\atanh\Big(\sqrt{\frac{-2\ka}{1-2\ka}}\Big)-\frac{3}{2}\sqrt{\frac{1-2\ka}{2}},
\end{equation}
where $w$ is continuous on $(-\infty,0)$, with $w(-\infty)=\infty$ and $w(0^+)=-\sqrt 2$. 
In addition, using that $x\leq\atanh(x)$, for all $0\leq x<1$, one can verify that  $w'<0$ on $(-\infty,0)$.
 Therefore, $w$ is strictly decreasing and $w$ has a unique zero in $(-\infty,0)$, which  we denote by $\kappa_0$. Using Newton's method, we can check that $\ka_0\approx-3.636$. This value allows us to establish the following result.

\begin{lemma}\label{lem:allure} Let $(c,\ka)\in \D_2$ with $c>0$. We have \begin{equation}\label{limits-E-p}
   E_\kappa(0^+)=E_\kappa(u_{0,\ka}),  \quad p_\ka(0^+)=\pi/2, \quad
      E_\kappa(\sqrt 2^-)=p_\ka(\sqrt 2^-)=0. 
\end{equation}
 If  $\ka\in [\ka_0,0)$, then $E'_\kappa <0$  and $p'_\kappa <0$ in $(0,\sqrt 2)$.
Also,  if  $\ka<\ka_0$, there is $\tilde c_\ka \in (0,\sqrt{2})$, such that  $E'_\kappa >0$ and $p'_\kappa >0$   in $(0,\tilde c_\ka)$, while  $E'_\kappa <0$   and $p'_\kappa <0$ in $(\tilde c_\ka,\sqrt{2})$.
	\end{lemma}
\begin{proof}
The limits in \eqref{limits-E-p} follow from \eqref{eq:ED2}--\eqref{eq:PD2}. 
If  $\ka\in [\ka_0,0)$, then,  by definition of $\ka$, 
$p'_\ka(0)=w(\ka)\leq 0$. Therefore, the conclusion follows from \eqref{eq:ineqconvex} and \eqref{def:derp:derE}.

Assume now that $\ka<\ka_0$, so that $p'_\ka(0)=w(\ka)>0$.
Recall that, from the proof of Proposition~\ref{prop:behavep},  
$p'_\kappa$ is a strictly convex function of $c^2$, satisfying
$p'_\kappa(\sqrt{2}^-)=0$. Also, from \eqref{d2p},    
 $\frac{d^2}{(dc^2)^2}p_\kappa(c)$ goes to $\infty$, as $c\to \sqrt 2^-$. Therefore, we conclude that there exists a unique zero of 
 $p'_\ka$ in $(0,\sqrt 2)$, which we denote by $\tilde c_\ka$, so that
 $p'_\ka>0$ in    $(0,\tilde c_\ka)$ and   $p'_\kappa <0$ in $(\tilde c_\ka,\sqrt{2})$.
\end{proof}
\begin{remark}
In Lemma \ref{lem:allure}, we  have  $p_\ka(c)\to \pi/2$,
as $c\to 0^+$. However, it is not true that the momentum of the black soliton $u_{0,\ka}$ is $\pi/2$. A naive approach will be to try to use the formula \eqref{def:moment}, in a generalized sense. However, since $u_{0,\kappa}$ is a real-valued function, 
we  have $\langle i u_{0,\kappa} , u_{0,\kappa}' \rangle=0$, so 
this would imply that the momentum of $u_{0,\kappa}$ is zero.
Indeed, a proper definition of the momentum of vanishing functions is a difficult problem, and it requires the use of the notion of untwisted momentum, explained in \cite{blacksoliton,deLGrSm1}. For this reason, the analysis of black solitons goes beyond the scope of this article.
\end{remark}
 An immediate consequence of the monotonicity of $E'_\ka$ in Lemma~\ref{lem:allure}, is that the critical speed $c^*_\kappa$  in \eqref{def:c*} is well-defined, and that there is 
 a bijection between the speeds and the momenta as follows.
\begin{corollary}
 \label{prop:difeo}
 \textup{(i)} If  $\ka\in [\ka_0,0)$, then $c^*_\kappa=0$.
 If $\ka<\ka_0$, then there exists a unique  $ c\in(0,\sqrt 2)$
 such that $E_\kappa(0)=E_\kappa(c)$, thus this value corresponds to $c^*_\kappa$. In any case,
  $E_\ka(c_\ka^*)=E_\ka(u_{0,\ka}).$\\
 \textup{(ii)} Define $\gq_0(\kappa)$  as $\gq_0(\kappa)=p_\ka(c_\ka^*)$, if  
 $c_\ka^*\in (0,\sqrt 2)$, and as $\gq_0(\kappa)=\pi/2$, if $c_\ka^*=0$.
 Then the  function $p_\ka : (c_\ka^*,\sqrt{2}) \to (0,\gq_0(\ka))$ 
can be extended by continuity to  $[c_\ka^*,\sqrt{2}]$, 
   is strictly decreasing, and it defines a continuous bijective function, 
     whose inverse function  we denote by  $\gc_\ka : [0,\gq_0(\kappa)] \to  [c_\ka^*,\sqrt 2]$, 
     with 
     $\gc_\ka(0)=\sqrt 2$ and  $\gc_\ka(\gq_0(\kappa))=c_\kappa^*$.
 \end{corollary}
\begin{proof}
In view of \eqref{limits-E-p}, 
 the assertions are a straightforward consequence of the monotonicity results in Lemma~\ref{lem:allure}. 
\end{proof}

Finally, we remark that in the case $\ka<\ka_0$, the function $p_\ka$ defines also a bijection between the intervals  $(0,\tilde c_\ka)$  and  $(\pi/2,p_\ka(\tilde c_\ka))$, but we do need to use this fact in the sequel.  In addition, we will show in the next section that $\gq_0(\kappa)$ is equal to $\gq_\ka^*$, defined in \eqref{q_*}.

Notice that  function $\gc_\ka$ in Corollary~\ref{prop:difeo}, allows to us determine that  $u_{\gc_\ka(\gq)}$ is the unique smooth soliton of momentum $\gq$, up to invariances, for all $\gq\in(0,\gq_0(\kappa))$. In particular, 
in view of the minimization problem \eqref{min},  we have 
\begin{equation}
\label{relation-E-E}
    \boE_\ka(\gq)\leq E_\ka(u_{\gc_\ka(\gq),\kappa}),\quad \text{for all } \gq\in(0,\gq_0(\kappa)).
\end{equation}

\section{Variational characterization of dark solitons}
\label{sec:minimization}
In this section, we prove Theorem~\ref{thm:min}.
The starting point is that equation \eqref{TWc} corresponds to the Euler--Lagrange equation associated with the minimization of the energy at constant momentum, and that  $c$ appears as a Lagrange multiplier. Since  $\boN\boX(\R)$ is not a vector space, we use the G\^ateau differential, denoted by $\d$,  as follows.
\begin{lemma}\label{lem:eulerlag}
	Let $v\in\boN\boX(\R)$ be a (complex valued) function and $h\in H^1(\R;\C)$. For  $t\in\R$   small enough, the functions $t\mapsto p(u+th)$ and $t\mapsto E_\ka(u+th)$ are differentiable, and
	$$\begin{aligned}
 		\frac{d}{d t}p(u+th)\rvert_{t=0}=\int_\R \inner{ih'}{u},\\
  \d E(u)[h]:=&\frac{d}{d t}E_\ka(u+th)\rvert_{t=0}=\int_\R \inner{u'}{h'}-\eta\inner{u}{h}-2\ka\inner{u}{u'}\inner{u}{h}.
	\end{aligned}$$
Also for all $c\geq0$,  $c\, \d p(u)= \d E_\ka(u)[h]$, for all $h\in H^1(\R)$, if and only if $u$ satisfies \eqref{TWc}.
\end{lemma}
We omit the proof of Lemma~\ref{lem:eulerlag}, since it is a straightforward adaptation of the computations in Lemma~6.1 in \cite{delaire-mennuni}.

\subsection{The minimization curve}
We first show that the minimization problem \eqref{min} {\em does not}  define a real-valued function of $\gq\in\R$, for  $\ka>0$.
Therefore, the variational problem is not well-suited to study the
dark solitons in $\D_1\cup\D_3$ given by Theorem~\ref{thm:classiftwregu}.

\begin{proposition}\label{prop:Edegen}
	For all $\ka>0$ and all $\gq\in\R$, we have $\boE_\ka(\gq)=-\infty$.
\end{proposition}
\begin{proof}
	Let $\ka>0$ and $\gq\in\R$. We define a family of functions in $\boN\boX(\R)$, indexed by  $t\geq0$,  such that $t\mapsto E_\ka(u(\cdot,t))$ and $t\mapsto p(u(\cdot,t)) $ are continuous. More precisely, let $u=\rho e^{i\theta}$, where
	\begin{align*}
		\rho(x,t)=\begin{cases}
			-tx^{\alpha_\ka}+t+\frac{1}{\sqrt{2\ka}}+1,\text{ for }x\in(0,1),\\
			(2-x)(\frac{1}{\sqrt{2\ka}}+1)+x-1,\text{ for }x\in(1,2),\\
			1,\text{ for }x\geq2,
		\end{cases}
  \quad 
  \theta(x,t)=\begin{cases}
			0,\text{ for }x\in(0,1),\\
			\gq J_\ka(x-1),\text{ for }x\in(1,2),\\ 
			\gq J_\ka, \text{ for }x\geq2,
		\end{cases}	
	\end{align*}
	for some $\alpha_\ka>1$ to be chosen later, and 
 $$J_\ka=\Big(2\int_1^2(1-\rho(x,t)^2)dx\Big)^{-1}=-3\ka/(1+3\sqrt{2\ka}).$$
	We define $u(x,t)=u(-x,t)$, for all $x\leq0$. Remark that $\partial_x u(t)=e^{i\theta(t)}(\partial_x\rho(t)+i\partial_x\theta(t)\rho(t))\in L^2(\R)$ for all $t\geq0$. Also $1-\rho^2(t)\in L^2(\R)$, for all $t\geq0$, since it is continuous and compactly supported in $[-2,2]$. We have $\rho(t,x)\geq1$, hence $u(\cdot,t)\in\boN\boX(\R)$ for all $t\geq0$. We compute, using the symmetry of $\rho$, 
 $$p(u(\cdot,t))=\int_\R(1-\rho(t)^2)\partial_x\theta(t)=2\int_1^2(1-\rho(t)^2)\gq J_\ka=\gq.$$ On the other hand, from \eqref{eq:energie} we obtain \begin{align}\label{eq:Edegen}
		E_\ka(u(\cdot,t))=&\int_0^1(t\alpha_\ka)^2 x^{2\alpha_k-2}+(1-\rho^2(x,t))^2/2-2\ka(t\alpha_\ka x^{\alpha_\ka-1}\rho(t,x))^2dx\\\label{eq:Edegen2}
		&+\int_1^2|\partial_xu(x,t)|^2+(1-\rho(x,t))^2/2-2\ka(\partial_x\rho(x,t)\rho(x,t))^2dx.
	\end{align}
	We can prove that the second integral \eqref{eq:Edegen2} is constant in $t$. Let us denote by $I(t)$ the first integral \eqref{eq:Edegen}. 
 It remains to show that,  for $\alpha_\ka>1$ large enough, $I(t) $ diverges to $-\infty$, as $t\to \infty$. Indeed,
	We have 		$I(t)=t^4 \beta_k+(t+1)^3R(t),$
where $$\beta_k=\Big(\frac{-4\ka\alpha_\ka^4}{(2\alpha_\ka-1)(4\alpha_\ka-1)(3\alpha_\ka-1)}+\frac{1}{2}-\frac{2}{\alpha_\ka+1}+\frac{3}{2\alpha_\ka+1}-\frac{2}{3\alpha_\ka+1}+\frac{1}{8\alpha_\ka+2}\Big),$$	
 and  $R$ is a bounded function, for $t\geq0$. Since  $\beta_\ka\to-\infty$ as $\alpha_\ka\to\infty$, we deduce that for $\alpha_\ka$ large enough, $E_\ka(u(\cdot,t))\to-\infty$ as $t\to\infty$, which completes the proof.
\end{proof}

From now on, \emph{we assume that $\ka<0$}, so that 
the energy density $e_\ka(u)$ satisfies
\begin{equation}
\label{density:e}
	e_\ka(u)=\frac{\abs{u'}^2}{2}+\frac{\eta^2}4+\frac{\abs\ka}{4}{(\eta')^2}\geq e_0(u)\geq 0.
\end{equation}
Hence,  several properties shown for the curve $E_{\min}$ in \cite{delaire-mennuni} remain true. Indeed,  the curve $E_{\min}$ in \cite{delaire-mennuni} is associated with 
the nonlocal energy 
$$E_{\mathcal W}(u)=\frac12 \int_{\R}\abs{u'}^2\,dx + \frac14\int_{\R}(\mathcal W*\eta)\eta\,dx=
\frac12 \int_{\R}\abs{u'}^2\,dx + \frac1{8\pi}\int_{\R}\widehat{\mathcal W}(\xi)\abs{\widehat \eta(\xi)}^2\,d\xi
,$$
where $\mathcal W$ is a tempered distribution, with bounded nonnegative Fourier transform $\widehat{\mathcal W}$, among other hypotheses.
In our case, by using Plancherel's theorem, we can recast the energy as
$$E_{\kappa}(u)=\frac12 \int_{\R}\abs{u'}^2\,dx +\frac 1{8\pi} \int_{\R}
\widehat{\mathcal W}_\kappa(\xi)
\abs{\widehat \eta(\xi)}^2\,d\xi,\quad \text{with }\widehat \boW_\ka=1+\abs\kappa\xi^2,$$
for all $u\in X(\R)$, so that $\eta\in H^1(\R)$.
In this manner, we recover the potential \eqref{W:kappa} mentioned in the introduction.
Thus, formally, $\W$ is the tempered distribution $\boW_\ka=\delta_0-\abs{\kappa} (\delta_0)''$, but we do not use this formulation,  since $\eta\in H^1(\R)$.
In conclusion, in most of the proofs in this subsection, we  will use that 
 the potential energy can be written as
\begin{equation}
\label{def:Ep}
E_{\text p}(u)=
\frac14 \int_{\R} ({\eta^2}+{\abs\ka}({\eta'})^2)dx, \quad \text{and }\quad
E_{\text p}(u)=\frac 1{8\pi} \int_{\R}
\widehat{\mathcal W}_\kappa(\xi)\abs{\widehat \eta(\xi)}^2\,d\xi.
\end{equation}
More precisely, we will rely on the potential energy written in the Fourier variable to invoke the arguments in  \cite{delaire-mennuni} not needing that $\widehat{\mathcal W}$ to be bounded. Of course, some arguments will be performed directly in the variable $x$, if they are simpler. 

\begin{proposition}\label{prop:Emin}
	Let $k<0$. Then the function 	$\boE_\ka$ is well-defined, is even and Lipschitz continuous, with 
	\begin{equation}
		\label{E:lipschitz}
		|\boE_\ka(\gp) - \boE_\ka(\gq)| \leq \sqrt{2}|\gp - \gq|, \quad\text{ for all } \gp, \gq\in\R.
	\end{equation}
	In particular, $\boE_\ka(\gq)\leq \sqrt 2\gq$, for all $\gq\geq0$.  Also $\boE_\ka$ is nondecreasing and  concave on $[0, \infty)$.
 Moreover, 
  for all $v\in\boX(\R)$, 
 \begin{equation}
 \label{prop:q*}
 E_\ka(v) \leq  \boE_\ka( \gq), \text{ for some }\gq\in[0,\q_\ka^*) \ \Rightarrow v\in \boN\boX(\R).
 \end{equation}
\end{proposition}
\begin{proof}
	It is straightforward to check that Lemmas 3.1, 3.2, 3.4 in \cite{delaire-mennuni} hold with the same proofs. 	Corollary~3.7 is still true, and the proof is simpler, using the first expression for $E_{\text p}$ in \eqref{def:Ep}. Thus,  using the notation in the proof of Corollary 3.7, we obtain
	$E_{\text p}(u_n)=E_{\text p}(v_n)+E_{\text p}(w_n)$, which gives the conclusion.
	As a consequence, we deduce, as in Corollary~3.8 and Proposition~3.9 in \cite{delaire-mennuni}, that
	$$
	\boE_\ka(\gq)=\inf\{E_\ka(v) : v\in \boX_0^\infty(\R),\  p(v)=\gq \},
	$$
	where 
	$$\boX^\infty_0(\R)=\{v\in \boN\boX(\R) \cap \boC^\infty(\R) : \exists R>0\text{ s.t.\ }  v \text{ is constant on }B(0,R)^c\},$$
	and that $\boE_\ka$ satisfies \eqref{E:lipschitz}.
	The fact that $\boE_\kappa$ is nondecreasing follows as in Lemma~3.11.
	
	The proof of concavity is exactly the same as in  Proposition~3.12, without any extra assumption needed. Indeed, using the reflection functions defined in the proof of  Proposition~3.12 and the first expression  in \eqref{def:Ep}, it is immediate to verify that 
	$E_{\text p}(u^+)+E_{\text p}(u^-)=2E_{\text p}(u),$ 
	so that the conclusion follows.

Finally, let $0\leq\gq<\gq_\ka^*$ so that, by definition of $\gq_\ka^*$, we can find $\gq<\check{\gq}<\gq_\ka^*$ satisfying 
\begin{equation}\label{eq:liftq}
    \forall v\in\boX(\R), E_\ka(v)\leq\boE_\ka(\check{\gq})\implies v\in\boN\boX(\R).
\end{equation} Since $\boE_\ka$ is nondecreasing on $[0,+\infty)$, we deduce  that $\boE_\ka(\gq)\leq\boE_\ka(\check{\gq})$, so that \eqref{eq:liftq} holds with $\gq$ instead of $\check{\gq}$, which proves \eqref{prop:q*}.

\end{proof}

\begin{proposition}\label{prop:minoration}
	Assume that $\ka<0$.
	\begin{enumerate}
		\item 
		Let   $u=\rho e^{i\theta} \in \boX(\mathbb{R})$  and
		assume that there is  $\ve\in  (0,1)$  such that
		$1-\ve\leq \abs{u}^2\leq 1+ \ve$ on an open set $\Omega\subset \R$. Then
		\begin{equation}
			\label{ctrlEPsurR}
			\frac12	\int_\Omega\abs{\eta \theta'}\leq \frac1{\sqrt{2}(1-\ve)} {\int_\Omega e_\kappa (u)}.
		\end{equation}
		In particular, if $1-\ve\leq \abs{u}^2\leq 1+ \ve$ on $\R$, then $\sqrt{2}(1-\ve)p(u)\leq E_\kappa(u).$
		\item For any $u\in \boX(\R)$, we have
		\begin{equation}
			\label{borne:infty}
			\norm{\eta}_{L^\infty(\R)}^2\leq (1+|\kappa|^{-1})E_\kappa(u).
		\end{equation}
		\item	There is a constant $K_0>0$  
		such that 
		\begin{equation}
			\label{Emin-inf}
			\sqrt 2\gq-K_0 \gq^{3/2}\leq \boE_\ka(\gq), \quad \text{for all }\gq \in \Bigg[0,\frac{\abs{\ka}}{8(1+\abs{\ka})}\Bigg).
		\end{equation}
			\end{enumerate}	
\end{proposition}
\begin{proof}
	Using the  Cauchy inequality $ab\leq a^2/2+b^2/2$,
	with $a=\eta/2$ and $b=\theta'^2$,
	we have
	\begin{align}
		\label{mom-young0}
		\left|\int_\Omega \frac{\eta}{2} \theta'\right| \leq  \int_{\Omega}\Big(\frac{\eta^2}4+
		\frac{\theta'^2}2\Big)\leq 
		\frac{1}{4}\int_{\Omega}\eta^2+\frac{1}{2(1-\ve)}\int_{\Omega}\rho^2\theta'^2. 
	\end{align}
	Bearing in mind \eqref{eq:Eupol} and that $\ka\leq 0$,  \eqref{ctrlEPsurR} follows.

	The estimate in (ii) is an immediate consequence of 
	$$
	\eta^2(x)=2\int_{-\infty}^x \eta \eta' \leq \int_\R (\eta^2 +\eta'^2)\leq 4(1-\kappa^{-1})E_\kappa(u).
	$$
		In view of \eqref{ctrlEPsurR} and \eqref{borne:infty},  the inequality in \eqref{Emin-inf} follows exactly as in Proposition 3.14 in \cite{delaire-mennuni}.
\end{proof}
\begin{proposition}
	\label{prop:strict}
	Let $k< 0$. We have
	\begin{equation}
		\label{ineq:strict}
		\boE_\kappa(\gq)< \sqrt 2\gq,\quad \text{for all }\gq >0. 
	\end{equation}
\end{proposition}	
\begin{proof}
In view of Proposition~\ref{E:lipschitz}, and since $\boE_\ka$ in concave on $\R^+$, 
we only need to prove that the strict inequality in \eqref{ineq:strict} holds for $\gq$
small. For this purpose, we will use the behavior of solitons $u_{c,\ka}$ in Proposition~\ref{prop:Eexpli}-(ii)
	for $c$ close to $\sqrt{2}$. By setting $\ve=\sqrt{2-c^2}$, and  computing a Taylor expansion in  formulas \eqref{eq:ED2}--\eqref{eq:PD2}, 
	we deduce that 
		$$
	E_\ka(u_{c,\kappa})=\frac{\ve^3}3\sqrt{1-2\kappa}+\frac{2\ve^5\ka }{15\sqrt{1-2\kappa}}+o(\ve^6),\quad 
	p(u_{c,\kappa})=\frac{\ve^3}{3\sqrt 2}\sqrt{1-2\kappa}+\frac{3+2\ka\ve^5}{60\sqrt{2}(1-2\kappa)}+o(\ve^6).
	$$ 
	Therefore, there is $\ve_0>0$ small, such that 
	$\boE_\kappa(p(u_{c,\kappa}))\leq E_\ka(u_{c,\kappa})<\sqrt 2 p(u_{c,\kappa})$, for all $c\in (c(\ve_0),\sqrt 2)$, where $c(\ve_0)=(2-\ve_0^2)^{1/2}$. 
	By Corollary~\ref{prop:difeo}, we can assume that  
 $p_\ka$ is a diffeomorphism from  $(c(\ve_0),\sqrt 2)$ to $(0,p_\ka(c(\ve_0))$.
Thus, we conclude that there exists $\gq_0>0$ such that $	\boE_\ka(\gq)< \sqrt 2\gq$,
	for all $\gq\in(0,p_\ka(c(\ve_0))$.
\end{proof}	
We can prove now the minimizing curve is strictly subadditive, which is the crucial property to deduce the compactness of minimizing sequences in Theorem~\ref{thm:compacite}.
\begin{corollary}
	\label{cor:subadditive}
	If  $\ka< 0$, then 		$\boE_\ka$ is strictly subadditive  on $\R^+$.
\end{corollary}
\begin{proof}
	From Propositions \ref{prop:Emin} and \ref{prop:minoration}-(iii), we deduce that 
	the right derivative of $\boE_\ka$ at the origin, denoted by $\boE_\ka^+$, exists and that $\boE_\ka^+(0)=\sqrt{2}$. By invoking Lemma 3.16 in \cite{delaire-mennuni}, Proposition~\ref{prop:strict}
	implies that   $\boE_\kappa$ is strictly subadditive on $\R^+$. 
\end{proof}
\subsection{Compactness of minimizing sequences}
We are now in a position to prove the compactness of minimizing sequences in the problem  $\boE_\kappa(\gq),$ with $\kappa<0$, and that the minimizers are the solitons $u_{c,\kappa}$, with $c=\gc_\ka(\gq)$. Notice that in the rest of this section,  $u_{c,\ka}$ refers to the dark soliton to \eqref{TWc}, given by Theorem~\ref{thm:classiftwregu}.

 We will use the general argument given \cite{delaire-mennuni},
based on the properties of the minimizing curve $\boE_\kappa$ and an adaptation of the concentration-compactness principle.
However, we need to guarantee the nonvanishing properties of the limit of the minimizing sequences. This is the purpose of the constant $\gq_\ka^*$ defined in \eqref{q_*}.
To relate $\gq_\ka^*$  with $\gq_0(\kappa )$, let us state some properties related to the energy of the black soliton $u_{0,\kappa}$ in \eqref{def:k:blacksoliton}, i.e.\ the solution to \eqref{TWc}  with $c=0$.

\begin{lemma}
\label{lem:ineg:q0}
    Let $\ka<0$ and  $\gq\geq0$.
    There exists a sequence $(u_n)\subset\boN\boX(\R)$ satisfying
	\begin{align}\label{eq:emincst}
		p(u_n)=\gq, \ \text{ for all } n\in\N,\quad\text{ and }\quad \lim\limits_{n\to\infty}E_\ka(u_n)=E_\ka(u_{0,\ka}).
	\end{align}
In addition, 
    \begin{equation}
        \label{minE0}
        \inf\{E_\ka(v) : v\in H^1_{\loc}(\R),~\inf_{x\in\R}|v(x)|=0\}=E_\ka(u_{0,\ka}).
    \end{equation}
 In particular, for all $\gq> 0$, we have  $0\leq\boE_{\ka}(\gq)\leq E_\ka(u_{0,\ka})$  and  $\gq_0(\kappa)\leq\gq_\ka^*$. 
  \end{lemma}
\begin{proof}
The existence of a sequence $(u_n)$ satisfying \eqref{eq:emincst} is analogous to the case $\ka=0$, done  in Proposition~3.4 in
\cite{berthoumieu2023minimizing}.

The minimization problem \eqref{minE0} correspond to Lemma~1 in \cite{bethuel2008existence} in the case $\ka=0$. In the case $\kappa<0$, the same proof holds, using that 
$u_{0,\ka}$ is the unique solution to \eqref{TWc} that vanishes at some point, up to a translation.

Finally, let us show  that \eqref{minE0} implies that
$\gq_0(\kappa)\leq\gq_\ka^*$. Let $\gq \in (0,\gq_0(\kappa))$ and let $v\in \mathcal{E}(\mathbb{R})$ such that 
$E_\ka(v)\leq \boE_\ka(\gq)$. By \eqref{relation-E-E}, we deduce that 
$E_\ka(v)\leq E_\ka(u_{\gc_\ka(\gq),\kappa})$, with $\gc_\ka(\gq) \in (c_\ka^*,\sqrt 2)$.
Since, by definition,  $E_\ka(u_{\gc_\ka(\gq),\kappa})=E_\kappa(\gc_\ka(\gq))$,  and, by Lemma~\ref{lem:allure},  the map $c\mapsto E_\kappa(c)$ is strictly decreasing on $[c_\ka^*,\sqrt 2]$, we get 
$E_\ka(v)< E_\kappa(c_\ka^*)$. By Corollary~\ref{prop:difeo}, we have
$E_\kappa(c_\ka^*)=E_\kappa(u_{0,\ka})$, so that  
$E_\ka(v)<E_\kappa(u_{0,\ka})$.  Thus, using  \eqref{minE0},
we infer that $\inf_{\R}\abs{v}>0$. This implies that $\gq\leq \gq_\ka^*$, 
which completes the proof.
\end{proof}

In order to obtain the stability of minimizers, we will prove a more general result than the one stated in Theorem~\ref{thm:min}, as follows.

\begin{theorem}
	\label{thm:compacite}
	Let $\kappa<0$, 
	  $\gq\in (0,\gq_0(\ka))$ and  $(u_n)$ in $\boN\boX(\R)$ be a sequence  satisfying
	\begin{equation}
		\label{mseq}
		E_\ka(u_n)\to \boE_{\ka}(\gq)\quad \text{ and } \quad p(u_n)\to \gq, \quad \text{	as } n\to\infty.
	\end{equation}
	There exist $\theta \in \R$ and 
	a sequence of points $(x_n)$
	such that, up to a subsequence still denoted by  $u_n$,
	the following convergences hold, 	as $n\to\infty$, 
	\begin{alignat}{2}
		\label{cvuniforme} 
		u_{n}(\cdot+x_n)&\to  e^{i\theta}u_{\gc(\gq),\ka}(\cdot ),  &\quad \text{ in }&L^\infty_{\loc}(\R),\\
		\label{cvforteepot}
		1-|u_{n}(\cdot+x_n)|^{2} &\to 1-|u_{\gc(\gq),\ka}(\cdot)|^{2},   &\quad\text{ in }&L^2(\R),\\
		\label{cvfortegradient}
		u_{n}'(\cdot+x_n)&\to e^{i\theta}u'_{\gc(\gq),\ka}(\cdot),  &\quad\text{ in }&L^2(\R).
	\end{alignat}
  Moreover, we have $\gq_\ka^*=\gq_0(\ka)$,
  $\boE_\ka(\gq)=E_\ka(u_{\gc(\gq),\ka})$ for  $\gq \in (0,\gq_\ka^*)$, and 
  $\boE_\ka(\gq)=E_\ka(u_{0,\ka})$ for  $\gq\geq\gq_\ka^*$. In particular, 
  $\boE_\ka$ is strictly increasing on $(0,\gq_\ka^*)$.
\end{theorem}
\begin{proof}
	As explained before, the proof follows the same steps as in Theorem~4.1 in \cite{delaire-mennuni}, without the extra decomposition introduced to handle the nonlocal interactions. The characterization of the limit function is obtained by invoking Corollary~\ref{prop:difeo}.
	Therefore, we only give a sketch of the proof.
	
	First, using Proposition~\ref{prop:strict}, we can set 	$\Sigma_\q=1-{\Emin(\gq)}/{(\sqrt{2}\q)}\in(0,1)$, for any $\gq>0$.
	In addition, without loss of generality, we can assume that 
	\begin{equation}
		\label{borneEn}
		E_\kappa(u_n)\leq 2\boE_{\ka}(\gq).
	\end{equation}
	Thus, we can apply Lemmas~2.4 and 2.5 in \cite{delaire-mennuni}, with 
	with $L=1+\Sigma_\gq$, $E=2\boE_{\ka}(\q)$ and $m_0=\tilde \Sigma_\gq:=\Sigma_\gq/L$,
	to  deduce that there exist $R>0$, two integers $\ell,l_*$, with $1\leq \ell\leq l_*$, depending on $E$ and $\gq$, but not on $n$,
	and  points  $x_1^n,x_2^n,\dots,x_{l_*}$,  satisfying 
	$$
	|x_{k}^{n}-x_{j}^{n}|\underset{n\to \infty}\longrightarrow \infty,
	\text{ for } 1\leq k\neq j\leq \ell,
	\text{ and }
	\ x_{j}^{n} \in \displaystyle\mathop{\cup}_{k=1}^{\ell}B(x_{k}^n,R),\text{ for } \ell<j \leq l_*.
	$$
	In addition,  the sequence $\eta_{n}=1-|u_{n}|^2$ satisfies 
	\begin{equation}
		\label{compacteq} 
		|\eta_n(x_j^n)|\geq \tilde\Sigma_\gq,\quad \forall 1\leq j\leq \ell, \quad  
		\text{ and }\quad 
		|\eta_n(x)|\leq \tilde\Sigma_\gq,\ \forall x\in \mathbb{R}\setminus\bigcup_{j=1}^{\ell}B(x_{j}^n,R+1).
	\end{equation}
	
	Applying standard weak compactness results for Hilbert spaces, the Rellich--Kondrachov theorem, and Fatou's lemma,  to the translated sequence $u_{n,j}(\cdot)=u_{n}(\cdot+x_{j}^{n})$, 
	we infer that there exist functions $v_{j}=\rho_j e^{i\phi_j}\in \boN\boX(\R)$, $j\in\{1,\dots,\ell\}$,  satisfying, up to a subsequence,
\begin{alignat}{2}
\label{weak:conv}
u_{n,j} \to v_j, \text{ in }L^{\infty}_{\loc}(\R), \
u'_{n,j} \rightharpoonup v_j',  \text{ in }L^2(\mathbb{R}), \
\eta_{n,j}=1-|u_{n,j}|^2\rightharpoonup \eta_j=1-|v_j|^2,  \text{ in }L^2(\mathbb{R}),
	\end{alignat}
	as $n\to\infty$, and also, 
	\begin{gather}
		\label{dem:ineq1}
		\int_{-A}^{A}|v_j'|^2\leq \underset{n\to \infty}\liminf\int_{-A}^{A}|u_{n,j}'|^2, 
		\ \text{ for all }A\in (0,\infty],\\
		\label{dem:ineq2}
		\int_{-A}^{A} ({\eta_j^2}+{\abs\ka}({\eta_j'})^2)
		\leq \underset{n\to \infty}\liminf\int_{-A}^{A}({\eta_{n,j}^2}+{\abs\ka}({\eta_{n,j}'})^2),  \text{ for all } A\in (0,\infty],
		\\
		\label{dem:ineq3}
		\underset{n \to \infty}\lim	\int_{-A}^{A}\eta_{n,j}\phi_{n,j}'=\int_{-A}^{A}\eta_j\phi_{j}', \text{ for all } A\in (0,\infty),
	\end{gather}
 as well as	
 \begin{equation}
 \label{dem:vj}
     \boE_\ka(\q_j)\leq 		E_\ka(v_j)\leq \boE_\ka(\gq),
 \end{equation}
 where $u_{n,j}=\rho_{n,j}e^{i\phi_{n,j}}$ and $\q_j=p(v_j)$.
Notice that the fact that $v_j$ belongs to $\boN\boX(\R)$, follows from 
  \eqref{prop:q*} and \eqref{dem:vj}, since $\gq<\gq_0(\kappa)\leq \gq_\ka^*.$
 We also remark that we only have an inequality in \eqref{dem:ineq2}, whereas we had an equality in the potential energy in (4.25) in \cite{delaire-mennuni}.
	The rest of the proof consists in showing the following steps.
	\begin{step}
		\label{step2}
		\textit{There exist $\tilde{\gq}\in \R$ and $\tilde{E}\geq 0$ such that}
		\begin{align}
			\boE_\kappa(\gq)\geq \sum_{j=1}^{\ell}\boE_{\kappa}(\gq_j)+\tilde{E}\quad \textit{ and }\quad
			\gq=\sum_{j=1}^{\ell}\gq_j+\tilde{\gq}.
			\label{inegqfinal}
		\end{align}
	\end{step}
	The proof of the first inequality is the same as in  \cite{delaire-mennuni}, handling
	the potential energy in the same manner as the kinetic energy in \eqref{dem:ineq1}.
	For the momentum, we use inequality \eqref{ctrlEPsurR}, instead of Lemma 2.3 in \cite{delaire-mennuni}. Thus, there is no need to introduce cut-off functions that caused the appearance of some reminder terms.
	\begin{step}
		\label{claim:ineq}
		\textit{We have}
		\begin{equation}
			\label{inegqEfinal}
			\sqrt{2}\left(1-\tilde\Sigma_\gq\right)|\tilde{\gq}|\leq \tilde{E}.
		\end{equation}
	\end{step}
	This inequality follows as in Claim~2 in  \cite{delaire-mennuni}, but using 
	inequality \eqref{ctrlEPsurR}, instead of introducing cut-off functions.
\begin{step}
	\textit{	We have $\tilde{E}=\tilde{\q}=0$ and $\ell=1$.}
	\label{pasdichoto}
\end{step}
The proof of this step is the same as in Claim 3 in  \cite{delaire-mennuni}, since it only uses the properties of the function $\boE_k$ in Propositions \ref{prop:Emin}, \ref{prop:strict} and Corollary~\ref{cor:subadditive}.

\begin{step}
\textit{The weak convergences in \eqref{weak:conv} are also strong in $L^2(\R)$ (for $j=1$).}
\end{step}
We set from now on $v=v_1$ and $\eta=1-\abs{v_1}^2$.
For the previous step, we have
\begin{equation}
	p(u_{n,1})\to \q=p(v) \quad \text{ and } \quad E_\ka(u_{n,1}) \to \Emin(\q)=E_\ka(v).
	\label{q:pv1}
\end{equation}
To prove that  $u'_{n,1}\to v'$ in $L^2(\R)$, it is enough to show that 
\begin{equation}
	\label{cv:norme:grad}
	\limsup_{n\to\infty }\Vert u_{n,1}'\Vert_{L^2(\mathbb{R})}\leq  \Vert v'\Vert_{L^2(\mathbb{R})}. 
\end{equation}
Arguing by contradiction, taking a subsequence that we still denote by $u_{n,1}$,
we suppose that 
$$M:=\lim_{n\to\infty }\norm{u_{n,1}'}_{L^2(\R)}^2, \quad \text{ with } \quad M>\Vert v'\Vert_{L^2(\mathbb{R})}^2.$$
Hence, using \eqref{q:pv1},
\begin{align*}
	\underset{n \to \infty}\lim~ E_{\p}\left(u_{n,1}\right)  &= \underset{n \to \infty}\lim  ~\big( E_\ka(u_{n,1} )-\norm{u_{n,1}'}_{L^2(\R)}^2/2 \big) 
	= E_\ka(v)-{M}/{2} 
	< E_\ka(v)-\Vert v'\Vert_{L^2(\mathbb{R})}^2/2
	= E_{\p}(v),
\end{align*}
which contradicts \eqref{dem:ineq2} (with $A=\infty$). Therefore,  $u'_{n,1}\to v'$ in $L^2(\R)$. Let us show that this also implies that 
\begin{equation}
	\label{eta1-etan1}
	\norm{\eta_{n,1}'-\eta'}_{L^2(\R)}\to 0.
\end{equation}
Indeed, noticing that  $\eta'-\eta_{n,1}'=2(\langle v,v' \rangle- \langle u_{n,1},u_{n,1}' \rangle)$,
we have
\begin{align}
	\label{dernier3}
	\norm{\eta_{n,1}'-\eta'}_{L^2(\R)} 
	&\leq 2 \Vert (v-u_{n,1})v'\Vert_{L^2(\R)}
	+ 2\Vert (v'-u_{n,1}')u_{n,1}\Vert_{L^2(\R)} .
\end{align}
From inequality \eqref{borne:infty}, we obtain the existence of a constant  $C(\q)>0$ such that $\norm{u_{n,1}}_{L^{\infty}(\R)}\leq C(\q).$ Thus, the second term in the right-hand side of \eqref{dernier3} goes to zero. By using  the dominated convergence theorem, we also infer that the other term in the right-hand side of \eqref{dernier3} also tends to zero, which completes the proof of \eqref{eta1-etan1}.

Finally, going back to \eqref{q:pv1}, we conclude that 
$$
\limsup_{n\to\infty }\Vert \eta_{n,1}\Vert_{L^2(\mathbb{R})}=  \Vert \eta\Vert_{L^2(\mathbb{R})}, 
$$
which combined with the last weak convergence in \eqref{weak:conv} (with $j=1$), implies  that $\eta_{n,1}\to \eta$ in $L^2(\R)$.

\begin{step}
\textit{There exists $(\theta,y)\in\R^2$ such that $v=e^{i\theta}u_{\gc(\q),\kappa}( \cdot - y)$.}
\end{step}
By Theorem~6.3 in \cite{delaire-mennuni}, since $p(v)=\gq$ and $\boE_{\ka}(\gq)=E_\ka(v)$, 
we conclude that $v$  is a solution to \eqref{TWc}, for some speed $c$. In addition, as in Theorem~4 in \cite{delaire-mennuni}, we deduce that $c\in(0,\sqrt 2)$. Therefore,
by Corollary~\ref{prop:difeo}, $v$ satisfies \eqref{TWc} with $c=\gc(\gq)$, and the conclusion follows from Theorem~\ref{thm:classiftwregu}.

\begin{step}
	\textit{We have $\gq_\ka^*=\gq_0(\kappa)$ and $\boE_{\ka}(\gq)=E_\ka(u_{0,\ka})$, for all $\gq\geq\gq_\ka^*$.}
\end{step}
By continuity of $\boE_\ka$, invoking Corollary~\ref{prop:difeo}, and using the previous step, we deduce that
\begin{equation}
\label{end:proof:main}
\boE_\ka(\gq_0(\kappa))=\lim_{\gq\to\gq_0(\kappa)^-}\boE_\ka(\gq)= \lim_{\gq\to\gq_0(\kappa)^-}E_\ka(u_{\gc(\gq),\ka})=E_\ka(u_{c^*_\ka,\ka})
=E_\ka(u_{0,\ka})
\end{equation}
where we used the definition \eqref{def:c*} for the equality.  Since
$u_{0,\ka}(0)=0$,  we conclude, using the definition of $\gq_\ka^*$,  that $\gq_0(\kappa)\geq\gq_\ka^*$. By Lemma~\ref{lem:ineg:q0},  we have $\gq_0\leq\gq_\ka^*$, hence, $\gq_0=\gq_\ka^*$.

Finally, we show that $\boE_\ka$ is constant in $[\gq_\ka^*,\infty)$.
Indeed, for $\gq\geq\gq_\ka^*$. using that $\boE_{\ka}$ is nondecreasing,
and \eqref{end:proof:main},
we have $\boE_\ka(\gq)\geq \boE_\ka(\gq_0(\kappa))\geq E_\ka(u_{0,\ka})$.
Since the reverse inequality was already proved in Lemma~\ref{lem:ineg:q0},  
we conclude that the equality holds.

At last, in view of Lemma~\ref{lem:allure}, and Steps 5 and  6, 
we conclude that $\boE_\kappa$ is strictly increasing on $(0,\gq_\ka^*)$, which  finishes the proof of the theorem.
\end{proof}



\section{Local well-posedness of \eqref{QGP} and orbital stability}
\label{section:cauchy}

The Cauchy problem with vanishing conditions at infinity associated with \eqref{quasilin0} is locally well-posed in Sobolev spaces $H^s(\R)$ of high regularity. This was shown for \eqref{quasilin0} in \cite{colinLWP,Colin2}, for $s\geq 3$.
Best regularity results were obtained in \cite{tataruIII} in $l^1H^s(\R)$ spaces, containing $H^s(\R)$, with $s>1/2+2$, and recently improved in \cite{ifrim-tataru,shao-zhou}.

Concerning \eqref{QGP}, we can decompose $\Psi$ as  $\Psi=v+\varphi$,  where $v$ satisfies condition \eqref{nonzero0} and $\varphi\in H^s(\R)$. Then, we obtain a quasilinear equation on $\varphi$ with vanishing conditions. In this setting, it is not straightforward to show well-posedness in $v+H^s(\R)$, 
using the latter results. This is due to the nonhomogenous nature of the equation on $\varphi$. However, more general quasilinear models have been considered in \cite{Kenig}. Their results provide local well-posedness in $v+H^s(\R)\cap H_w(\R)$, for any $s\geq1/2+11$, with some additional smoothness and decay assumptions on $v$, where $H_w(\R)$ is some weighted  Sobolev space.

For our purposes,  we will use the approach developed in \cite{BenzoniLWP}, where they proved that the Euler--Korteweg system is locally well-posed. As shown now, \eqref{QGP} can be written as an Euler--Korteweg system, with initial conditions given by $H^s(\R)$-perturbations of a dark soliton for $s\geq1$.

Assume that $\ka\leq0$ and $c>0$.
Let $\Psi_0=u_{c,\ka}+\phi_0$ satisfying $\inf_{x\in\R}|\Psi_0|>0$, where $\phi_0\in{H}^s(\R)$ with $s\geq1$ and $u_{c,\ka}$ is the dark soliton given by Theorem~\ref{thm:classiftwregu}.
Since $\Psi_0\in\boN\boX(\R)$, then for any local solution $\Psi\in\boC([0,T];u_{c,\ka}+{H}^s(\R))$, we deduce that 
there exists $T>0$ such that 
$\Psi(t)$ also belongs to $\boN\boX(\R)$, for all $t\in[0,T]$, so that we can perform the Madelung transform \cite{CarlesDanSaut}. Namely, writing
$\Psi=\sqrt{\rho}e^{i\theta}$, with  $\rho(t)\in L^\infty(\R)\cap\dot{H}^s(\R)$ and $\theta(t)\in\boC(\R)\cap\dot{H}^s(\R)$ for all $t\in[0,T]$.
We conclude that  $(\rho,\theta)$ satisfies in $[0,T]$ the system
\begin{equation}\label{eq:madelungeq}
    \left\{\begin{aligned}
       & \partial_t\rho-2\partial_x(\partial_x\theta\rho)=0,\\
       & -\partial_t\theta+(\partial_x\theta)^2=\frac{1-2\ka\rho}{2\rho}\partial_{xx}\rho-\frac{(\partial_x\rho)^2}{4\rho^2}+1-\rho.
    \end{aligned}\right.
\end{equation}
Therefore, setting 
\begin{equation}
\label{def:change:rho}
\Tilde{\rho}(x,t)=\rho(x,-t/2), \ \ \Tilde{v}(x,t)=\partial_x(\theta(x,-t/2)),\     \
K(y)=(1-2\ka y)/4y,\ \ g_0(y)=(y-1)/2,
\end{equation}
for all $y>0$, we deduce that 
$(\tilde{\rho},\tilde{v})$ satisfies the  Euler--Korteweg system in $[0,T]$:

\begin{equation}\label{eq:EulK0}
\left\{
    \begin{aligned}
        &\partial_t\Tilde{\rho}+\partial_x(\Tilde{v}\Tilde{\rho})=0,\\
&\partial_t\Tilde{v}+\Tilde{v}\partial_x\Tilde{v}=\partial_x\Big( K(\Tilde{\rho})\partial_{xx}\Tilde{\rho}+\frac{K'(\Tilde{\rho})}{2}(\partial_x\tilde{\rho})^2-g_0(\Tilde{\rho})\Big).
        \end{aligned}
\right.
\end{equation}


The result established in \cite{BenzoniLWP} concerning the general  Euler--Korteweg system \eqref{eq:EulK0} is as follows.
\begin{theorem}[Theorem~5.1 in \cite{BenzoniLWP}]\label{thm:lwpEK}
    Take $s>2+1/2$. Let $\rho_0\in L^\infty(\R)\cap \dot H^s(\R)$ be a function taking values in a compact subset of $(J_-,J_+)\subset (0,\infty)$. 
    Assume that $K,g_0\in \boC^\infty(\R)$, with $K>0$ in $(J_-,J_+).$ 
If $v_0\in H^{s-1}(\R)$, then there exists $T>0$ such that \eqref{eq:EulK0} has a unique solution $(\rho,v)$ on $\R\times[0,T]$, satisfying $\rho(\cdot,0)=\rho_0$, $v(\cdot,0)=v_0$, and
    \begin{equation}\label{lwpEK}
        \begin{cases}
            (\rho-\rho_0)\in\boC([0,T];H^s(\R)),\\
        (\partial_x\rho,v)\in\boC([0,T];H^{s-1}(\R)),\\
        \rho(\R\times[0,T])\subset\subset(J_-,J_+).
        \end{cases}
    \end{equation}
Also, the flow map is continuous in a neighborhood of $(\rho_0,u_0)$ in $(\rho_0+H^s(\R))\times H^{s-1}(\R)$.
\end{theorem}
For all $\ka\leq0$, it is clear that $K$ given in \eqref{def:change:rho} satisfies 
$K(y)>0,$ for all $y>0$, so that we can apply Theorem~\ref{thm:lwpEK} with $(J_-,J_+)=(0,\infty).$ 
In the next result, we obtain the local well-posedness of \eqref{QGP} for small $H^s(\R)$-perturbations of $u_{c,\ka}$ using Theorem~\ref{thm:lwpEK} and defining $\Psi=\rho e^{i\theta}\in u_{c,\ka}+H^s(\R)$ such that \eqref{eq:madelungeq}--\eqref{def:change:rho} hold. Also, we deduce the continuity of the flow map and the conservation of energy and momentum.
\begin{corollary}[Well-posedness of \eqref{QGP}]\label{coro:lwp}
    Let $s>1/2+2$, $\ka<0$ and $c>0$. If $\Psi_0\in u_{c,\ka}+H^s(\R)$ belongs to
    $\boN\boX(\R)$,
    then there exists $T_{\Psi_0}>0$, the maximal time of existence such that for every $T\in (0,T_{\Psi_0})$, \eqref{QGP} has a unique solution $\Psi$ on  $\R\times[0,T]$, satisfying $\Psi(\cdot,0)=\Psi_0$ and 
    \begin{equation}\label{lwp}
        \begin{cases}
            \Psi\in\boC([0,T];u_{c,\ka}+H^s(\R)),\\
            \inf_{\R\times[0,T]}|\Psi|>0.
        \end{cases}
    \end{equation}
    Moreover, the flow map is continuous in a neighborhood of $\Psi_0$ in $u_{c,\ka}+H^s(\R)$, and the energy and momentum are conserved, i.e.\ $E_\ka(\Psi(\cdot, t))=E_\ka(\Psi_0)$ and $p(\Psi(\cdot, t))=p(\Psi_0)$,
    for all $t\in[0,T].$
    \end{corollary}
\begin{proof}
\emph{Existence.}
   We start by performing the Madelung transform to the initial condition: $\Psi_0=\sqrt{\rho_0}e^{i\theta_0}$. Then, setting $v_0=\partial_x\theta_0$,
   we can apply Theorem~\ref{thm:lwpEK} to obtain $(\Tilde{\rho},\Tilde{v})$ a solution to \eqref{eq:EulK0} satisfying \eqref{lwpEK}, with the functions 
   $K$ and $g_0$ defined in \eqref{def:change:rho}.
      Setting $\rho(x,t)=\Tilde{\rho}(x,-2t)$ and $v=v(x,-2t)$, it remains to define $\theta\in\boC(\R\times[0,T])$ so that $\Psi=\sqrt{\rho}e^{i\theta}$ satisfies \eqref{QGP} and \eqref{lwp}. 
   Indeed, taking $$\theta(\cdot,t)=\theta_0+\int_0^tv^2- \frac{1-2\ka\rho}{2\rho}\partial_{xx}\rho-\frac{(\partial_x\rho)^2}{4\rho^2}+\rho-1,$$
yields $\theta\in\boC^1([0,T];L^2(\R))$, thus $\theta$ satisfies the second equation in \eqref{eq:EulK0}.    
  Since $ \inf_\R\rho>0$, we deduce that $\sqrt{\rho} e^{i\theta}$ is a nonvanishing local solution to \eqref{QGP}. We verify that $\Psi$ is still a $H^s(\R)-$perturbation of $u_{c,\ka}$ for $t>0$ to ensure \eqref{lwp}.  Since $\Psi_0-u_{c,\ka}\in H^s(\R)$, we can conclude if we prove that $\Psi-\Psi_0\in \boC([0,T];H^s(\R))$. It is easy to check that $\partial_x(\Psi-\Psi_0)$ belongs in $\boC([0,T];H^{s-1}(\R))$ using \eqref{lwpEK}, thus, we just need to verify that $\Psi-\Psi_0\in\boC([0,T]; L^2(\R)).$
    Since $(\rho,\theta)$ satisfies \eqref{eq:madelungeq}, we deduce using \eqref{lwpEK} that $\partial_t\theta\in\boC([0,T];H^{s-2}(\R))$, in particular $\theta-\theta_0\in \boC^1([0,T];L^2(\R))$ and we see that $\Psi-\Psi_0\in\boC([0,T]; L^2(\R))$ follows from this observation and \eqref{lwpEK}.

    \emph{Uniqueness.} Assume that $\check{\rho}e^{i\check{\theta}}$ is another solution to \eqref{QGP} defined on $[0,\check{T}]$ with the same initial condition, then $\Tilde{T}=\min\{\check{T},T\}$ and relabeling $\Tilde{T}$ as $T$, we have $\check{\rho}=\rho$ and $\partial_x\theta=\partial_x\check{\theta}$ in $[0,T]$ by Theorem~\ref{thm:lwpEK}. Since $\check{\rho}e^{i\check{\theta}}$ satisfies \eqref{QGP}, we infer that $\check{\theta}$ satisfies the second line in \eqref{eq:EulK0}. Therefore, we deduce that $\partial_t\check{\theta}=\partial_t\theta$ in $\boC([0,T];L^2(\R))$, so that, integrating the latter identity, we obtain $\check{\theta}=\theta$. We conclude using Theorem~\ref{thm:lwpEK} that there exists $T$ such that the flow map is well-defined in the vicinity of any initial condition $\Psi_0\in u_{c,\ka}+H^s(\R)$ satisfying $\inf_\R|\Psi_0|>0$. Moreover, this map takes its value in $\boC([0,T];u_{c,\ka}+H^s(\R))$.
    
    \emph{Continuity with respect to the initial data.} Let  $\Psi_0\in u_{c,\ka}+H^s(\R)$ such that $\inf_\R|\Psi_0|>0$, then the flow map is well-defined in a neighborhood $\boV$ of $\Psi_0$ and valued in  $\boC([0,T],u_{c,\ka}+H^s(\R))$, for some $T>0$. By simplicity, we only show that this map is continuous at $\Psi_0$, the continuity in $\boV$  is similar. Let $(\Psi_0^{(n)})\subset\boV$ satisfying $||\Psi_0^{(n)}-\Psi_0||_{H^s(\R)}\to0$ as $n\to\infty$. Let $\Psi^{(n)}=\rho^{(n)}e^{i\theta^{(n)}}$ and $\Psi=\rho e^{i\theta}$ be the unique solution of \eqref{QGP} in $\R\times[0,T]$ satisfying \eqref{lwp} and $\Psi^{(n)}(\cdot,0)=\Psi_0^{(n)}$ and $\Psi(\cdot,0)=\Psi_0$. By the continuity of the flow map in Theorem~\ref{thm:lwpEK} we get 
    \begin{equation}\label{eq:contdatum}
        \lim\limits_{n\to\infty}||(\rho^{(n)},\partial_x\theta^{(n)})-(\rho,\partial_x\theta)||_{\boC([0,T];H^s(\R)\times H^{s-1}(\R))}=0.
    \end{equation}
    Since $\theta^{(n)}$ and $\theta$ satisfies the second equation in \eqref{eq:madelungeq}, we deduce using \eqref{eq:contdatum} that $||\partial_t\theta^{(n)}-\partial_t\theta||_{\boC([0,T];L^2(\R))}\to0$, as $n\to\infty$, thus $||\theta^{(n)}-\theta||_{\boC([0,T];L^2(\R))}\to0$, as $n\to\infty$. We see that this relation and \eqref{eq:contdatum} are enough to ensure that $||\rho^{(n)} e^{i\theta^{(n)}}-\rho e^{i\theta}||_{\boC([0,T];H^s(\R))}\to0$ as $n\to\infty$, hence the flow map is continuous in $\Psi_0$.

    \emph{Conservation laws.} Assume that $\Psi_0$ satisfies $\inf_\R|\Psi_0|>0$ and is regular enough so that  the solution $\Psi$ of \eqref{QGP} belongs in $\boC([0,T];u_{c,\ka}+H^s(\R))$, for some $s>4+1/2$. Then the energy and momentum of $\Psi$ are conserved in time (one just needs to derive with respect to time in the integrals).
    The conclusion follows for $s>2+1/2$, by 
    using a density argument and the continuity of the flow map.

    As a conclusion, the following the maximal time of existence is well-defined:
    \begin{equation}\label{def:Tpsi}
        T_{\Psi_0}=\sup\left\{T>0:\quad \begin{aligned}
          &\text{There exists a unique  solution to \eqref{QGP} $\Psi$ defined on $[0,T]$}\\ &\text{satisfying \eqref{lwp} and $\Psi(\cdot,0)=\Psi_0$.}  
        \end{aligned}\right\},
    \end{equation}
     associated with  an initial condition $\Psi_0\in\boN\boX(\R)\cap u_{c,\ka}+H^s(\R)$. 
    \end{proof}
Finally, we can prove the stability of dark solitons, as stated in  
Theorem~\ref{thm:stability}, by invoking Theorem~\ref{thm:compacite}, and the Cazenave--Lions argument \cite{cazlions}.
\begin{proof}[Proof of Theorem~\ref{thm:stability}] Using that $\gc$ is a bijection between $(0,\gq_\ka^*)$ and $(c_\ka^*,\sqrt{2 })$, we just need to show the result for $u_{\gc(\gq),\ka}$, parametrized by $\gq\in(0,\gq_\ka^*)$.
 By contradiction,  we 
 suppose that for some $\gq\in(0,\gq_\ka^*),$ the dark soliton $u_{\gc(\gq),\ka}$ is not orbitally stable. Then there exist $\ve_0>0$, and a sequence $(v^{(n)}_0)\subset H^s(\R)$,  $s>5/2$, 
 such that 
the solution $\Psi^{(n)}(x,t)=u_{\gc(\gq),\ka}(x)+v^{(n)}(x,t)$ to \eqref{QGP}, with initial data  $\Psi^{(n)}_0=u_{\gc(\gq),\ka}+v^{(n)}_0$,  defined for $t\in[0, T_{\Psi_0^{(n)}} )$,  satisfies
\begin{align}\label{eq:nonstab1}
    &d(\Psi^{(n)}_0,u_{\gc(\gq),\ka})<1/n,\\
\label{eq:nonstab2}
    \text{and }\quad\inf_{(y,\theta)\in\R^2}&d(\Psi^{(n)}(t_n),u_{\gc(\gq),\ka}(\cdot-y)e^{i\theta})>\ve_0,
\end{align}
 for some $t_n\in (0,T_{\Psi_0^{(n)}}).$
Let us recall that Lemma~5.5 in \cite{delaire-mennuni} establishes that 
	if $v_{n}, v\in \mathcal{X}(\mathbb{R})$  satisfy
	$	d(v_n,v)\to 0$, 
	then,
	\begin{equation}
		\label{convstablemmeenergie2}
		\Vert |v_n|-|v| \Vert_{L^{\infty}(\R)} \to 0
		\quad \text{and}	\quad \Vert |v_n|^2-|v|^2 \Vert_{L^{2}(\R)} \to 0.
	\end{equation}
	In particular, this implies the continuity of the energy $E_\kappa(v_n)\to E_\kappa(v)$ (with respect to $d$). In addition, if $v_n,v\in \boN\boX(\R)$, then we also have the continuity of the momentum, i.e.\ that $p(v_n)\to p(v)$.	
 By conservation of energy, we have \begin{align}\label{eq:consE}
     E_\ka(\Psi^{(n)}(t))=E_\ka(\Psi_0^{(n)}),\quad \text{ for all }0\leq t\leq t_n.
 \end{align}
From \eqref{eq:consE}, we infer that, up to a subsequence that we do not relabel, we have $\Psi^{(n)}(t)\in\boN\boX(\R)$ for all $0\leq t \leq t_n$. Indeed, in view of \eqref{eq:nonstab1} and the continuity of the energy with respect to $d$, we obtain,
\begin{align}\label{eq:limemin}
    E_\ka(\Psi^{(n)}_0)<E_\ka(u_{\gc(\gq),\kappa})+\delta_n=\boE_\ka(\gq)+\delta_n,
\end{align}
for some positive sequence $\delta_n\to0$. Thus, for $n$ large enough, using \eqref{eq:consE}--\eqref{eq:limemin} and that $\boE_\ka$ is strictly increasing in $(0,\gq_\ka^*)$ (see Theorem~\ref{thm:compacite}), we deduce that there is $\tilde \gq \in 
(\gq,\gq_\ka^*)$ such that 
\begin{equation}
  E_\ka(\Psi^{(n)}(t))<\boE(\tilde \gq),\quad\text{ for all } 0\leq t\leq t_n.
\end{equation} By the definition of $\gq_\ka^*$ in \eqref{q_*}, we deduce that $\Psi^{(n)}(t)\in\boN\boX(\R)$, for all $0\leq t\leq t_n$. We can now invoke the conservation of momentum and the continuity of the momentum (with respect to $d$), to conclude that $p(\Psi^{(n)}(t_n))\to\gq$. Again, using \eqref{eq:consE} and the continuity of the energy, we  obtain $E_\ka(\Psi^{(n)}(t_n))\to\boE_\ka(\gq)$.
From Theorem~\ref{thm:compacite}, we deduce that there exist $\theta\in\R$ and a sequence of points $(x_n)$ such that 
 \begin{align}\label{eq:convstab}
     &1-|\Psi^{(n)}(\cdot+x_n,t_n)|^2\to 1-|u_{\gc(\gq),\ka}|^2,\quad\text{ in }L^2(\R),\\
     \label{eq:convstab2}
     &(\Psi^{(n)})'(\cdot+x_n,t_n)\to e^{i\theta}u_{\gc(\gq),\ka}',\quad\text{ in } L^2(\R).
 \end{align}
Let us show that  \eqref{eq:convstab}--\eqref{eq:convstab2} imply  that 
\begin{equation}
\label{last:conv}
d(\Psi^{(n)}(t_n),e^{i\theta}u_{\gc(\gq),\ka}(\cdot-x_n))\to0.
\end{equation}
Indeed, it is immediate that  \eqref{eq:convstab2} leads to  $||(\Psi^{(n)})'(t_n)- e^{i\theta}u_{\gc(\gq),\ka}'(\cdot-x_n)||_{L^2(\R)}\to 0$. In addition, we have the estimate: 
\begin{align}\label{eq:strongernorm}
     \norm{|\Psi^{(n)}|-|u_{\gc(\gq),\ka}(\cdot-x_n)|}_{L^2(\R)}^2\leq\frac{\norm{|\Psi^{(n)}|^2-|u_{\gc(\gq),\ka}(\cdot-x_n)|^2}_{L^2(\R)}^2}{\inf_{\R}((|\Psi^{(n)}|+|u_{\gc(\gq),\ka}(\cdot-x_n)|)^2)}.
 \end{align}
 Since $0\leq1-\abs{u_{\gc(\gq),\ka}}^2\leq1-\gc(\gq)^2/2$ (see Proposition~\ref{prop:globaleta}), we deduce that $|u_{\gc(\gq),\ka}|\geq \gc(\gq)/\sqrt{2}$. Therefore, using \eqref{eq:strongernorm} and \eqref{eq:convstab}, we conclude that $\norm{|\Psi^{(n)}|-|u_{\gc(\gq),\ka}(\cdot-x_n)|}_{L^2(\R)}\to0$, which establishes \eqref{last:conv}. This contradicts \eqref{eq:nonstab2}. 
\end{proof}


\begin{merci}
	The authors acknowledge support from the Labex CEMPI (ANR-11-LABX-0007-01).
	E.~Le Quiniou was also supported by the R\'egion Hauts-de-France.
\end{merci}
\bibliographystyle{abbrv}

\end{document}